\documentclass[journal,twoside,web]{ieeecolor}



\usepackage{amsthm}
\usepackage{mathrsfs}
\usepackage{lipsum}
\usepackage{generic}

\usepackage{amsmath,amssymb,amsfonts}
\usepackage{graphicx}
\usepackage{algorithm,algorithmic}

\usepackage{textcomp}
\usepackage{xcolor}

\usepackage{setspace}
\usepackage{mathtools}
\usepackage{enumerate}

\usepackage{caption}
\usepackage{subcaption}
\usepackage{algorithm}
\usepackage{float}
\newtheorem{definition}{Definition}
\newtheorem{theorem}{Theorem}

\newtheorem{lemma}{Lemma}
\newtheorem{corollary}{Corollary}

\theoremstyle{definition}
\newtheorem{remark}{Remark}

\usepackage{hyperref}
\makeatletter
\let\NAT@parse\undefined
\makeatother
\usepackage{cite}

\newtheorem{innerexample}{Example}[section]

\newcommand{\R}{\mathbb{R}}

\newcommand{\sgn}{\mathrm{sgn}}
\newenvironment{example}[1][]%
{\begin{innerexample}[#1]\pushQED{\qed}}%
{\popQED\end{innerexample}}

\renewenvironment{proof}[1][Proof]{\par\pushQED{\qed}\normalfont\textit{#1.}~}{\popQED\par}

\newif\ifrevisions
\revisionsfalse
\newcommand{\rev}[1]{\ifrevisions\textcolor{blue}{#1}\else#1\fi} % inline text/math (this round)
\newcommand{\revh}[1]{\ifrevisions\textcolor{red}{#1}\else#1\fi}  % conference-to-journal round
\newenvironment{revblock}{\ifrevisions\color{blue}\fi}{}          % wrap whole paragraphs/equations
\title{Global Exponential and Inverse Optimal Stabilization of a 3D Nonholonomic Vehicle\\ in Spherical Coordinates
}

\author{Kwang Hak Kim, Velimir Todorovski, Mamadou Diagne, and Miroslav Krsti\'c%
\thanks{This work was supported by the Office of Naval Research under grants N00014-23-1-2376 and N00014-23-1-2831. The results and opinions in this paper are solely of the authors and do not reflect the position or the policy of the U.S. Government.}
\thanks{K.H. Kim, V. Todorovski, M. Diagne and M. Krsti\'c are with the Department of Mechanical and Aerospace Engineering, UC San Diego, 9500 Gilman Drive, La Jolla, CA, 92093-0411, {\tt\small \{kwk001,vtodorovski,mdiagne,krstic\}@ucsd.edu}}%
}

\begin{document}
\setlength\abovedisplayskip{2pt}
        \setlength\belowdisplayskip{2pt}
        \setlength\abovedisplayshortskip{2pt}
        \setlength\belowdisplayshortskip{2pt}
        \allowdisplaybreaks
        \setlength{\parindent}{1em}
        \setlength{\parskip}{-0em}  
        % \addtolength{\oddsidemargin}{2 pt} 
\maketitle
% \thispagestyle{empty}
% \pagestyle{empty}

%%%%%%%%%%%%%%%%%%%%%%%%%%%%%%%%%%%%%%%%%%%%%%%%%%%%%%%%%%%%%%%%%%%%%%%%%%%%%%%%

\begin{abstract}
For the 3D nonholonomic vehicle actuated in surge velocity, pitch rate, and yaw rate, we construct continuous time-invariant feedback laws that stabilize both the roll-free and roll-augmented kinematic models to a target position and orientation in spherical coordinates. The design uses a spherical coordinate transformation with a singular set of codimension two and measure zero. We first present a backstepping design that yields explicit unidirectional, cusp-free feedback laws and strict barrier control Lyapunov functions (CLFs). These establish global exponential stability in spherical coordinates with user-assigned decay rates, while ensuring bounded Cartesian trajectories and exponential convergence to the target, consistent with Brockett’s obstruction. The barrier terms guarantee avoidance of the coordinate singularities away from the target. The strict CLFs are then used to develop families of asymptotically stabilizing inverse optimal feedback laws for both systems that minimize a meaningful infinite-horizon cost. These laws inherit infinite gain margins from driftlessness, respect user-defined input bounds through Legendre--Fenchel transformations, and are implementable with either bidirectional or unidirectional surge actuation. Two penalty choices are illustrated: a quadratic input penalty whose associated state running cost approaches a prescribed quadratic form, and an input-constraining penalty that enforces user-defined bounds. This paper lays the foundation for constructive stabilization of the 3D nonholonomic vehicle.
\end{abstract}

%%%%%%%%%%%%%%%%%%%%%%%%%%%%%%%%%%%%%%%%%%%%%%%%%%%%%%%%%%%%%%%%%%%%%%%%%%%%%%%%

\section{Introduction}
 
Many spatial (3D) vehicles, from underwater platforms such as AUVs~\cite{lapierre2003nonlinear,caccia2000guidance} to aerial systems such as fixed-wing aircraft~\cite{vamvoudakis2022nonequilibrium,lugo2014dubins,chitsaz2007time}, are effectively nonholonomic. For such systems, lacking heave and sway actuation, the geometric constraints make the model attractive for analysis and controller design. However, the minimal kinematic model~\cite{fossen1994guidance} also gives rise to a fundamental obstruction: in the usual Cartesian coordinates, continuous, and even discontinuous~\cite{ryan1994brockett,coron1994relation}, time-invariant feedback laws cannot asymptotically stabilize the system~\cite{brockett1983asymptotic}. In this paper, we consider the spherical representation of the 3D nonholonomic kinematic model to circumvent the obstruction and develop a continuous time-invariant feedback design for the 3D nonholonomic vehicle directly in spherical coordinates, establish global exponential stability (GES) through strict barrier control Lyapunov functions (CLFs), extend the design to the roll-augmented kinematics, and develop general inverse optimal redesigns.

\subsection{Related works}

\paragraph{Navigating Brockett's obstruction} To navigate Brockett's obstruction in the canonical unicycle system, several strategies are utilized, each with inherent trade-offs. While time-varying laws~\cite{samson1990velocity,coron1992global_controllable} provide stabilization, they are often delay-sensitive and oscillatory. Discontinuous controllers~\cite{astolfi1996discontinuous,bloch1996stabilization_slidingmode} frequently encounter chattering, and hybrid strategies~\cite{hespanha1999_hybrid_stabilization,prieur2003robust} struggle with zig-zagging trajectories. Polar coordinate transformation, by introducing a singularity at the origin, circumvent these issues~\cite{aicardi1995,restrepo_3d_2019,wang24_force_controlled_safestable, todorovski2026global} through the singularity, allowing for the design of continuous time-invariant feedback and strict CLFs~\cite{Part1_todorovskiCLF2025,Part2_kimIOC2025}. 
% Beyond the theoretical advantages such as inverse optimality~\cite{Part2_kimIOC2025}, this framework suits egocentric sensing: the spherical states depend only on the range and bearing to the target, obtainable from LiDAR, radar, or a camera, together with the vehicle attitude and the target heading.

\paragraph{Stabilization of 3D nonholonomic vehicles} The spatial nonholonomic model actuated in forward surge velocity, pitch rate, and yaw rate~\cite{fossen1994guidance} is inherent to the 3D motion of many aerial systems~\cite{vamvoudakis2022nonequilibrium,lugo2014dubins,chitsaz2007time} and underwater platforms~\cite{lapierre2003nonlinear,caccia2000guidance}. While this problem has been approached extensively via path-following and reference tracking formulations~\cite{alonge2001trajectory,aguiar2007trajectory,fossen2023alos,fossen2024alos3d}, the stabilization of nonholonomic vehicles is well known to be a fundamentally more challenging task~\cite{jiang2010controlling,deluca1998feedback}. This discrepancy stems from the fact that tracking is based on a moving reference, for which the system can be made to follow under suitable persistence of excitation conditions. As in the planar case, the stabilization of 3D nonholonomic vehicles primarily relies on time-varying~\cite{li2018design,pettersen1999time,do2002global}, discontinuous~\cite{egeland1994exponential,egeland1996feedback}, or logic-based hybrid strategies~\cite{aguiar2002global}. These approaches suffer from the same problems inherent to the unicycle case (e.g., oscillatory transients, chattering, zig-zagging). \rev{Notably, \cite{Heetal2022exp3D} proposes a discontinuous time-invariant feedback with an almost-global exponential claim; however, its gain condition involves a state-dependent bound which, as the authors acknowledge, is inherently difficult to verify, and no explicit gain selection is available.}
 
\paragraph{Polar-like coordinate designs} Beyond circumventing Brockett's obstruction through a coordinate singularity, spherical coordinates suit egocentric sensing: the spherical states depend only on the range and bearing to the target, obtainable from LiDAR, radar, or a camera, together with the vehicle attitude and the target heading. \rev{Lyapunov-based designs in polar-like coordinates exist for the 3D vehicle: \cite{aicardi2001cusp} obtains, from a Lyapunov-like analysis, an almost-globally convergent time-invariant design with ``cusp''-free paths, and \cite{restrepo_3d_2019} develop feedback with cascade arguments and negative-semidefinite Lyapunov derivatives. However, neither yields a strict CLF with arbitrary gains and assignable decay rates, which prevents systematic extensions such as inverse optimal control~\cite{Part2_kimIOC2025}.}
 
\paragraph{Geometric coordinate-free designs} An intrinsic, coordinate-free treatment of the vehicle direction on the unit sphere $S^2$, or, with roll, of the vehicle attitude on $SO(3)$, the set of all rigid rotations, would avoid the parameterization singularities inherent to Euler angles~\cite{bullo2005geometric,chaturvedi2011rigid}. Such formulations, however, face a topological obstruction. The resulting state spaces are vector bundles over compact manifolds, on which no continuous time-invariant feedback is globally stabilizing, even absent the nonholonomic constraint~\cite{bhat2000topological}, so continuous geometric designs are at best almost global~\cite{chaturvedi2011rigid}. We therefore adopt an Euler-angle description with unbounded angles, which is free of this obstruction and whose sole singularity admits an exact, physically meaningful regularization. The resulting design excludes only a codimension-two measure-zero set while also providing strict CLFs with assignable decay rates. 
 
\paragraph{Inverse optimality} Solving and storing solutions of the Hamilton--Jacobi--Bellman (HJB) equation becomes intractable as the state dimension grows, a curse of dimensionality that led Kalman~\cite{kalman1964inverse_opt} to pose the inverse problem of optimal control for linear systems. Following its nonlinear formulation in~\cite{moylan1973nonlinear}, the framework was broadened along several lines: to stochastic dynamics~\cite{deng1997stochastic}, to differential games under disturbances with state-dependent bounds~\cite{freeman1996robust}, and, through input-to-state (ISS) stabilizing controllers, to disturbances of arbitrary magnitude~\cite{krstic1998inverse}. A constructive treatment of nonlinear inverse optimality first appeared in \cite{sepulchre1997constructive}, built around feedback laws of the form $-(L_gV)^{\rm T} = - g^{\rm T}\nabla V$, with $g$ the input vector field of a system affine in control. This methodology has proven effective in applications ranging from rigid spacecraft attitude control \cite{krstic_rigidspacecraft_1999,bharadwaj1998geometry} to inverse optimal safety \cite{krstic2023inverse,krstic2024offender}; for the stabilization of nonholonomic systems beyond the unicycle~\cite{Part2_kimIOC2025}, however, it remains largely undeveloped. A notable exception is \cite{do2015global}, which treats stochastic nonholonomic stabilization in the inverse optimal setting, though it establishes optimality only at the subsystem level and, owing to properties intrinsic to stochastic systems, without strict CLFs.
 
\subsection{Contributions and Organization}
 
\begin{revblock}
Building on the polar-coordinate framework and strict CLFs of~\cite{Part1_todorovskiCLF2025,Part2_kimIOC2025}, this paper makes the following contributions.

\emph{Exponential stabilization and strict CLFs.} We construct continuous time-invariant feedback laws that achieve global exponential stabilization of the 3D nonholonomic vehicle in spherical coordinates, with a user-assigned decay rate. The central construction is a strict barrier CLF obtained through a coupled yaw--pitch backstepping design. It addresses a specifically spatial difficulty: one surge input drives both horizontal and vertical motion, while the spherical coordinate singularity includes positions away from the target. The CLF certifies convergence and prevents encounter with this non-target singularity. The construction covers both roll-free and roll-augmented kinematics and uses strictly positive surge. In Cartesian coordinates, it guarantees boundedness and exponential attractivity, consistently with Brockett's obstruction.

\emph{Inverse-optimal control with actuation constraints.} These strict CLFs enable constructive inverse-optimal feedback designs for the spatial vehicle, rather than serving only as certificates for the exponential controllers. The resulting controller families asymptotically stabilize both models and minimize infinite-horizon costs involving state and control penalties. They have infinite gain margins and permit user-defined input bounds through the choice of penalties. The families include both bidirectional surge and unidirectional designs for vehicles restricted to forward propulsion. The latter preserve stabilization and optimality despite removing the option of reversal, i.e., whenever forward translation cannot decrease the CLF, the attitude controls retain the ability to do so.
Compared to the conference versions of this paper~\cite{kim2026spherical3d,kim2026IOC_3D_ACC}, this work extends the results to the roll-augmented system, presents unidirectional surge inverse optimal designs, and establishes a general inverse optimal redesign result for arbitrary driftless control affine systems.
\end{revblock}

The paper is organized as follows. Section~\ref{sec:model} presents the spherical model and its roll-augmented extension, Section~\ref{sec:backstepping} the backstepping design and the GES results, Section~\ref{sec:ioc_general} the general inverse optimal lemma and its stability margins, and Sections~\ref{sec:ioc_rollfree} and~\ref{sec:ioc_roll_aug} the inverse optimal designs for the roll-free and roll-augmented vehicles, each with bidirectional and unidirectional surge.

\section{3D Model in Spherical Coordinates}\label{sec:model}

We first present the spherical model of the 3D vehicle on the largest well-defined set, whose very coordinate singularity permits circumventing Brockett's condition, together with compatible stability and strict CLF definitions. The absence of roll proves a matter of convenience: a static rotation of the rate inputs extends any stabilizing feedback to the full kinematics.

%%%%%%%%%%%%%%%%%%%%%%%%%%%%%%%%%%%%%%%%%%%%%%%%%%%%%%%%%%%%%%%%%%%%%%%%%%%%%%%%
\subsection{Roll-free kinematic model}\label{sec:model_roll_free}

Consider the 3D nonholonomic kinematic model in the North-East-Down (NED) frame~\cite{fossen1994guidance}:
\begin{subequations}\label{eq:cart_3d_sys}
\begin{align}
\dot x &= v \cos\psi \cos\theta\label{eq:x_dot}\\
\dot y &= v \sin\psi \cos\theta\label{eq:y_dot}\\
\dot z &= -v \sin\theta\\
\dot \theta &= q\\
\dot \psi &= \frac{r}{\cos\theta}\label{eq:psi_dot}\,
,
\end{align}
\end{subequations}
where $(x,y,z)\in\R^3$ is the position, $\psi \in \mathbb{R}$ is yaw, $\theta \in\mathbb{R}$ is pitch (positive nose-up), and $(v,q,r)$ are forward velocity (surge), pitch rate, and yaw rate inputs respectively. To avoid the singularity caused by the division of $\cos\theta$ in~\eqref{eq:psi_dot}  and simplify the design, we define
\begin{equation}
\label{eq:yawrate_actual}
r = \tilde{r} \cos\theta\,,
\end{equation}
where $\tilde r$ is an auxiliary yaw rate. This choice is physically meaningful because when $\cos\theta = 0$, i.e., when the vehicle is in a vertical  $(\theta = \pm \pi/2)$, changes in yaw do not affect the direction of motion, so steering in the yaw direction is inconsequential. This can be seen from~\eqref{eq:x_dot} and~\eqref{eq:y_dot}, where $\dot x=\dot y=0$ occurs whenever $\cos\theta=0$, regardless of $\psi$, implying that the yaw angle only parameterizes the horizontal velocity direction and becomes irrelevant when the horizontal component vanishes. 

\begin{figure}[t]
\centering
\includegraphics[width=.9\linewidth]{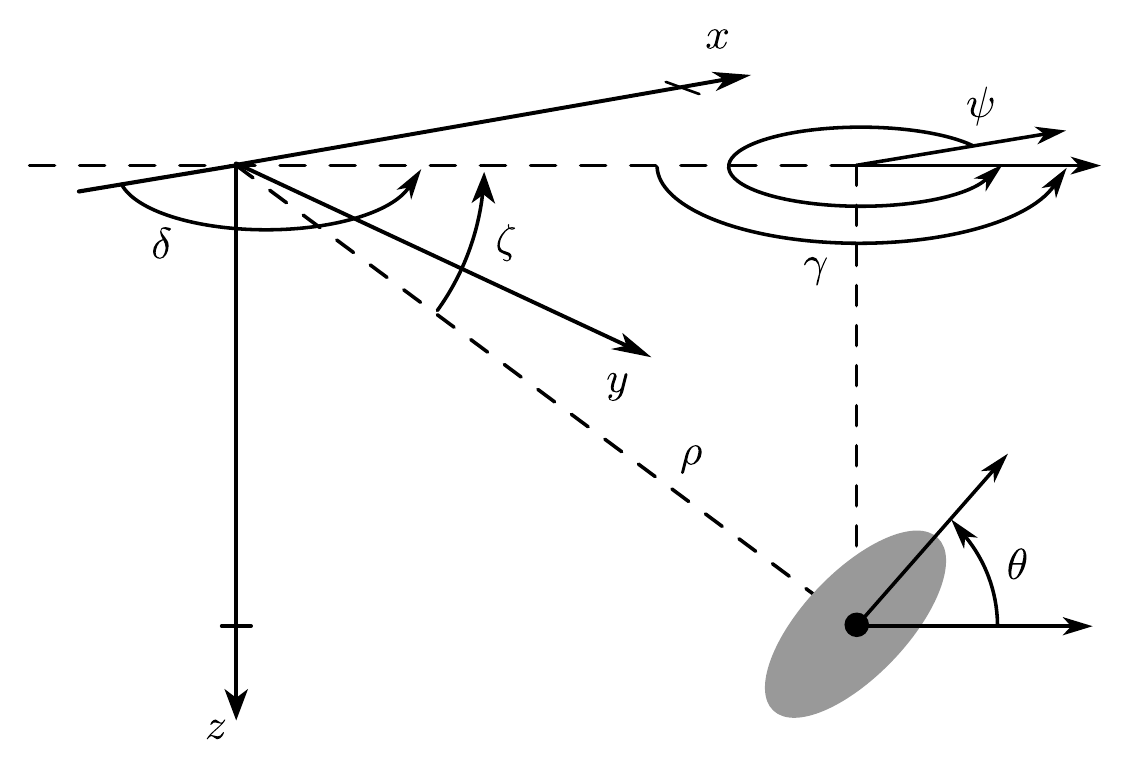}
\caption{Kinematic model of a 3D nonholonomic vehicle.}
\label{fig:3d_unicycle}
\vspace{-0.15cm}
\end{figure}

System~\eqref{eq:cart_3d_sys}, however, is not stabilizable with a time-invariant continuous control law due to Brockett's condition~\cite{brockett1983asymptotic}. We circumvent this obstruction by introducing a spherical coordinate state transformation defined in Table~\ref{tab:spherical_coords} and Figure~\ref{fig:3d_unicycle}. The corresponding spherical dynamics are given as

\vspace{-0.5cm}
\begin{subequations}\label{eq:spheric_3d_sys}
\begin{align}
\dot{\rho} &= v\bigl(\sin\theta\sin\zeta-\cos\theta\cos\zeta\cos\gamma\bigr)\label{eq:sph_rho}\\
\dot{\delta} &= \frac{v\cos\theta}{\rho\cos\zeta}\sin\gamma\label{eq:sph_delta}\\
\dot{\gamma} &= \frac{v\cos\theta}{\rho\cos\zeta}\sin\gamma - \tilde{r} \label{eq:sph_gamma}\\
\dot{\zeta} &= \frac{v}{\rho}\bigl(\cos\theta\sin\zeta\cos\gamma+\sin\theta\cos\zeta\bigr) \label{eq:sph_zeta}\\
\dot{\theta} &= q. \label{eq:sph_theta}
\end{align}
\end{subequations}

However, the spherical transformation in Table~\ref{tab:spherical_coords} is undefined on the $z$-axis ${\{x=y=0\}}$. Namely, when the vehicle lies on the line through the target orthogonal to the target plane. When the horizontal range $\sqrt{x^2+y^2} = \rho\cos\zeta$ vanishes, the term ${\rm atan2}(-y,-x)$ is undefined, leaving the angles $\delta$ and $\gamma$ undetermined. This is reflected in~\eqref{eq:sph_delta} and~\eqref{eq:sph_gamma} through the factor $1/(\rho\cos\zeta)$ resulting in a singularity, motivating the restriction to the natural domain
\begin{align}
\mathcal D &\coloneqq \left\{\rho > 0\right\}\times \mathcal{T}\label{eq:domainD}\\
\mathcal{T} &\coloneqq \left\{(\delta,\gamma,\zeta,\theta)\in\mathbb{R}^4 : |\zeta|<\frac{\pi}{2}\right\}\,,
\end{align}
on which $\cos\zeta>0$ and the coordinate map is well defined. The excluded set is a codimension-two subspace of $\mathbb{R}^5$ and therefore has measure zero and, moreover, its complement is connected.

Importantly, this is not a consequence of our coordinate choice. Any spherical (or cylindrical) transformation must exclude at least one ``polar'' set since no bearing angle can be defined when the horizontal projection vanishes. The singularity can be shifted by changing the reference plane but not eliminated. Hence, \eqref{eq:domainD} is the largest open set on which a continuous time-invariant design is possible in these coordinates, and this very singularity permits circumventing Brockett’s condition.

\begin{table}[t!]
\centering
\renewcommand\arraystretch{1.4}
\small
\begin{tabular}{|l|l|}
\hline
\textbf{Spherical Coord.} & \textbf{Description} \\
\hline
$\rho = \sqrt{x^2 + y^2 + z^2}$ & Distance to target \\ \hline
$\delta = {\rm atan2}(-y,-x)$ & Azimuth angle \\ \hline
$\gamma = {\rm atan2}(-y,-x) - \psi$ & $z$-axis line-of-sight (LOS) angle \\ \hline
$\revh{\zeta} = \arcsin\!\left(\dfrac{-z}{\rho}\right)$ & Elevation angle \\ \hline
$\theta$ & Pitch angle \\ \hline
\end{tabular}
\caption{Spherical transformation and interpretation. If the target is at ${(x^*, y^*, z^*, \psi^*) \neq 0}$, the transformation generalizes to ${\rho=\sqrt{(x-x^*)^2+(y-y^*)^2+(z-z^*)^2}}$, ${\delta= {\rm atan2}(y^*-y ,x^*-x ) -\psi^*}$, ${\gamma= \delta-\psi+\psi^*}$, and $\zeta = \arcsin\big(\frac{z^*-z}{\rho}\big)$. For $\theta^* \neq 0$ see Remarks~\ref{rem:target_pitch} and~\ref{rem:arbitrary_pose}.}
\label{tab:spherical_coords}
\end{table}

For that reason, we now introduce a definition of global exponential stability that is compatible with the domain and coordinate singularity of our transformation. To this end, define 
\begin{equation}
    \Theta \coloneqq (\rho,\delta,\gamma,\zeta,\theta)^\top\,, \label{eq:Theta}
\end{equation}
% $\Theta \coloneqq (\rho,\delta,\gamma,\zeta,\theta)^\top$
and the metric
\begin{align}\label{eq:metricD}
|\Theta|_{\mathcal D}
\coloneqq \rho+|\delta|+|\gamma|+|\tan\zeta|+|\theta|.
\end{align}
The choice of $|\tan\zeta|$ is natural since it becomes unbounded as $|\zeta|\to \pi/2$, capturing the coordinate singularity at the boundary of $\mathcal D$. 

Additionally, the angular states are continuously unwrapped, with $\delta(t_0) = {\rm atan2}(-y_0,-x_0)$, after which $\delta$ evolves by~\eqref{eq:sph_delta}, so the branch cut of ${\rm atan2}$ introduces no discontinuity in~\eqref{eq:spheric_3d_sys}, and $\gamma = \delta - \psi$ is defined consistently with this lifted representation.

We now state the definitions of global asymptotic/exponential stability and a strict CLF.
\begin{definition}[GAS on $\mathcal{D}$]
\label{def-our-GAS}
Consider the system \eqref{eq:spheric_3d_sys} with feedback laws $v,q$, and $\tilde r$
% $v(\rho,\delta,\gamma,\zeta,\theta),q(\rho,\delta,\gamma,\zeta,\theta)$ and $r(\rho,\delta,\gamma,\zeta,\theta)$ 
that are continuous on a state space $\mathcal{D}$ with respect to its metric. If there exists a class $\mathcal{KL}$ function $\beta$ such that, for all $t\geq t_0$, it holds that $|\Theta(t)|_{\mathcal{D}}\leq \beta\left(|\Theta(t_0)|_{\mathcal{D}},t-t_0\right)$, we say that the 
%equirlibrium 
point $\Theta= (0,0,0,0,0)$ is {\em globally asymptotically stable (GAS) on $\mathcal{D}$}. 
\end{definition}

\begin{definition}[GES on $\mathcal{D}$]
\label{def-our-GES}
Consider the system \eqref{eq:spheric_3d_sys}, with a feedback law $v,q$ and $\tilde r$ that is continuous on a state space $\mathcal{D}$ with respect to its metric. If there exist $K \geq 1$ and $\lambda > 0$ such that, for all $t\geq t_0$, it holds that $|\Theta(t)|_{\mathcal{D}}\leq K|\Theta(t_0)|_{\mathcal{D}}e^{-\lambda (t-t_0)}$, we say that the point $\Theta = (0,0,0,0,0)$ is {\em globally exponentially stable (GES) on $\mathcal{D}$}. 
\end{definition}

% \begin{definition}[Strict CLF for the 3D nonholonomic vehicle]
% \label{def-CLF}
% A continuously differentiable function $V(\Theta)$ is a strict \textit{control Lyapunov function} (CLF) for \eqref{eq:spheric_3d_sys} if (i) there exist class $\mathcal{K}$ functions $(\alpha_1, \alpha_2)$ such that, for all ${\Theta \in \mathcal{D}}$, ${\alpha_1(|\Theta|_{{\mathcal{D}}}) \le V(\Theta) \le \alpha_2(|\Theta|_{{\mathcal{D}}})}$, and (ii) there exists $\left(\dfrac{v}{\rho},q,\dfrac{r}{\cos\theta}\right)\in \mathbb{R}^3$ such that $\dot V(\Theta) < 0$ for all $\Theta \neq 0$ in~$\mathcal{D}$.
% \end{definition}

\begin{definition}[Strict CLF for the 3D nonholonomic vehicle]
\label{def-CLF}
A continuously differentiable function $\Theta\mapsto V$
is a strict CLF with respect to
\eqref{eq:spheric_3d_sys} if (i) there exist class $\mathcal{K}$
functions $(\bar\alpha_1,\bar\alpha_2)$ such that, for all $\Theta$ in
$\Sigma = \{\rho\geq 0\}\times \mathcal{T}$,
$\bar\alpha_1(|\Theta|_{\mathcal{D}}) \le V(\Theta) \le
\bar\alpha_2(|\Theta|_{\mathcal{D}})$, and (ii) for all $\Theta \neq (0,0,0,0,0)$
in $\Sigma$, there exists
$\left(\dfrac{v}{\rho},q,\dfrac{r}{\cos\theta}\right)\in \mathbb{R}^3$
such that $\dot V(\Theta) < 0$.
\end{definition}

\begin{remark}[Obstruction to arbitrary target pitch]\label{rem:target_pitch}
The transformations in Table~\ref{tab:spherical_coords} hold for arbitrary target positions and yaw angles $(x^*,y^*,z^*,\psi^*)$ because spatial translations and $z$-axis rotations are symmetries of both the kinematic model~\eqref{eq:cart_3d_sys} and the spherical transformation. No analogous symmetry exists for pitch. The orientation in~\eqref{eq:cart_3d_sys} is constrained to the sequence $R_z(\psi)R_y(\theta)$, where $R_z$ and $R_y$ denote elementary rotations about the $z$- and $y$-axes. Expressing this attitude relative to a frame pitched by $\theta^*$ yields $R_y(-\theta^*)R_z(\psi)R_y(\theta)$, which induces a relative roll angle that the model cannot represent. Consequently, the coordinate shift $\theta-\theta^*$ does not leave the spherical kinematics~\eqref{eq:spheric_3d_sys} invariant, requiring $\theta^*=0$ throughout. This restriction is inherent to the roll-free model. The roll-augmented kinematics in Section~\ref{sec:roll_ext} eliminate this constraint, as discussed in Remark~\ref{rem:arbitrary_pose}.
\end{remark}

\subsection{Extension to \rev{Roll-Augmented Kinematics}}\label{sec:roll_ext}

The model~\eqref{eq:cart_3d_sys} carries no roll state. We now show that this is a matter of convenience rather than a necessity. Any feedback that stabilizes the spherical system~\eqref{eq:spheric_3d_sys} can be extended to the full kinematics with roll, through a static rotation of the rate inputs and a simple roll feedback law, such that GAS and GES are preserved (the latter with rate $\min\{\lambda, k_\phi\}$, where $k_\phi$ is the roll gain), along with the strict CLF property.

Consider the roll-augmented kinematic model also in the North-East-Down (NED) frame~\cite{fossen1994guidance}
\begin{subequations}\label{eq:roll_model}
\begin{align}
\dot x &= v \cos\psi \cos\theta \label{eq:roll_x}\\
\dot y &= v \sin\psi \cos\theta \label{eq:roll_y}\\
\dot z &= -v \sin\theta \label{eq:roll_z}\\
\dot \phi &= p + \bigl(q\sin\phi + r\cos\phi\bigr)\tan\theta
\label{eq:roll_phi}\\
\dot \theta &= q\cos\phi - r\sin\phi \label{eq:roll_theta}\\
\dot \psi &= \frac{q\sin\phi + r\cos\phi}{\cos\theta}\,,
\label{eq:roll_psi}
\end{align}
\end{subequations}
where $\phi \in \mathbb{R}$ is the roll angle, $p$ is the roll rate
input, and \eqref{eq:roll_phi}--\eqref{eq:roll_psi} are the Euler
kinematics~\cite{fossen1994guidance}. At $\phi = 0$, \eqref{eq:roll_model} reduces to~\eqref{eq:cart_3d_sys}. Applying the
transformation of Table~\ref{tab:spherical_coords} while preserving the roll dynamics yields 
\begin{subequations}\label{eq:roll_sph}
\begin{align}
\dot{\rho} &= v\bigl(\sin\theta\sin\zeta-\cos\theta\cos\zeta\cos\gamma\bigr)\\
\dot{\delta} &= \frac{v\cos\theta}{\rho\cos\zeta}\sin\gamma\\
\dot{\gamma} &= \frac{v\cos\theta}{\rho\cos\zeta}\sin\gamma
- \frac{q\sin\phi + r\cos\phi}{\cos\theta} \\
\dot{\zeta} &= \frac{v}{\rho}\bigl(\cos\theta\sin\zeta\cos\gamma+\sin\theta\cos\zeta\bigr)\\
\dot{\phi} &= p + \bigl(q\sin\phi + r\cos\phi\bigr)\tan\theta\label{eq:roll_sph_phi}\\
\dot{\theta} &= q\cos\phi - r\sin\phi\,.
\end{align}
\end{subequations}
Accordingly, define the augmented state
\begin{align}
\Xi \coloneqq \bigl(\Theta^\top\!, \phi\bigr)^{\top}\,,
\end{align}
along with the domain and metric
\begin{align}
    \mathcal{D}_{\mathrm{a}} \coloneqq \mathcal{D}\times\{\phi \in \mathbb{R}\}, \quad
|\Xi|_{\mathcal{D}_{\mathrm{a}}} \coloneqq
|\Theta|_{\mathcal{D}} + |\phi|\,.
\label{eq:aug_metric}
\end{align}

Then, define the rotation
\begin{align}
M(\phi) &\coloneqq
\begin{bmatrix} \cos\phi & -\sin\phi\\ \sin\phi & \cos\phi \end{bmatrix}\,,
\label{eq:M_def}
\end{align}
and the \rev{rotated pitch/yaw rates}
\begin{align}
\begin{bmatrix}\bar q\\ \bar r\end{bmatrix}
\coloneqq M(\phi)\begin{bmatrix} q\\ r\end{bmatrix}
= \begin{bmatrix} q\cos\phi - r\sin\phi\\
q\sin\phi + r\cos\phi \end{bmatrix},
\label{eq:channel_rates}
\end{align}
so that~\eqref{eq:roll_theta} and~\eqref{eq:roll_psi} read
$\dot\theta = \bar q$, $\dot\psi = \bar r/\cos\theta$, and
\eqref{eq:roll_phi} reads $\dot\phi = p + \bar r\tan\theta$. 
% Since
% $M(\phi) \in SO(2)$ for every $\phi \in \mathbb{R}$, the channel map
% $(q,r) \mapsto (\bar q, \bar r)$ is, for each fixed $\phi$, a linear
% isometry with inverse $M(\phi)^\top$. 
Mirroring~\eqref{eq:yawrate_actual},
define
\begin{align}\label{eq:roll_sub}
\bar r = \tilde{\bar r}\cos\theta\,,
\end{align}
under which
\begin{align}\label{eq:roll_reg}
\dot\theta = \bar q\,, \qquad
\dot\psi = \tilde{\bar r}\,, \qquad
\dot\phi = p + \tilde{\bar r}\sin\theta\,,
\end{align}
i.e., the $\tan\theta$ term in \eqref{eq:roll_phi} is
normalized to $\sin\theta$, as \eqref{eq:yawrate_actual}
regularizes the $1/\cos\theta$ factor in~\eqref{eq:cart_3d_sys}. In
particular, no restriction on $\theta$ is required. Consequently, in the inputs $(v, p, \bar q, \tilde{\bar r})$, the
system \eqref{eq:roll_sph} coincides with \eqref{eq:spheric_3d_sys} augmented by
$\dot\phi = p + \tilde{\bar r}\sin\theta$. 

Then, any feedback law rendering the origin of~\eqref{eq:spheric_3d_sys} GAS/GES on $\mathcal{D}$  extends to the roll-augmented system~\eqref{eq:roll_sph} on $\mathcal{D}_{\mathrm{a}}$, as does any strict CLF. The following Remark and Lemma formalize this notion.

\begin{remark}\label{rem:roll_def}
GAS/GES on $\mathcal{D}_{\mathrm{a}}$ is defined as in
Definitions~\ref{def-our-GAS} and~\ref{def-our-GES} respectively, with $(\mathcal{D}, |\cdot|_{\mathcal{D}})$
replaced by $(\mathcal{D}_{\mathrm{a}}, |\cdot|_{\mathcal{D}_{\mathrm{a}}})$ and the inputs $(v,p,q,r)$.
A strict CLF for \eqref{eq:roll_sph} is understood as in
Definition~\ref{def-CLF} with $\Theta$, $\Sigma$, and
$|\cdot|_{\mathcal{D}}$ replaced by $\Xi$,
$\Sigma_{\mathrm{a}} \coloneqq \Sigma \times \{\phi \in \mathbb{R}\}$, and
$|\cdot|_{\mathcal{D}_{\mathrm{a}}}$, and with the input vector
$\bigl(\tfrac{v}{\rho}, p, \bar q, \tfrac{\bar r}{\cos\theta}\bigr) \in
\mathbb{R}^4$.
\end{remark}

\begin{lemma}[\rev{Roll augmentation of feedbacks and strict CLFs}]\label{lem:roll_ext}
Let $\alpha_v, \alpha_q, \alpha_{\tilde r} \colon \mathcal{D} \to \mathbb{R}$ be continuous on $\mathcal{D}$ with respect to its metric, and suppose that the feedback \begin{align}\label{eq:base_law} 
(v, q, \tilde r) = \bigl(\alpha_v(\Theta), \alpha_q(\Theta), \alpha_{\tilde r}(\Theta)\bigr) \,,
\end{align} 
renders the point $\Theta = (0,0,0,0,0)$ of \eqref{eq:spheric_3d_sys} GAS (respectively, GES with constants $K \geq 1$ and $\lambda > 0$) on $\mathcal{D}$ in the sense of Definition~\ref{def-our-GAS} (or Definition~\ref{def-our-GES}). For any $k_\phi > 0$, consider the feedback
\begin{subequations}\label{eq:roll_law}
\begin{align}
v &= \alpha_v(\Theta)\,, \label{eq:roll_law_v}\\
p &= -k_\phi\,\phi - \alpha_{\tilde r}(\Theta)\sin\theta\,,
\label{eq:roll_law_p}\\
\begin{bmatrix} q\\ r\end{bmatrix}
&= M(\phi)^{\!\top}
\begin{bmatrix} \alpha_q(\Theta)\\
\alpha_{\tilde r}(\Theta)\cos\theta \end{bmatrix}\,.
\label{eq:roll_law_qr}
\end{align}
\end{subequations}
Then, the point
$\Xi= (0,\ldots,0)$ of \eqref{eq:roll_sph} in closed loop with
\eqref{eq:roll_law} is GAS (respectively, GES with constants $K$ and
$\min\{\lambda, k_\phi\}$) on $\mathcal{D}_{\mathrm{a}}$ in the sense
of Definition~\ref{def-our-GAS} (or Definition~\ref{def-our-GES}), as extended to
$\mathcal{D}_{\mathrm{a}}$ in Remark~\ref{rem:roll_def}. If, in
addition, a continuously differentiable $V$ is a strict CLF for
\eqref{eq:spheric_3d_sys} in the sense of Definition~\ref{def-CLF},
then
\begin{align}\label{eq:Va_def}
V_{\mathrm{a}}(\Xi) \coloneqq V(\Theta) + \frac{k_\phi}{2}\phi^2\,,
\end{align}
is a strict CLF for \eqref{eq:roll_sph} in the sense of
Definition~\ref{def-CLF}, as extended in Remark~\ref{rem:roll_def}.
\end{lemma}
\begin{proof}
Under \eqref{eq:roll_law_qr}, \eqref{eq:channel_rates} gives $(\bar q, \bar r) = \bigl[\alpha_q(\Theta), \alpha_{\tilde r}(\Theta)\cos\theta\bigr]^\top$. Hence, $\dot\Theta$ in \eqref{eq:roll_sph} with \eqref{eq:roll_law} coincides with $\dot\Theta$ in \eqref{eq:spheric_3d_sys} with \eqref{eq:base_law}, while \eqref{eq:roll_law_p} yields $\dot\phi = -k_\phi\phi$, so that $|\phi(t)| = |\phi(t_0)|e^{-k_\phi(t-t_0)}$. Moreover, \eqref{eq:roll_law} is continuous on $\mathcal{D}_{\mathrm{a}}$ with respect to its metric. In the GAS case, with $\beta$ from Definition~\ref{def-our-GAS}, $|\Xi(t)|_{\mathcal{D}_{\mathrm{a}}} \leq \beta\bigl(|\Theta(t_0)|_{\mathcal{D}}, t-t_0\bigr) + |\phi(t_0)|e^{-k_\phi(t-t_0)} \leq \beta_{\mathrm{a}}\bigl(|\Xi(t_0)|_{\mathcal{D}_{\mathrm{a}}}, t-t_0\bigr)$, where $\beta_{\mathrm{a}}(s,\tau) \coloneqq \beta(s,\tau) + s\,e^{-k_\phi\tau}$ is class $\mathcal{KL}$. In the GES case, $|\Xi(t)|_{\mathcal{D}_{\mathrm{a}}} \leq K|\Theta(t_0)|_{\mathcal{D}}\,e^{-\lambda(t-t_0)} + |\phi(t_0)|e^{-k_\phi(t-t_0)} \leq K|\Xi(t_0)|_{\mathcal{D}_{\mathrm{a}}}\, e^{-\min\{\lambda,\,k_\phi\}(t-t_0)}$.

For the CLF claim, $V_{\mathrm{a}}$ is continuously differentiable, and, with $s \coloneqq |\Xi|_{\mathcal{D}_{\mathrm{a}}} = |\Theta|_{\mathcal{D}} + |\phi|$, since $\max\{|\Theta|_{\mathcal{D}}, |\phi|\} \geq s/2$, it holds on $\Sigma_{\mathrm{a}}$ that $\min\{\bar\alpha_1(s/2),\, k_\phi s^2/8\} \leq V_{\mathrm{a}}(\Xi) \leq \bar\alpha_2(s) + k_\phi s^2/2$, and both bounds are class $\mathcal{K}$. Along \eqref{eq:roll_sph}, $\dot V_{\mathrm{a}} = \dot V + k_\phi\phi\bigl(p + \tilde{\bar r}\sin\theta\bigr)$, where $\dot V$ is evaluated along \eqref{eq:spheric_3d_sys} under $(q, \tilde r) \to (\bar q, \tilde{\bar r})$. Then, it follows, with~\eqref{eq:roll_law_p}, $\dot V_{\mathrm{a}} = \dot V - k_{\phi}^2\phi^2 < 0$, and $V_{\mathrm{a}}$ is a strict CLF in the sense of Definition~\ref{def-CLF} as extended in Remark~\ref{rem:roll_def}.
\end{proof}

% Lemma~\ref{lem:roll_ext} thus reduces the roll-augmented problem to the roll-free one: once continuous feedback laws $(\alpha_v, \alpha_q, \alpha_{\tilde r})$ and a strict CLF are constructed for \eqref{eq:spheric_3d_sys} on $\mathcal{D}$, Lemma~\ref{lem:roll_ext} extend them to \eqref{eq:roll_sph} on $\mathcal{D}_{\mathrm{a}}$ with no further analysis.

\begin{remark}[Arbitrary target pitch via roll augmentation]\label{rem:arbitrary_pose}
Roll augmentation removes the target-pitch restriction of Remark~\ref{rem:target_pitch}. Unlike the roll-free family, the attitude in~\eqref{eq:roll_model} spans all of $SO(3)$, the set of all rigid rotations, which is closed under re-expression relative to an arbitrary target frame. Moreover, let $R = R_z(\psi)R_y(\theta)R_x(\phi)$ denote the vehicle attitude, with $R_x$ the elementary rotation about the $x$-axis, and let $R^*$ denote the target attitude. Since the body rates $(p,q,r)$ are independent of the reference frame, the relative attitude $R^{*\top}R$ satisfies the same kinematics driven by the same inputs, and likewise for the position expressed in the target frame. The roll-augmented control designs thus extend directly to any terminal pose $(x^*,y^*,z^*,\phi^*,\theta^*,\psi^*)$ in the target frame. Their singular sets relocate accordingly, mapping the chart polar line to the target vertical axis and the singular set to perpendicular relative pitch.
\vspace{-0.5cm}
\end{remark}

%%%%%%%%%%%%%%%%%%%%%%%%%%%%%%%%%%%%%%%%%%%%%%%%%%%%%%%%%%%%%%%%%%%%%%%%%%%%%%%%
\section{Backstepping Transformation\\ and Strict Barrier CLFs}\label{sec:backstepping}

We now introduce a backstepping transformation to develop the feedback control laws for the systems~\eqref{eq:spheric_3d_sys} and~\eqref{eq:roll_sph}. First, let $k_2,k_3>0$ be arbitrary gains and define
\begin{align}
e_1 &= \gamma + \arctan(k_2\delta\cos\zeta)\label{eq:e1_def}\\
e_2 &= \theta + \arctan\bigl(\eta(\delta,\zeta)\bigr)\,,\label{eq:e2_def}
\end{align}
where
\begin{align}\label{eq:eta_def}
\eta(\delta,\zeta)\coloneqq
\frac{k_3\sin\zeta+\tan\zeta}{\sqrt{1+k_2^2\delta^2\cos^2\zeta}}\,.
\end{align}
The expression $\eta(\delta,\zeta)$ is well defined on $\mathcal D$ since $|\zeta|<\pi/2$ implies $\cos\zeta\neq 0$. \rev{The desired yaw and pitch encoded in~\eqref{eq:e1_def}--\eqref{eq:eta_def} are chosen jointly with the strictly positive surge law of Theorem~\ref{thm:GES_main} so that the pitch and yaw rate inputs cancel the indefinite terms of the backstepping construction and enforce a strict Lyapunov decrease.} Moreover, non-``unwinding'' control laws can also be derived, as in~\cite[Sec.~IX]{Part1_todorovskiCLF2025}, by suitably modifying the backstepping transformations~\eqref{eq:e1_def} and~\eqref{eq:e2_def}. However, unwinding can alternatively be avoided \rev{by initializing the angle estimates in $[-\pi,\pi]$, after which the exponential decay of the lifted angles bounds the net rotation}. Hence, we focus on the global case.

\subsection{Roll-free system}\label{sec:roll_free_control_CLF}

Now, consider the candidate Lyapunov function
\begin{align}\label{eq:V_def}
V(\Theta)
= \frac{1}{2}\left(k_1\rho^2+k_2\delta^2+k_3\tan^2\zeta+k_4 e_1^2+k_5 e_2^2\right),
\end{align}
with $k_1,k_2,k_3,k_4,k_5 >0$, which is aligned with the backstepping construction: $\rho$, $\delta$, and $\tan\zeta$ capture the remaining position-level errors, while $e_1$ and $e_2$ quantify how far the yaw and pitch states are from the backstepped ``desired'' angles in \eqref{eq:e1_def} and \eqref{eq:e2_def}. 

The $\tan^2\zeta$ term in~\eqref{eq:V_def} grows unbounded as $|\zeta| \to \pi/2$, making~\eqref{eq:V_def} a ``barrier Lyapunov function'': any feedback rendering its time derivative along the closed-loop solutions negative definite enforces $|\zeta(t)|<\pi/2$ for all $t$ whenever $|\zeta_0| < \pi/2$. The barrier occurs at $\cos\zeta = 0$, i.e., when the horizontal range $\rho\cos\zeta = \sqrt{x^2+y^2}$ vanishes. Equivalently, on $\{x^2+y^2>0\}$ we have $\tan\zeta = -z/\sqrt{x^2+y^2}$, so $\tan^2\zeta$ blows up as the state approaches the excluded polar set $\{x=y=0\}$ with $z \neq 0$. Hence $V$ is positive definite and radially unbounded on $\mathcal D$ with respect to the metric $|\Theta|_{\mathcal D}$, and $\tan^2\zeta$ acts as a barrier preventing trajectories from approaching the coordinate singularity.

\subsubsection{Exponentially stabilizing feedback and strict barrier CLF} 

Having established the backstepping transformation, we now construct the feedback control law that renders the origin of~\eqref{eq:spheric_3d_sys} GES on $\mathcal{D}$. Specifically, we leverage the surge velocity to remove the singularity from $1/\rho$, and design the pitch and yaw rate inputs based on the backstepping transformations~\eqref{eq:e1_def} and~\eqref{eq:e2_def}. This result is formalized in our main theorem.

\begin{theorem}[\rev{Explicit continuous feedback achieves GES with user-assigned rate}]\label{thm:GES_main}
Consider the system \eqref{eq:spheric_3d_sys}\rev{, with $(e_1,e_2)$ as in~\eqref{eq:e1_def} and~\eqref{eq:e2_def}, respectively, and arbitrary control gains $k_1,k_2,k_3,k_4,k_5 > 0$}. Then, the forward surge velocity input
\begin{align}\label{eq:v_ctrl}
v = k_1 \rho N B\,,
\end{align}
the pitch rate input

\begin{align}\label{eq:pitch_ctrl}
    q=& -k_5e_2 +\frac{k_1k_2^2}{k_5}\delta^2\sigma_2(e_2,\theta)B\nonumber\\
    &+\frac{k_1}{N^2B}\Bigl[k_2^2\delta \cos(\theta)\cos(\zeta)\sin(\gamma)\bigl(k_3\sin(\zeta) + \tan(\zeta)\bigr)\nonumber\\
    &- \bigl(\cos(\theta)\sin(\zeta)\cos(\gamma) + \sin(\theta)\cos(\zeta)\bigr)C\Bigr]\nonumber\\
    &-\frac{k_1^2}{k_5}\rho^2 B\Bigl(\sin(\zeta)\sigma_1(e_2,\theta)N - \cos(\zeta) \sigma_2(e_2,\theta)\Bigr)\nonumber\\
    &-\frac{k_1k_3}{k_5}\frac{\tan(\zeta)B}{\cos^2\zeta}\Bigl(\sin(\zeta)\sigma_2(e_2,\theta) + \cos(\zeta)\sigma_1(e_2,\theta)N\Bigr)\,,
\end{align}
and the yaw rate input
\begin{align}\label{eq:yaw_ctrl}
    \tilde{r} =& k_4e_1 + \frac{k_1NB}{\cos\zeta}\cos(\theta)\sin(\gamma) +\frac{k_1k_2B}{N}\Biggl[\cos(\theta)\sin(\gamma) \nonumber\\
    &-\delta\left(\cos(\theta)\sin^2(\zeta)\cos(\gamma) + \frac{\sin(\theta)\sin(2\zeta)}{2}\right)\Biggr]\nonumber\\
    &-\frac{k_1^2}{k_4}\rho^2 \cos(\zeta)N\Bigl(\sigma_2(e_1,\gamma) + e_2\sigma_2(e_1,\gamma)\sigma_2(e_2,\theta)B\Bigr)\nonumber\\
    &+\frac{k_1k_2}{k_4}\frac{\delta\sigma_1(e_1,\gamma)N}{\cos\zeta}\Bigl(1+e_2\sigma_2(e_2,\theta)B\Bigr)\nonumber\\
    &+\frac{k_1k_3}{k_4}\frac{\tan(\zeta)N}{\cos^2\zeta}\sin(\zeta)\Bigl(\sigma_2(e_1,\gamma)\nonumber\\
    &+e_2\sigma_2(e_1,\gamma)\sigma_2(e_2,\theta)B\Bigr)\,,
\end{align}
with~\eqref{eq:yawrate_actual}, renders the point $\Theta = (0,0,0,0,0)$ GES for the closed-loop system on $\mathcal D$ in accordance with Definition~\ref{def-our-GES}\rev{: the closed-loop solution exists, is unique, and satisfies
\begin{align}\label{eq:ges_bound}
    |\Theta(t)|_\mathcal{D}\leq K|\Theta(t_0)|_\mathcal{D}e^{-c (t-t_0)}\,,
\end{align}
where $c = \min\{k_1, k_1k_2, k_1k_3, k_4, k_5\}$, $K = \sqrt{c_+/c_-}$ with
\begin{align}\label{eq:cpm_def}
c_- &= \frac{1}{10}\min\left\{k_1, \frac{k_2}{1+2k_2^2}, \frac{k_3}{1+2(k_3+1)^2}, \frac{k_4}{2}, \frac{k_5}{2}\right\}\nonumber\\
c_+ &= \frac{1}{2}\max\left\{k_1, k_2+2k_4k_2^2, 2k_4, k_3+2k_5(k_3+1)^2, 2k_5\right\},
\end{align}
the functions $N$, $B$, and $C$ are defined as
\begin{align}
N(\delta,\zeta) \coloneqq& \sqrt{1+k_2^2\delta^2\cos^2\zeta}\label{eq:N_def}\\
B(\delta,\zeta) \coloneqq& \sqrt{1+\eta^2(\delta,\zeta)} \label{eq:B_def}\\
C(\delta,\zeta) \coloneqq& 1+\revh{k_3}\cos\zeta+\tan^2\zeta\nonumber\\
&+k_2^2\delta^2\bigl(2+\revh{k_3}\cos\zeta-\cos^2\zeta\bigr)\,,
\end{align}
with the explicit dependence on $(\delta,\zeta)$ hereafter suppressed for brevity, and the auxiliary functions $\sigma_1$, $\sigma_2$ as
\begin{align}
\sigma_1(a,b) &\coloneqq \frac{\sin(a-b)+\sin b}{a} \label{eq:sigma1_def}\\
\sigma_2(a,b) &\coloneqq \frac{\partial\sigma_1(a,b)}{\partial b}
= \frac{\cos b - \cos(a-b)}{a}\,.\label{eq:sigma2_def}
\end{align}}
Moreover, the Lyapunov function~\eqref{eq:V_def} is a strict CLF in accordance with Definition~\ref{def-CLF}.
\end{theorem}

\begin{proof}
We first note that for all $\Theta \in \mathcal D$, $N\geq 1$, $B\ge 1$, $C \revh{> 1}$, and $\cos\zeta > 0$. Additionally, the auxiliary functions~\eqref{eq:sigma1_def} and~\eqref{eq:sigma2_def} admit continuous extensions at $a=0$ with $\sigma_1(0,b)=\cos b$ and $\sigma_2(0,b)=-\sin b$, and with \eqref{eq:e1_def} and \eqref{eq:e2_def}, yield the following equivalent expressions
\begin{align}
\sin\gamma &= -\frac{k_2\delta\cos\zeta}{\sqrt{1+k_2^2\delta^2\cos^2\zeta}}
+ e_1\sigma_1(e_1,\gamma) \label{eq:sin_gamma_exp}\\
\cos\gamma &= \frac{1}{\sqrt{1+k_2^2\delta^2\cos^2\zeta}}
+ e_1\sigma_2(e_1,\gamma) \label{eq:cos_gamma_exp}\\
\sin\theta &= -\frac{\eta(\delta,\zeta)}{\sqrt{1+\eta^2(\delta,\zeta)}}
+ e_2\sigma_1(e_2,\theta) \label{eq:sin_theta_exp}\\
\cos\theta &= \frac{1}{\sqrt{1+\eta^2(\delta,\zeta)}}
+ e_2\sigma_2(e_2,\theta). \label{eq:cos_theta_exp}
\end{align}
Then, after lengthy calculations, substituting \eqref{eq:v_ctrl} and using \eqref{eq:sin_gamma_exp}--\eqref{eq:cos_theta_exp}, the transformed dynamics of $(\rho,\delta,\gamma)$ yield
\begin{subequations}\label{eq:backstepped_sys1}
    \begin{align}
    \dot\rho =& -k_1\rho\left(\frac{1}{\cos \zeta} + \revh{k_3}\sin^2\zeta\right)\nonumber\\
    &+ k_1\rho B\Bigl(\sin(\zeta)\sigma_1(e_2,\theta)N - \cos(\zeta) \sigma_2(e_2,\theta)\Bigr)e_2\nonumber\\
    & -k_1\rho \cos(\zeta)N\Bigl(\sigma_2(e_1,\gamma) + e_2\sigma_2(e_1,\gamma)\sigma_2(e_2,\theta)B\Bigr)e_1\\
    \dot\delta =& -k_1k_2\delta -\Bigl(k_1k_2\delta\sigma_2(e_2,\theta)B\Bigr)e_2 \nonumber\\
    &+\frac{k_1\sigma_1(e_1,\gamma)N}{\cos\zeta}\Bigl(1+e_2\sigma_2(e_2,\theta)B\Bigr)e_1\\
    \dot\zeta =& -k_1k_3\tan(\zeta)\cos^2(\zeta) \nonumber\\
    &+k_1B\Bigl(\sin(\zeta)\sigma_2(e_2,\theta) + \cos(\zeta)\sigma_1(e_2,\theta)N\Bigr)e_2\nonumber\\
    &+ k_1N\revh{\sin(\zeta)}\Bigl(\sigma_2(e_1,\gamma)+e_2\sigma_2(e_1,\gamma)\sigma_2(e_2,\theta)B\Bigr)e_1\,,
    \end{align}
\end{subequations}
and the error dynamics of $(e_1,e_2)$ are given as
\begin{subequations}\label{eq:backstepped_sys2}
% \hspace*{-0.5cm}
    \begin{align}
    \dot e_1 =& -\revh{\tilde{r}} + \frac{k_1NB}{\cos\zeta}\cos(\theta)\sin(\gamma)+\frac{k_1k_2B}{N}\Biggl[\cos(\theta)\sin(\gamma)\nonumber\\
    &-\delta\left(\cos(\theta)\sin^2(\zeta)\cos(\gamma) + \frac{\sin(\theta)\sin(2\zeta)}{2}\right)\Biggr]\\
    \dot e_2 =& \, q-\frac{k_1}{N^2B}\Bigl[k_2^2\delta \cos(\theta)\cos(\zeta)\sin(\gamma)\bigl(\revh{k_3}\sin(\zeta) + \tan(\zeta)\bigr)\nonumber\\
    &- \bigl(\cos(\theta)\sin(\zeta)\cos(\gamma) + \sin(\theta)\cos(\zeta)\bigr)C\Bigr]\,.
    \end{align}
\end{subequations}
Subsequently, with~\eqref{eq:yawrate_actual}, choosing $q$ and $\tilde{r}$ as in~\eqref{eq:pitch_ctrl} and~\eqref{eq:yaw_ctrl}, and taking the time derivative of~\eqref{eq:V_def} along the solution of~\eqref{eq:backstepped_sys1} and~\eqref{eq:backstepped_sys2} yield
\begin{align}\label{eq:Vdot}
\dot V =& -k_1^2\rho^2\left(\frac{1}{\cos\zeta}+\revh{k_3}\sin^2\zeta\right)
-k_1k_2^2\delta^2\nonumber\\
&-k_1k_3^2\tan^2\zeta-k_4^2e_1^2-k_5^2e_2^2.
\end{align}
Since $|\zeta|<\pi/2$ on $\mathcal D$, we have $\frac{1}{\cos\zeta}+k_3\sin^2\zeta \geq 1$. Thus,
\begin{align}\label{eq:Vdot_bound}
\dot V \le -2cV,
\end{align}
where $c$ is as in the theorem statement. \rev{Moreover, from \eqref{eq:e1_def} and~\eqref{eq:e2_def}, using $|\arctan s| \le |s|$, $N \ge 1$, and $|\sin\zeta| \le |\tan\zeta|$ on $\mathcal D$, it holds that $|e_1 - \gamma| \le k_2|\delta|$ and $|e_2 - \theta| \le (k_3+1)|\tan\zeta|$, and conversely $|\gamma| \le |e_1| + k_2|\delta|$ and $|\theta| \le |e_2| + (k_3+1)|\tan\zeta|$. Squaring, $e_1^2 \le 2\gamma^2 + 2k_2^2\delta^2$ and $e_2^2 \le 2\theta^2 + 2(k_3+1)^2\tan^2\zeta$ give $V \le c_+|\Theta|_{\mathcal D}^2$, while $\gamma^2 \le 2e_1^2 + 2k_2^2\delta^2$, $\theta^2 \le 2e_2^2 + 2(k_3+1)^2\tan^2\zeta$, and $|\Theta|_{\mathcal D}^2 \le 5\left(\rho^2+\delta^2+\gamma^2+\tan^2\zeta+\theta^2\right)$ give $c_-|\Theta|_{\mathcal D}^2 \le V$, with $c_\pm$ as in~\eqref{eq:cpm_def}. In particular, $V$ is positive definite and radially unbounded on $\mathcal D$, and \eqref{eq:V_def} is a strict CLF in accordance with Definition~\ref{def-CLF}. Integrating \eqref{eq:Vdot_bound} and applying the quadratic bounds yields \eqref{eq:ges_bound} with $K = \sqrt{c_+/c_-}$. Finally, with $v/\rho = k_1NB$, the closed-loop right-hand side of \eqref{eq:spheric_3d_sys} under \eqref{eq:v_ctrl}--\eqref{eq:yaw_ctrl} is smooth on $\Sigma$, and \eqref{eq:Vdot_bound} confines every solution to the compact sublevel set $\{V \le V(\Theta(t_0))\}$ of $\Sigma$, so the closed-loop solution exists and is unique on $[t_0,\infty)$.} Thus, the point $\Theta = (0,0,0,0,0)$ is GES on $\mathcal{D}$ in accordance with Definition~\ref{def-our-GES}.
\end{proof}

From~\eqref{eq:Vdot_bound} we have $V(t)\le V(t_0)$ for all $t\ge t_0$. Since $V(\Theta)\to\infty$ as $|\zeta|\to \pi/2$, no trajectory starting in $\mathcal D$ can reach the boundary $|\zeta|=\pi/2$, and hence, the system cannot approach the singular set. $\mathcal{D}$ is the largest set on which the spherical coordinates (and any continuous time-invariant design in these coordinates) are well defined, and the unavoidable exclusion of the undefined set is precisely the coordinate singularity that permits circumventing Brockett’s obstruction. In Cartesian coordinates, $\mathcal D$ corresponds to $\{(x,y,z,\psi,\theta)\in\mathbb R^5\mid x^2+y^2>0\}$, i.e., all initial conditions except the measure-zero $z$-axis where the coordinate transformation is undefined.

\subsubsection{Exponential attractivity in cartesian coordinates} 

While Theorem~\ref{thm:GES_main} establishes GES in spherical coordinates on $\mathcal D$, one cannot infer exponential \emph{stability} of the closed-loop system in Cartesian coordinates $(x,y,z,\psi,\theta)$, a fundamental restriction also noted in the planar unicycle case~\cite[Cor. 2]{Part1_todorovskiCLF2025}. This is not a drawback of the construction. It is consistent with the topological obstruction behind Brockett’s condition and with the classical fact that static stabilization of a 3D nonholonomic vehicle in Cartesian coordinates is impossible, even with discontinuous feedback \cite[Remark~1.6]{coron1994relation}, \cite[pg.~43]{praly2022fonctions}. Nevertheless, the spherical exponential estimate does yield the practically relevant conclusion that $(x,y,z,\psi,\theta)=(0,0,0,0,0)$ is attractive and, in fact, exponentially convergent for every initial condition with $x_0^2+y_0^2>0$. Corollary~\ref{cor:cartesian_GES} quantifies this implication.

\begin{corollary}
\label{cor:cartesian_GES}
Consider the system~\eqref{eq:cart_3d_sys} in closed-loop with the control law~\eqref{eq:v_ctrl}, \eqref{eq:pitch_ctrl}, and~\eqref{eq:yaw_ctrl} with~\eqref{eq:yawrate_actual}. Then, for all initial conditions $(x_0,y_0,z_0,\psi_0,\theta_0)\in\mathbb{R}^5$ such that $x_0^2 + y_0^2 > 0$, the following holds:
\begin{align}\label{eq:cart_GES_bound}
&|x(t)|+|y(t)|+|z(t)|+|\psi(t)|+|\theta(t)|\nonumber\\
&\le \beta\Biggl(|x_0|+|y_0| + |z_0| + |\psi_0|+|\theta_0|\nonumber\\
&\qquad \quad +|{\rm atan2}(-y_0,-x_0)| + \frac{|z_0|}{\sqrt{x_0^2 + y_0^2}}\,, \,t-t_0\Biggr),
\end{align}
where $\beta \in \mathcal{KL}$ is defined as $\beta(r,t-t_0) \coloneqq 2\sqrt{3}Kre^{-\rev{c}(t-t_0)}$.
\end{corollary}

Corollary~\ref{cor:cartesian_GES} follows immediately from Theorem~\ref{thm:GES_main} by translating the spherical exponential estimate~\eqref{eq:ges_bound} into Cartesian variables through the coordinate relations in Table~\ref{tab:spherical_coords}. In particular, $\rho=\sqrt{x^2+y^2+z^2}$ and $\psi=\delta-\gamma$ imply $|x|+|y|+|z|\le \sqrt{3}\,\rho$ and $|\psi|\le|\delta|+|\gamma|$, while $x_0^2+y_0^2>0$ gives $|\tan\zeta_0|=|z_0|/\sqrt{x_0^2+y_0^2}$. Then,~\eqref{eq:cart_GES_bound} follows directly from~\eqref{eq:ges_bound}.

% \begin{proof}
% By Table~\ref{tab:spherical_coords}, $\rho=\sqrt{x^2+y^2+z^2}$ and $\psi=\delta-\gamma$.
% Hence $|x|+|y|+|z|\le \sqrt{3}\,\rho$ and $|\psi|\le |\delta|+|\gamma|$, which implies
% \begin{align}
% |x|+|y|+|z|+|\psi|+|\theta|
% &\le \sqrt{3}\rho + |\delta|+|\gamma|+|\theta|
% \le \sqrt{3}\big(\rho+|\delta|+|\gamma|+|\tan\zeta|+|\theta|\big)
% = \sqrt{3}\,|\Theta|_{\mathcal D}.
% \end{align}
% Combining this inequality with the GES estimate
% $|\Theta(t)|_{\mathcal D}\le K|\Theta(t_0)|_{\mathcal D}e^{-\lambda(t-t_0)}$
% from Theorem~\ref{thm:GES_main} yields~\eqref{eq:cart_GES_bound}.
% Finally, \eqref{eq:Theta_metric_cartesian} follows by substituting the coordinate
% definitions $\rho=\sqrt{x^2+y^2+z^2}$, $\delta={\rm atan2}(y,x)+\pi$,
% $\gamma={\rm atan2}(y,x)-\psi+\pi$, and $\tan\zeta=z/\sqrt{x^2+y^2}$ into
% \eqref{eq:metricD}.
% \end{proof}

%%%%%%%%%%%%%%%%%%%%%%%%%%%%%%%%%%%%%%%%%%%%%%%%%%%%%%%%%%%%%%%%%%%%%%%%%%%%%%%%

\subsection{Roll-augmented system}\label{sec:roll_results}

Now, from Lemma~\ref{lem:roll_ext}, the following results for the roll-augmented kinematics~\eqref{eq:roll_sph} follow.
 
\begin{theorem}[\rev{Roll augmentation preserves GES and the strict CLF}]\label{thm:roll_GES}
Consider the system~\eqref{eq:roll_sph} and let
$k_1,\ldots,k_5,k_\phi > 0$. Then, the
feedback~\eqref{eq:roll_law}, with $\alpha_v$, $\alpha_q$, and
$\alpha_{\tilde r}$ given by \eqref{eq:v_ctrl}, \eqref{eq:pitch_ctrl},
and~\eqref{eq:yaw_ctrl}, respectively with $(e_1,e_2)$ defined in \eqref{eq:e1_def} and~\eqref{eq:e2_def}, renders the point $\Xi= (0,\ldots,0)$ GES
on $\mathcal{D}_{\mathrm{a}}$ with constants $K$ and
$\min\{\rev{c}, k_\phi\}$, where $K$ and \rev{$c$} are as
in~\eqref{eq:ges_bound}, in accordance with
Definition~\ref{def-our-GES} as extended in Remark~\ref{rem:roll_def}.
Moreover, the function
\begin{align}\label{eq:Va_explicit}
V_{\mathrm{a}}(\Xi)
= \frac{1}{2}\left(k_1\rho^2+k_2\delta^2+k_3\tan^2\zeta+k_4e_1^2+k_5e_2^2+k_\phi\phi^2\right),
\end{align}
with $k_1,k_2,k_3,k_4,k_5,k_\phi > 0$, is a
strict CLF for~\eqref{eq:roll_sph} in accordance with
Definition~\ref{def-CLF} as extended in Remark~\ref{rem:roll_def}.
\end{theorem}
\begin{proof}
Directly from Lemma~\ref{lem:roll_ext} and Theorem~\ref{thm:GES_main}.
\end{proof}
 
As in the roll-free case, Theorem~\ref{thm:roll_GES} does not imply
exponential stability in the Cartesian variables
$(x,y,z,\phi,\psi,\theta)$, but it does yield exponential attractivity
of the origin for every initial condition off the polar set
$\{x = y = 0\}$. The translation to Cartesian coordinates is identical
to that of Corollary~\ref{cor:cartesian_GES}, using additionally
$|\Xi|_{\mathcal{D}_{\mathrm{a}}} = |\Theta|_{\mathcal{D}} + |\phi|$
from~\eqref{eq:aug_metric}.
 
\begin{corollary}\label{cor:roll_cartesian}
Consider the system~\eqref{eq:roll_model} in closed loop with the
control law of Theorem~\ref{thm:roll_GES}. Then, for all initial
conditions $(x_0,y_0,z_0,\phi_0,\psi_0,\theta_0) \in \mathbb{R}^6$ such
that $x_0^2 + y_0^2 > 0$, the following holds:
\begin{align}\label{eq:roll_cart_bound}
&|x(t)|+|y(t)|+|z(t)|+|\phi(t)|+|\psi(t)|+|\theta(t)|\nonumber\\
&\le \beta_{\mathrm{a}}\Biggl(|x_0|+|y_0|+|z_0|+|\phi_0|+|\psi_0|+|\theta_0|\nonumber\\
&\qquad\quad +|{\rm atan2}(-y_0,-x_0)| + \frac{|z_0|}{\sqrt{x_0^2 + y_0^2}}\,,\,t-t_0\Biggr),
\end{align}
where ${\beta_{\mathrm{a}} \in \mathcal{KL}}$ and ${\beta_{\mathrm{a}}(r,t-t_0) \coloneqq
2\sqrt{3}Kre^{-\min\{\rev{c},\,k_\phi\}(t-t_0)}}$.
\end{corollary}
%%%%%%%%%%%%%%%%%%%%%%%%%%%%%%%%%%%%%%%%%%%%%%%%%%%%%%%%%%%%%%%%%%%%%%%%%%%%%%%%

\subsection{Numerical parking simulation example}

We evaluate the closed-loop performance of the roll-augmented controller of Theorem~\ref{thm:roll_GES} from three initial conditions spanning different octants and orientations. By Lemma~\ref{lem:roll_ext}, this closed-loop system produces very similar trajectories as Theorem~\ref{thm:GES_main} with just the exponentially decaying roll-state. Hence, the roll-free case is not shown separately. The resulting spatial paths and control inputs shown in Fig.~\ref{fig:thm2_traj} and the state transients in Fig.~\ref{fig:thm2_states}. The vehicle docks smoothly at the origin, with the Cartesian position and orientation converging exponentially to zero and without the chattering or oscillations typical of time-varying and discontinuous nonholonomic controllers. This transient follows directly from the exponentially stable spherical closed loop (Fig.~\ref{fig:thm2_states}, left), in which the elevation angle stays away from the coordinate singularity at $\zeta = \pm\pi/2$, while the control inputs remain continuous throughout the maneuver (Fig.~\ref{fig:thm2_traj}, right).

\begin{figure*}[!t]
\centering
\begin{minipage}[c]{0.34\textwidth}\centering\includegraphics[width=.95\linewidth]{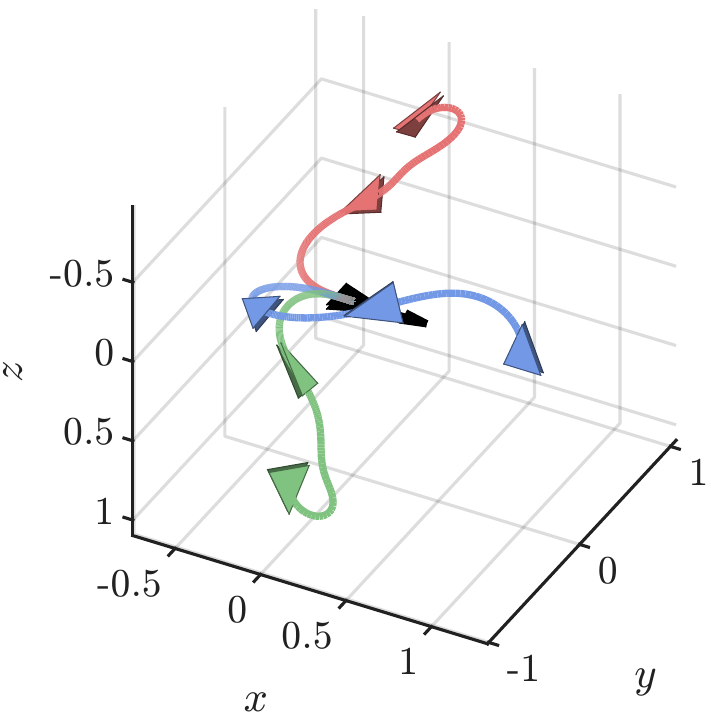}\end{minipage}\hfill%
\begin{minipage}[c]{0.62\textwidth}\centering\includegraphics[width=.95\linewidth]{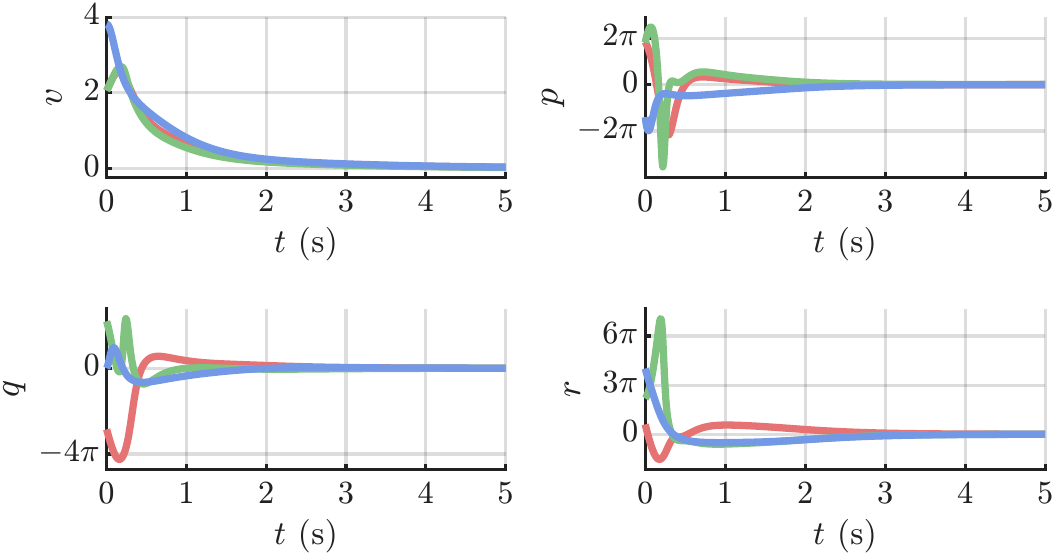}\end{minipage}
\caption{Closed-loop response of the roll-augmented system~\eqref{eq:roll_sph} under the feedback~\eqref{eq:roll_law} of Theorem~\ref{thm:roll_GES} from $\Xi_1(0) = (1,-\pi/2,-3\pi/4,\pi/4,\pi/4,\pi/3)$, $\Xi_2(0) = (1,\pi/2,\pi/2,-\pi/4,-\pi/2,-\pi/4)$, and $\Xi_3(0) = (1,\pi,0,0,\pi/2,\pi/2)$ (red, green, blue), with gains $(k_1,k_2,k_3,k_4,k_5,k_\phi) = (0.6,2,2,2,2,1)$. The black dart and arrow mark the target at the origin. The inputs $(v,p,q,r)$ of~\eqref{eq:roll_law} with~\eqref{eq:v_ctrl}, \eqref{eq:pitch_ctrl}, and~\eqref{eq:yaw_ctrl}, continuous throughout the maneuver.}
\label{fig:thm2_traj}
\vspace{-0.3cm}
\end{figure*}

\begin{figure*}[!t]
\centering
\begin{minipage}[c]{0.48\textwidth}\centering\includegraphics[width=\linewidth]{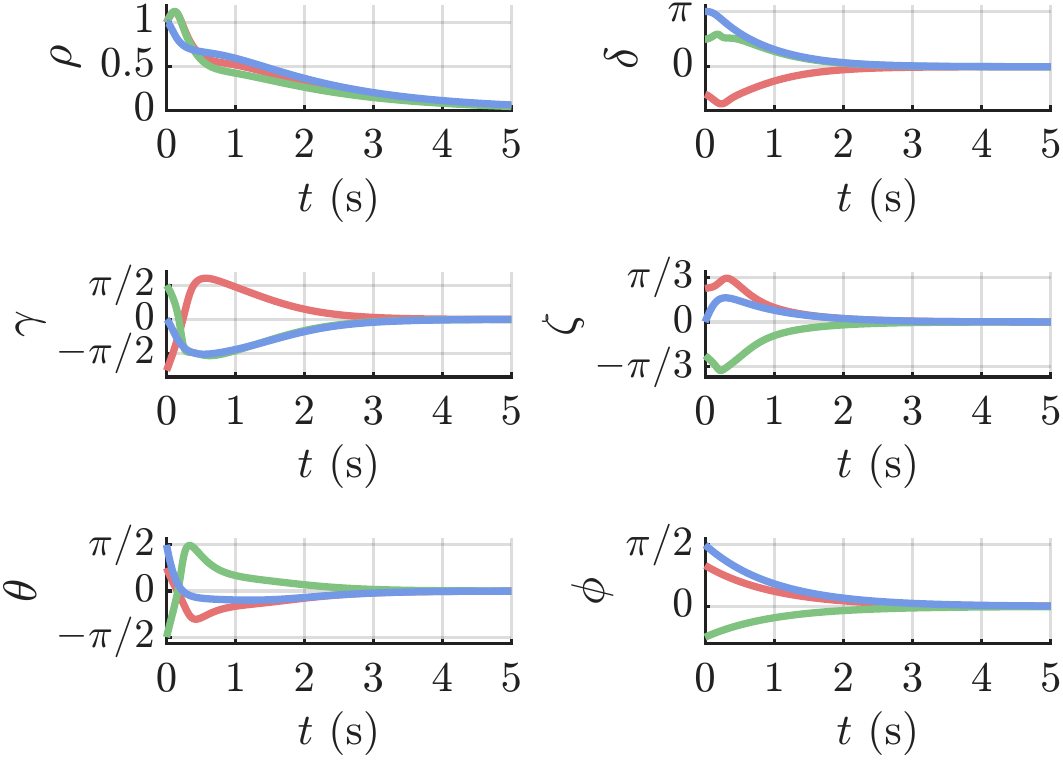}\end{minipage}\hfill%
\begin{minipage}[c]{0.48\textwidth}\centering\includegraphics[width=\linewidth]{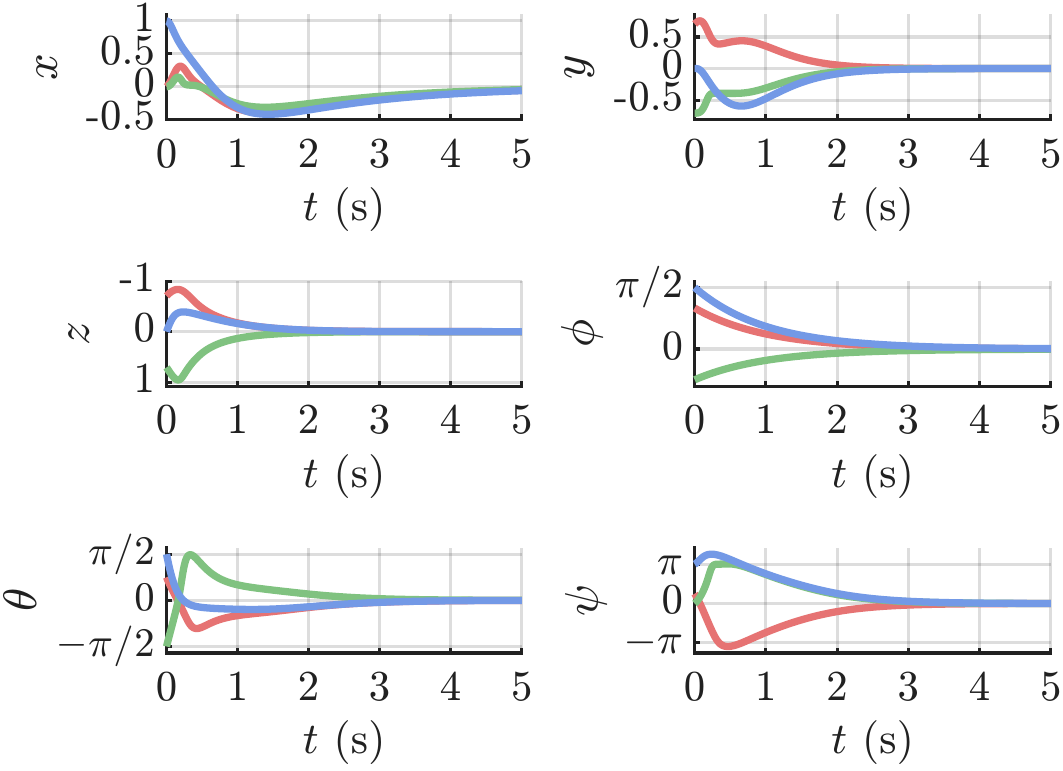}\end{minipage}
\caption{Theorem~\ref{thm:roll_GES}, same trajectories as Fig.~\ref{fig:thm2_traj}: spherical states (left) and Cartesian position and Euler angles (right). The elevation $\zeta$ stays clear of $\pm\pi/2$ and the roll decays as $\phi(t) = \phi(0)e^{-k_\phi t}$.}
\label{fig:thm2_states}
\vspace{-0.5cm}
\end{figure*}
%%%%%%%%%%%%%%%%%%%%%%%%%%%%%%%%%%%%%%%%%%%%%%%%%%%%%%%%%%%%%%%%%%%%%%%%%%%%%%%%

\section{Inverse Optimal Design for Driftless Multi-Input Systems}\label{sec:ioc_general}

Classical optimal control requires solving the HJB PDE, which is analytically intractable in all but a few special cases, and remains so even for the three-state planar unicycle. The five-dimensional spherical dynamics~\eqref{eq:spheric_3d_sys} (six in~\eqref{eq:roll_sph}) are no easier. The appeal of inverse optimal control (IOC) is that it \rev{provides} a meaningful optimality guarantee without solving the HJB PDE, while permitting user-assignable input constraints~\cite{Part2_kimIOC2025}. The latter addresses a practical limitation of the nominal design, whose control effort in \eqref{eq:v_ctrl}--\eqref{eq:yaw_ctrl} grows unbounded for initial conditions approaching the excluded set. We develop the design for general driftless control-affine systems, in the class of damping $L_gV$ controllers~\cite{jurdjevic1978controllability}, \cite[Prop.~5.9.1]{sontag2013mathematical}. We first recall the Legendre--Fenchel transform.
 
\begin{definition}(Legendre--Fenchel transform)\label{def:legendre_Fenchel_3d}
Let $\mu$ be a class $\mathcal{K}_\infty[0,w)$ function whose derivative $\mu'$ is also a class $\mathcal{K}_\infty[0,w)$ function, where $w>0$ is finite or infinite. The mapping 
\begin{align} 
\ell\mu(s) \coloneqq \int_{0}^{s}(\mu')^{-1}(\tau)\,d\tau \end{align} 
is the Legendre--Fenchel transform, where $(\mu')^{-1}$ denotes the inverse of $d\mu(s)/ds$.
\end{definition}

\subsection{Families of inverse optimal stabilizers}

We now derive a family of inverse optimal controllers for the general driftless system based on the existence of a strict CLF.
\begin{lemma}[Inverse optimality for multi-input driftless systems]\label{lem:mi_ioc}
Let $m \in \mathbb{N}$ and let $\mathcal{X} \subseteq \mathbb{R}^n$ be a state space with metric $|\cdot|_{\mathcal{X}}$. Consider the system
\begin{align}\label{eq:mi_sys}
\dot{x} = \sum_{i=1}^{m} g_i(x)\,\bar u_i\,,
\end{align}
where the vector fields $g_i$ are continuous on $\mathcal{X}$ with respect to its metric and $\bar u_i \in \mathbb{R}$, $i\in\{1,\ldots,m\}$, are the inputs. Let $x \mapsto V$ be continuously differentiable and such that (i) there exist class $\mathcal{K}$ functions $(\bar\alpha_1,\bar\alpha_2)$ such that $\bar\alpha_1(|x|_{\mathcal{X}}) \le V(x) \le \bar\alpha_2(|x|_{\mathcal{X}})$ for all $x \in \mathcal{X}$\rev{, with the sets $\{x \in \mathcal{X} : V(x) \le c\}$ compact for every $c > 0$}, and (ii) the Lie derivatives
\begin{align}\label{eq:mi_nu}
\nu_i(x) \coloneqq L_{g_i}V(x)\,,
\end{align}
where $ i\in\{1,\ldots,m\}$, satisfy $(\nu_1(x),\ldots,\nu_m(x)) = (0,\ldots,0)$ if and only if $x = 0$. Let $\mu_i\in\mathcal{K}_\infty[0,w_i)$, where $w_i > 0$ is finite or infinite, and let $\varepsilon_i(x)$ be continuous positive scalar-valued functions. Then, the cost functional
\begin{align}\label{eq:mi_J}
J = \int_0^\infty \Biggl[l(x) + \sum_{i=1}^{m}\mu_i\!\left(\frac{|\bar u_i|}{\varepsilon_i}\right)\Biggr] {\rm d}t\,,
\end{align}
where
\begin{align}\label{eq:mi_l}
l(x) \coloneqq \sum_{i=1}^{m}\ell\mu_i(\varepsilon_i|\nu_i|)\,,
\end{align}
and for all initial conditions on $\mathcal{X}$, is minimized\rev{, within the class of controls under which the solution of~\eqref{eq:mi_sys} exists on $[t_0,\infty)$ and $x(t) \to 0$,} by the feedback law
\begin{align}\label{eq:mi_opt}
\bar u_i^* = -\varepsilon_i (\mu_i')^{-1}(\varepsilon_i|\nu_i|)\,\sgn(\nu_i)\,,
\end{align}
where $ i\in\{1,\ldots,m\}$. \rev{The feedback~\eqref{eq:mi_opt} renders $x = (0,\ldots,0)$ GAS on $\mathcal{X}$.}
Additionally, the continuous feedback
\begin{align}\label{eq:mi_damp}
\bar u_i = -\varepsilon_i \frac{\ell\mu_i(\varepsilon_i|\nu_i|)}{\varepsilon_i|\nu_i|}\,\sgn(\nu_i)\,,
\end{align}
where $ i\in\{1,\ldots,m\}$,
is continuous in $(\nu_1,\ldots,\nu_m)$ and renders $x=(0,\ldots,0)$ GAS on $\mathcal{X}$, i.e., there exists a class $\mathcal{KL}$ function $\beta$ such that, for all $t\geq t_0$, it holds that $|x(t)|_{\mathcal{X}}\leq \beta\bigl(|x(t_0)|_{\mathcal{X}},t-t_0\bigr)$.
\end{lemma}

\begin{proof}
\textbf{Step 1:} \eqref{eq:mi_damp} is continuous in $(\nu_1,\ldots,\nu_m)$ and stabilizes \eqref{eq:mi_sys}.
From Definition~\ref{def:legendre_Fenchel_3d} and the Leibniz rule,
\begin{align}
(\ell\mu)'(s) = (\mu')^{-1}(s).
\end{align}
Applying L'H\^opital's rule gives
\begin{align}\label{eq:mi_lhopital}
\lim_{s\to 0}\frac{\ell\mu(s)}{s}=\lim_{s\to 0}(\mu')^{-1}(s)=0\,.
\end{align}
Since $\mu'$ is class $\mathcal{K}_\infty$ and thus $(\mu')^{-1}$ is class $\mathcal{K}_\infty$ and equals $0$ at $s=0$.
Hence each component in \eqref{eq:mi_damp} is continuous in $\nu_i$. Next, along \eqref{eq:mi_sys},
\begin{align}
\dot V
= \sum_{i=1}^{m}\nu_i\,\bar u_i\,.
\end{align}
Substituting \eqref{eq:mi_damp} yields
\begin{align}\label{eq:mi_Vdot_damp}
\dot V\big|_{\eqref{eq:mi_damp}}
= -\sum_{i=1}^{m}\ell\mu_i(\varepsilon_i|\nu_i|) < 0\,.
\end{align}
By \cite[Lemma A1]{krstic1998inverse}, $\ell\mu_i$ are class $\mathcal{K}_\infty$ functions, and, by condition (ii), $(\nu_1,\ldots,\nu_m) = (0,\ldots,0)$ if and only if $x = 0$. Thus, \eqref{eq:mi_Vdot_damp} holds. \rev{Moreover, on each compact set $\{x \in \mathcal{X} : a \le V(x) \le c\}$ with $0 < a \le c$, the continuous right-hand side of \eqref{eq:mi_Vdot_damp} has a negative-valued maximum, so $V \to 0$, and by the comparison bounds of (i),} $x=0$ is GAS on $\mathcal{X}$ under \eqref{eq:mi_damp}.

\textbf{Step 2:} \eqref{eq:mi_opt} stabilizes \eqref{eq:mi_sys}.
Substituting \eqref{eq:mi_opt} into $\dot V=\sum_{i=1}^{m}\nu_i\bar u_i$ gives
\begin{align}\label{eq:mi_Vdot_opt}
\dot V\big|_{\eqref{eq:mi_opt}}
= -\sum_{i=1}^{m} \varepsilon_i|\nu_i|\,(\mu_i')^{-1}(\varepsilon_i|\nu_i|).
\end{align}
Using the identity (see~\cite[Lemma A1]{krstic1998inverse})
\begin{align}
s(\mu')^{-1}(s)=\ell\mu(s)+\mu\big((\mu')^{-1}(s)\big),
\end{align}
\eqref{eq:mi_Vdot_opt} becomes
\begin{align}
\dot V\big|_{\eqref{eq:mi_opt}}
&= -\sum_{i=1}^{m} \ell\mu_i(\varepsilon_i|\nu_i|)
   -\sum_{i=1}^{m} \mu_i\!\Big((\mu_i')^{-1}(\varepsilon_i|\nu_i|)\Big)\nonumber\\
&\le -\sum_{i=1}^{m} \ell\mu_i(\varepsilon_i|\nu_i|)
= \dot V\big|_{\eqref{eq:mi_damp}} < 0\,.
\end{align}
Hence, \rev{from Step 1,} $x=0$ is GAS on $\mathcal{X}$ under \eqref{eq:mi_opt}.

\textbf{Step 3:} \eqref{eq:mi_opt} minimizes \eqref{eq:mi_J}.
Add and subtract $\dot V$ inside the integral and use $\lim_{t\to\infty}V(x(t))=0$\rev{, which holds for every control in the admissible class,} to write
\begin{align}
J
&= \int_0^\infty \Biggl[l(x)+\sum_{i=1}^{m}\mu_i\!\left(\frac{|\bar u_i|}{\varepsilon_i}\right)
+\dot V-\dot V\Biggr] {\rm d}t\nonumber\\
&= V(x(0))-\lim_{t\to\infty}V(x(t))\nonumber\\
&\quad +\int_0^\infty\Biggl[l(x)
+\sum_{i=1}^{m}\mu_i\!\left(\frac{|\bar u_i|}{\varepsilon_i}\right)
+\sum_{i=1}^{m}\nu_i\,\bar u_i\Biggr] {\rm d}t\nonumber\\
&= V(x(0)) +\int_0^\infty\sum_{i=1}^{m}\Biggl[
\ell\mu_i(\varepsilon_i|\nu_i|)+\mu_i\!\left(\frac{|\bar u_i|}{\varepsilon_i}\right)+\nu_i\,\bar u_i
\Biggr] {\rm d}t.\label{eq:mi_J_proof}
\end{align}
By the generalized Young inequality associated with the Legendre--Fenchel transform~\cite[Thm.~156]{hardy_inequalities_1989}, we get, for each $i\in\{1,\ldots,m\}$,
\begin{align}
\left(\frac{\bar u_i}{\varepsilon_i}\right)\left(-\varepsilon_i\nu_i\right)
\le \mu_i\!\left(\frac{|\bar u_i|}{\varepsilon_i}\right)+\ell\mu_i(\varepsilon_i|\nu_i|)\,.
\end{align}
Multiplying by $-1$ yields
\begin{align}\label{eq:mi_young}
\nu_i\,\bar u_i
\ge -\mu_i\!\left(\frac{|\bar u_i|}{\varepsilon_i}\right)-\ell\mu_i(\varepsilon_i|\nu_i|)\,.
\end{align}
Therefore, \eqref{eq:mi_J_proof} achieves its minimum when
\eqref{eq:mi_young} holds with equality for each $i$. The equality condition in \cite[Thm.~156]{hardy_inequalities_1989} implies (for each $i$)
\begin{align}
a_i = (\mu_i')^{-1}(|b_i|)\frac{b_i}{|b_i|},
\end{align}
with $a_i\coloneqq \bar u_i/\varepsilon_i$ and $b_i = -\varepsilon_i\nu_i$, which yields
\begin{align}
\bar u_i = -\varepsilon_i(\mu_i')^{-1}(\varepsilon_i|\nu_i|)\,\sgn(\nu_i)=\bar u_i^*\,,
\end{align}
where $ i\in\{1,\ldots,m\}$. Thus, \eqref{eq:mi_young} are equalities if and only if $\bar u_i=\bar u_i^*$ for all $i$. Substituting the equalities into
\eqref{eq:mi_J_proof} gives $J=V(x(0))$, proving that \eqref{eq:mi_opt} minimizes \eqref{eq:mi_J}.
\end{proof}

\subsection{Stability margins}

A key practical benefit of inverse optimality is robustness to actuator uncertainty, namely, the mismatch between the commanded inputs $\bar u^* = [\bar u_1^*,\ldots,\bar u_m^*]^\top$ of~\eqref{eq:mi_opt} and the inputs the system actually receives. Because~\eqref{eq:mi_sys} is driftless, the gain margin is infinite (see~\cite{Part2_kimIOC2025}). Indeed, replacing each $\bar u_i^*$ by $\kappa_i\bar u_i^*$ with any constant $\kappa_i \in (0,\infty)$ scales each term of $\dot V = \sum_{i=1}^{m}\nu_i\,\bar u_i$ by $\kappa_i > 0$ and thus leaves $\dot V$ negative definite, whereas optimal controllers for systems with drift have a guaranteed gain margin of only $[\tfrac{1}{2},\infty)$~\cite{sepulchre1997constructive}. Likewise, GAS on $\mathcal{X}$ is preserved when each \rev{input} is distorted by any locally Lipschitz static actuator nonlinearity $\varphi_i(\bar u_i^*)$ that preserves the sign of the commanded input (i.e., $s\varphi_i(s)>0$ for $s\neq 0$), a sector margin in the sense of~\cite{sepulchre1997constructive}. The classical phase margin, a frequency-domain notion, does not extend to nonlinear systems~\cite{sepulchre1997constructive}. However, since $\dot V = \sum_{i=1}^{m}\nu_i\,\bar u_i$, the system~\eqref{eq:mi_sys} is \emph{lossless} (i.e., $V$ acts as an energy altered only through the inputs). Consequently, the $L_{g_i}V$-proportional members of~\eqref{eq:mi_opt} remain stabilizing when the plant receives $a(I+\mathcal{P})\bar u$ in place of $\bar u$, that is, the commanded input plus a dynamic distortion $\mathcal{P}\bar u$, with $\mathcal{P}$ being a strictly passive system in the sense of~\cite{byrnes1991passivity}, scaled by an unknown constant gain $a > 0$. Any $a>0$ is admissible, in contrast to the requirement $a\geq\tfrac{1}{2}$ for systems with drift~\cite[Prop.~2]{krstic_rigidspacecraft_1999}, \cite{sepulchre1997constructive}. Since the controllers of Sections~\ref{sec:ioc_rollfree} and~\ref{sec:ioc_roll_aug} are instances of~\eqref{eq:mi_opt} and~\eqref{eq:mi_damp}, they inherit these margins in the inputs~$\bar u_i$.

%%%%%%%%%%%%%%%%%%%%%%%%%%%%%%%%%%%%%%%%%%%%%%%%%%%%%%
\section{Inverse Optimal Design for \rev{Roll-Free} System}\label{sec:ioc_rollfree}

In this section, we derive a family of inverse optimal controllers for the system~\eqref{eq:spheric_3d_sys}  based on the strict CLF~\eqref{eq:V_def} established in Theorem~\ref{thm:GES_main}.

\subsection{Roll-free bidirectional IOC}

\begin{theorem}[\rev{Inverse optimal feedback renders roll-free vehicle GAS}]\label{thm:IOC_3d}
Consider the system~\eqref{eq:spheric_3d_sys} on $\mathcal{D}$.
Define $\Theta\coloneqq[\rho,\delta,\gamma,\zeta,\theta]^\top$ and rewrite \eqref{eq:spheric_3d_sys} as
\begin{align}\label{eq:ioc_sys_3d}
\dot{\Theta} = \bar{g}_1(\Theta)\frac{v}{\rho} + g_2(\Theta)\,q + \bar g_3(\Theta)\frac{r}{\cos\theta}
\end{align}
with
\begin{align}
\bar{g}_1(\Theta) &\coloneqq
\begin{bmatrix}
\rho\bigl(\sin\theta\,\sin\zeta-\cos\theta\,\cos\zeta\,\cos\gamma\bigr)\\
\dfrac{\cos\theta}{\cos\zeta}\sin\gamma\\
\dfrac{\cos\theta}{\cos\zeta}\sin\gamma\\
\cos\theta\,\sin\zeta\,\cos\gamma+\sin\theta\,\cos\zeta\\
0
\end{bmatrix},
\end{align}
and
\begin{align}
g_2(\Theta)&\coloneqq
\begin{bmatrix}
0 & 0 & 0 & 0 & 1
\end{bmatrix}^{\!\top}\\
\bar g_3(\Theta)&\coloneqq
\begin{bmatrix}
0 & 0 & -1 & 0 & 0
\end{bmatrix}^{\!\top}\,.
\end{align}
Let $\mu_i\in\mathcal{K}_\infty[0,w_i)$, where $w_i > 0$ is finite or infinite and $i\in\{1,2,3\}$, and let $\varepsilon_i(\Theta)$ be scalar-valued functions\rev{, continuous and positive on $\Sigma$}. Then, the cost functional
\begin{equation}\label{eq:J_thrm_ioc_bidir_no_roll}
J = \int_0^\infty \Biggl[l(\Theta) + \mu_1\!\left(\frac{|v|}{\varepsilon_1\rho}\right)
+ \mu_2\!\left(\frac{|q|}{\varepsilon_2}\right)
+ \rev{\mu_3\!\left(\frac{|\tilde r|}{\varepsilon_3}\right)}\Biggr] {\rm d}t\,,
\end{equation}
where
\begin{align}\label{eq:run_cost_ioc_bidir_no_roll}
l(\Theta) \coloneqq \ell\mu_1(\varepsilon_1|\nu_1|)+\ell\mu_2(\varepsilon_2|\nu_2|)+\ell\mu_3(\varepsilon_3|\nu_3|)\,,
\end{align}
and $\nu_1(\Theta)\coloneqq L_{\bar{g}_1}V(\Theta)$, $\nu_2(\Theta)\coloneqq L_{g_2}V(\Theta)$, and $\nu_3(\Theta)\coloneqq L_{\bar g_3}V(\Theta)$ denoted for the strict CLF~\eqref{eq:V_def},
with $e_1,e_2$ defined in~\eqref{eq:e1_def}, \eqref{eq:e2_def} and $k_2,k_3>0$, \rev{and $\tilde r$ as in~\eqref{eq:yawrate_actual},}
and for all initial conditions on $\mathcal{D}$, is minimized\rev{, within the class of controls under which the solution of~\eqref{eq:ioc_sys_3d} exists on $[t_0,\infty)$ and $\Theta(t) \to 0$,} by the feedback law
\begin{subequations}\label{eq:ioc_ctrl_opt_3d}
\begin{align}
v^* &= -\rho\,\varepsilon_1 (\mu_1')^{-1}(\varepsilon_1|\nu_1|)\,\sgn(\nu_1)\\
q^* &= -\varepsilon_2 (\mu_2')^{-1}(\varepsilon_2|\nu_2|)\,\sgn(\nu_2)\\
r^* &= -\cos(\theta)\varepsilon_3 (\mu_3')^{-1}(\varepsilon_3|\nu_3|)\,\sgn(\nu_3)\,.
\end{align}
\end{subequations}
\rev{The feedback~\eqref{eq:ioc_ctrl_opt_3d} renders $\Theta = (0,0,0,0,0)$ GAS on $\mathcal{D}$ in accordance with Definition~\ref{def-our-GAS}.}
Additionally, the continuous feedback
\begin{subequations}\label{eq:ioc_ctrl_3d}
\begin{align}
v &= -\rho\,\varepsilon_1 \frac{\ell\mu_1(\varepsilon_1|\nu_1|)}{\varepsilon_1|\nu_1|}\,\sgn(\nu_1)\\
q &= -\varepsilon_2 \frac{\ell\mu_2(\varepsilon_2|\nu_2|)}{\varepsilon_2|\nu_2|}\,\sgn(\nu_2)\\
r &= -\cos(\theta)\varepsilon_3 \frac{\ell\mu_3(\varepsilon_3|\nu_3|)}{\varepsilon_3|\nu_3|}\,\sgn(\nu_3)\,,
\end{align}
\end{subequations}
is continuous in $(\nu_1,\nu_2,\nu_3)$ and renders $\Theta= (0,0,0,0,0)$ GAS on $\mathcal{D}$ in accordance with Definition~\ref{def-our-GAS}.
\end{theorem}

\begin{proof}
The system~\eqref{eq:ioc_sys_3d} is of the form~\eqref{eq:mi_sys} with $m = 3$, $x = \Theta$, $(\mathcal{X}, |\cdot|_{\mathcal{X}}) = (\rev{\Sigma}, |\cdot|_{\mathcal{D}})$\rev{, on which the vector fields in \eqref{eq:ioc_sys_3d} extend continuously and the sublevel sets of $V$ are compact by the quadratic bounds in the proof of Theorem~\ref{thm:GES_main}}, $(g_1, g_2, g_3) = (\bar g_1, g_2, \bar g_3)$, and the inputs
\begin{align}\label{eq:u_bar_3d}
(\bar u_1,\bar u_2,\bar u_3)
= \left(\frac{v}{\rho},\ q,\ \frac{r}{\cos\theta}\right)\,,
\end{align}
under which \eqref{eq:mi_J}, \eqref{eq:mi_l} coincide with \eqref{eq:J_thrm_ioc_bidir_no_roll}, \eqref{eq:run_cost_ioc_bidir_no_roll}, and \eqref{eq:mi_opt} and \eqref{eq:mi_damp} with \eqref{eq:ioc_ctrl_opt_3d} and \eqref{eq:ioc_ctrl_3d}. Since $V$ is a strict CLF by Theorem~\ref{thm:GES_main}, conditions (i) and (ii) of Lemma~\ref{lem:mi_ioc} hold by Definition~\ref{def-CLF}. Hence, by Lemma~\ref{lem:mi_ioc}, \eqref{eq:ioc_ctrl_opt_3d} minimizes \eqref{eq:J_thrm_ioc_bidir_no_roll} \rev{and renders $\Theta = (0,0,0,0,0)$ GAS on $\mathcal{D}$}, and \eqref{eq:ioc_ctrl_3d} is continuous in $(\nu_1,\nu_2,\nu_3)$ and renders $\Theta = (0,0,0,0,0)$ GAS on $\mathcal{D}$ in accordance with Definition~\ref{def-our-GAS}\rev{, where solutions from $\mathcal{D}$ remain in $\mathcal{D}$ since $\dot\rho = \rho F(\Theta)$ with $F$ bounded on each sublevel set of $V$, so that $\rho(t) > 0$ at every finite time}.
\end{proof}

\begin{revblock}
The continuity and positivity of the scalings on $\Sigma$, rather than on $\mathcal{D}$ alone, excludes scalings that grow unbounded as $\rho \to 0$ and ensures forward completeness of the closed loop. For instance, on the ray $\delta=\gamma=\zeta=\theta=0$, where $\nu_1 = -k_1\rho^2$, the choice $\varepsilon_1 = \rho^{-2}$, continuous and positive on $\mathcal{D}$, with the quadratic penalty of Example~\ref{ex:quad}, gives $v^* = k_1/\rho$ and $\dot\rho = -k_1/\rho$, which reaches the excluded boundary $\rho = 0$ in finite time with unbounded surge. The scalings~\eqref{eq:eps_quad}, \eqref{eq:eps_quad_uni}, \eqref{eq:eps_ln_bi}, and~\eqref{eq:eps_ln_uni} are continuous and positive on $\Sigma$ and $\Sigma_{\mathrm{a}}$.
\end{revblock}

\subsection{Roll-free unidirectional IOC}

\rev{On platforms such as fixed-wing aircraft and many AUVs, surge reversal is infeasible, and the bidirectional result of Theorem~\ref{thm:IOC_3d} is therefore unsuitable.} Accordingly, in this section, we seek an $L_gV$-based controller that guarantees $v(t) \ge 0$ for all $t$, while retaining the benefits of inverse optimality, infinite gain margins, and user-defined input bounds from certain Legendre--Fenchel transformations. 

Consider the control law
\begin{subequations}
    \begin{align}
    v & = -\rho\min\{0,\nu_1(\Theta)\}\label{eq:v_uni_IOC_simple}\\
    q &= -\nu_2(\Theta)\\
    r &= -\cos(\theta)\nu_3(\Theta)\,,
    \end{align}
\end{subequations}
with $\Theta$ and $\nu_i(\Theta)$ defined as in Theorem~\ref{thm:IOC_3d}. For any strict CLF $V(\Theta)$, we obtain
\begin{align}
    \dot V = -\nu_1(\Theta)\min\{0,\nu_1(\Theta)\} - \nu_2^2(\Theta) - \nu_3^2(\Theta)\,.
\end{align}
Since $(\nu_1,\nu_2,\nu_3)=(0,0,0)$ only at the origin, it suffices to show that $(\nu_2,\nu_3)\neq(0,0)$ whenever $\nu_1\rev{\ge}0$ in order to guarantee that $\dot V$ is negative definite. This property is established for the CLF~\eqref{eq:V_def} in the following lemma.

\begin{figure*}[!t]
\centering
\begin{minipage}[c]{0.28\textwidth}\centering\includegraphics[width=\linewidth]{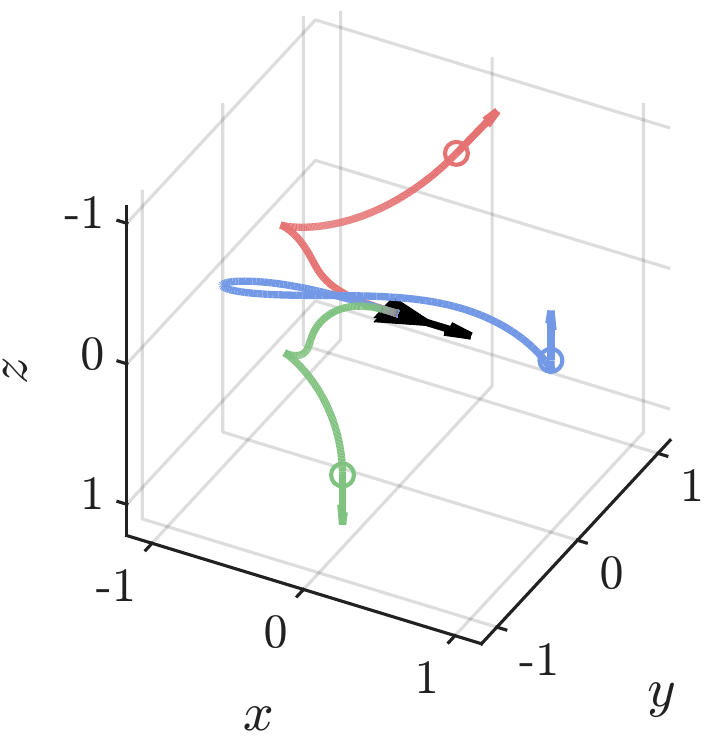}\end{minipage}\hfill%
\begin{minipage}[c]{0.20\textwidth}\centering\includegraphics[width=\linewidth]{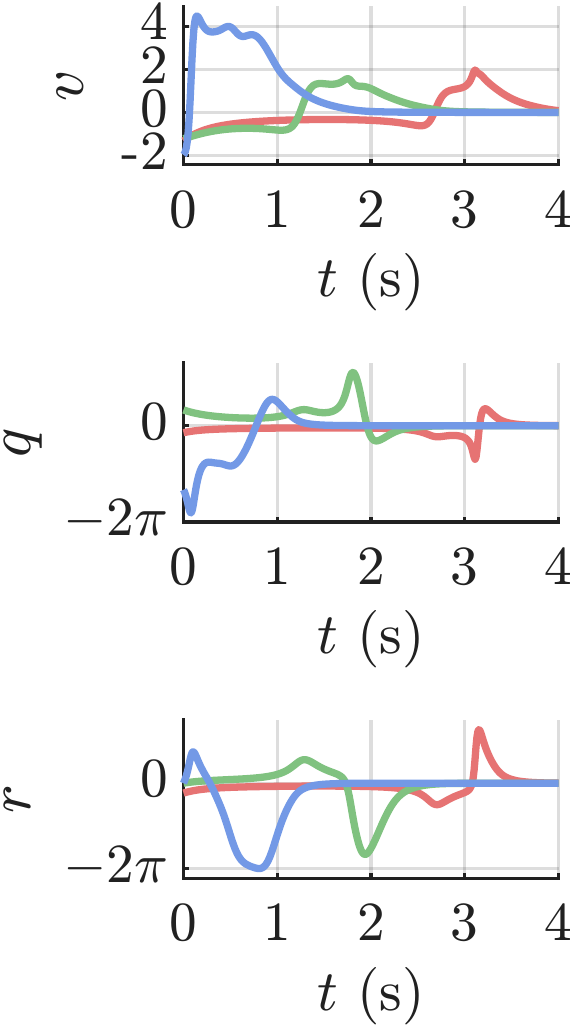}\end{minipage}\hfill%
\begin{minipage}[c]{0.28\textwidth}\centering\includegraphics[width=\linewidth]{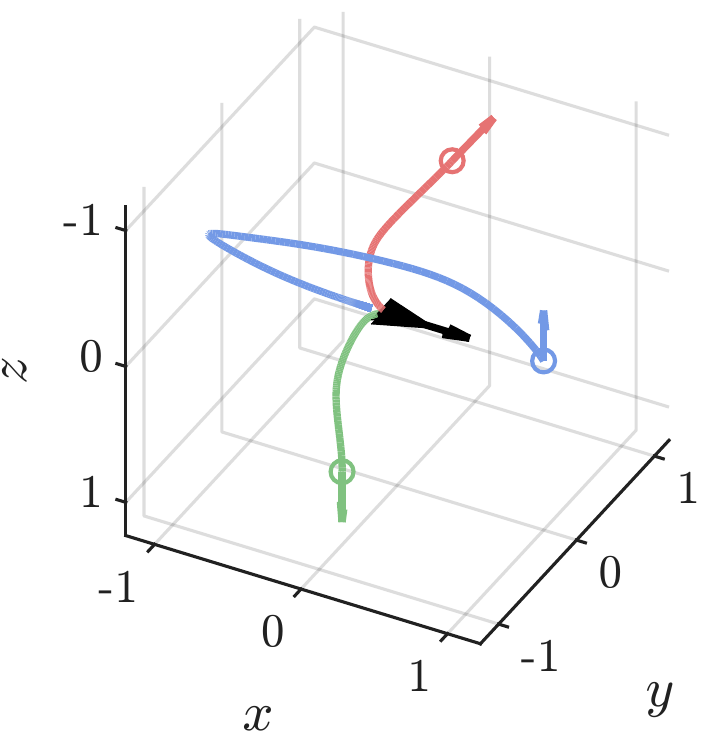}\end{minipage}\hfill%
\begin{minipage}[c]{0.20\textwidth}\centering\includegraphics[width=\linewidth]{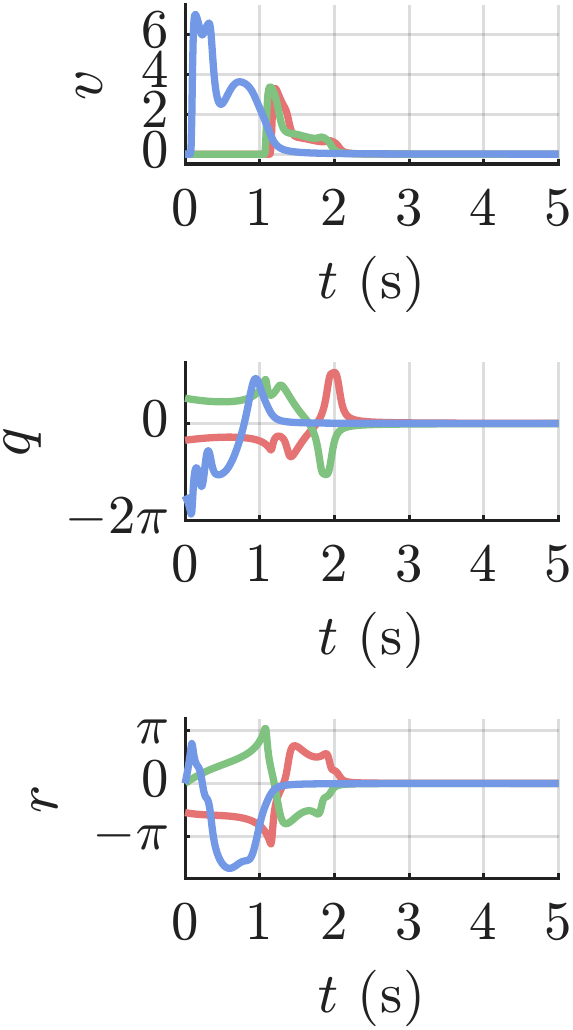}\end{minipage}
\caption{Closed-loop 3D trajectories of Example~\ref{ex:quad} (NED convention) and optimal inputs of the roll-free system~\eqref{eq:spheric_3d_sys} under the bidirectional law~\eqref{eq:ex_quadratic_roll_free_bi} (left) and the unidirectional law~\eqref{eq:ex_quadratic_roll_free_uni} (right) from $\Theta_1(0) = (1,-\pi/2,-3\pi/4,\pi/4,\pi/4)$, $\Theta_2(0) = (1,\pi/2,\pi/2,-\pi/4,-\pi/2)$, and $\Theta_3(0) = (1,\pi,0,0,\pi/2)$ (red, green, blue), with the CLF gains $(k_1,k_2,k_3,k_4,k_5) = (0.6,2,2,2,2)$, $Q = I_{5\times 5}$, and $\hat\varepsilon = 10^{-6}$ (bidirectional) and $10^{-4}$ (unidirectional). The black dart and arrow mark the target at the origin while the circles and arrows mark the initial position and heading. In the bidirectional case convergence is fast and the vehicle reverses surge mid-trajectory. In contrast, in the unidirectional case the surge satisfies $v(t) \geq 0$ and the vehicle turns in place when $\nu_1 \geq 0$.}
\label{fig:sim_quad_traj}
\vspace{-0.5cm}
\end{figure*}

\begin{lemma}[Strictly negative surge Lie derivative at zero attitude errors]
\label{lem:unidir_nu1}
Let $V$ be the strict CLF \eqref{eq:V_def}. Then, for all $\Theta\in\Sigma\setminus\{0\}$, \rev{$\bigl(\nu_2(\Theta),\nu_3(\Theta)\bigr) = (0,0) \Longrightarrow \nu_1(\Theta) < 0$ and, equivalently, $\nu_1(\Theta) \ge 0 \Longrightarrow \bigl(\nu_2(\Theta),\nu_3(\Theta)\bigr) \neq (0,0)$}.
% \begin{align}
% \nu_1(\Theta)> 0 \quad\Longrightarrow\quad 
% \bigl(\nu_2(\Theta),\nu_3(\Theta)\bigr) \neq (0,0).
% \label{eq:nu_separation}
% \end{align}
% In particular, $(\nu_1(\Theta),\nu_2(\Theta),\nu_3(\Theta))=(0,0,0)$ if and only if $\Theta=0$.
\end{lemma}

\begin{proof}
From \eqref{eq:e1_def} and \eqref{eq:e2_def},
$\partial e_1/\partial\gamma = \partial e_2/\partial\theta = 1$, hence
$\nu_2(\Theta) = k_5e_2$ and $\nu_3(\Theta) = -k_4 e_1$, so
$(\nu_2(\Theta),\nu_3(\Theta)) = (0,0)$ if and only if
$e_1 = e_2 = 0$. On this set, $\dot V = \nu_1\,v/\rho$ for any inputs
$(q,\tilde r)$. Evaluating along the feedback of
Theorem~\ref{thm:GES_main}, where $v/\rho = k_1NB$
by~\eqref{eq:v_ctrl}, and comparing with~\eqref{eq:Vdot} at
$e_1 = e_2 = 0$ yields
\begin{align}
\nu_1\big|_{e_1=e_2=0}
=
-\frac{1}{BN}\biggl[
&k_1\rho^2\left(k_3\sin^2\zeta+\frac{1}{\cos\zeta}\right)\nonumber\\
&+k_2^2\delta^2
+k_3^2\tan^2\zeta
\biggr]\,,
\label{eq:nu1_final}
\end{align}
which is strictly negative for all $\Theta\in\Sigma\setminus\{0\}$, since the bracket vanishes only at $\rho=\delta=\zeta=0$, which together with $e_1=e_2=0$ is the origin. Thus, \rev{$(\nu_2(\Theta),\nu_3(\Theta)) = (0,0)$ implies $\nu_1(\Theta) < 0$ and, equivalently,
$\nu_1(\Theta) \ge 0 \implies (\nu_2(\Theta),\nu_3(\Theta))\neq(0,0)$}.
\end{proof}

Additionally, the forward velocity control~\eqref{eq:v_uni_IOC_simple} is, while continuous, not differentiable at $\nu_1=0$. To obtain unidirectional controllers with tunable differentiability, we introduce the function $h:\mathbb{R}\to\mathbb{R}_{\ge0}$ defined by
\begin{align}\label{eq:h_def}
h(s)=
\begin{cases}
0, & s\ge 0,\\
(-s)^n, & s<0,
\end{cases}
\end{align}
for $n\in\mathbb{Z}^+$ where $h$ is $C^{n-1}$\rev{; the surge law in Corollary~\ref{cor:unidir_IOC} below inherits this differentiability wherever $\tilde{\varepsilon}_1$ and $(\mu_1')^{-1}$ are correspondingly differentiable}.

\rev{Kinematic laws of this type are commonly used as virtual references for an underlying dynamic (force and moment) control layer, and each backstepping step through the vehicle dynamics differentiates the surge command once, so $n$ sets the number of derivatives available for that cascade.} 

\rev{Setting $\varepsilon_1(\Theta)=\tilde{\varepsilon}_1(\Theta)h(\nu_1)$ in~\eqref{eq:ioc_ctrl_opt_3d} yields a unidirectional controller. This scaling vanishes on the set $\{\nu_1 \ge 0\}$, where Theorem~\ref{thm:IOC_3d} does not apply, which motivates the following corollary, in which Lemma~\ref{lem:unidir_nu1} guarantees strict decrease of the CLF through the attitude inputs wherever the surge is disabled.}

\begin{figure*}[!t]
\centering
\begin{minipage}[c]{0.48\textwidth}\centering\includegraphics[width=.95\linewidth]{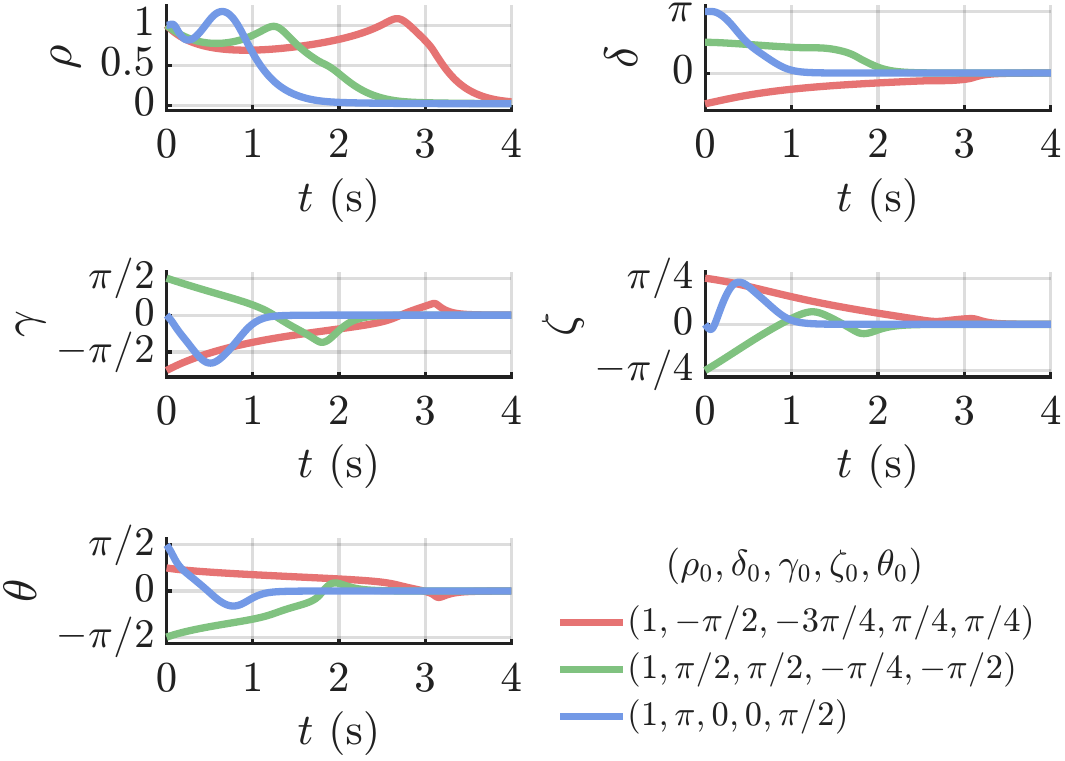}\end{minipage}\hfill%
\begin{minipage}[c]{0.48\textwidth}\centering\includegraphics[width=.95\linewidth]{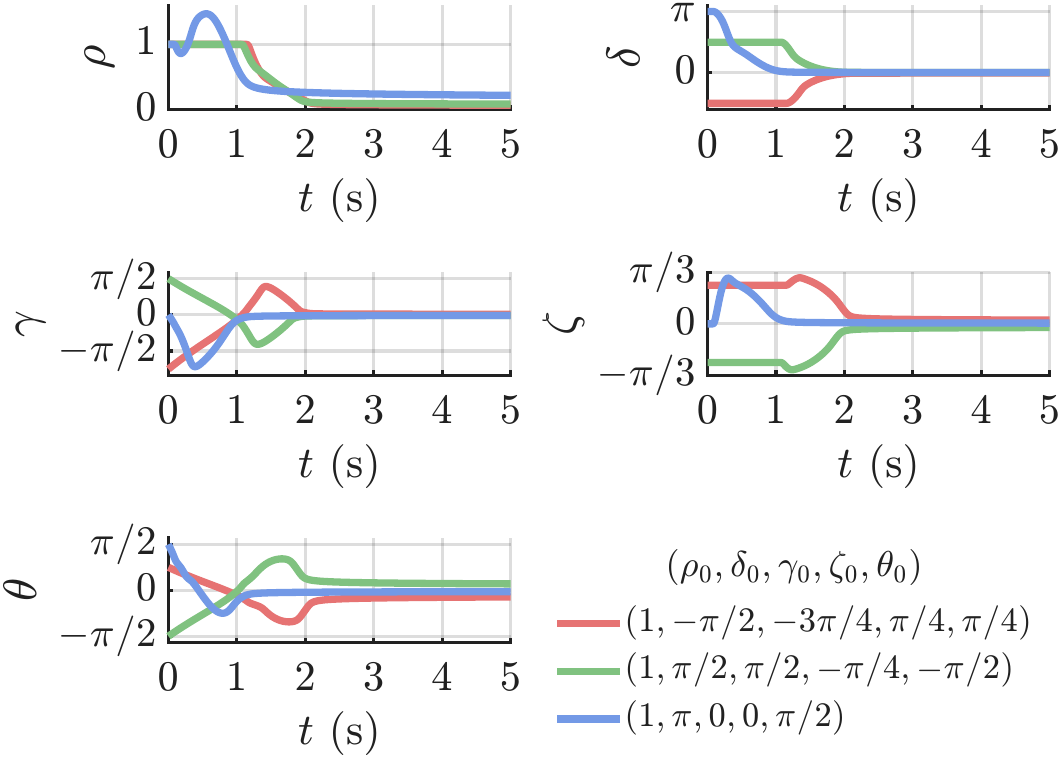}\end{minipage}
\caption{Example~\ref{ex:quad}, spherical states of the trajectories in Fig.~\ref{fig:sim_quad_traj}: bidirectional surge (quadratic penalty~\eqref{eq:mu_quad} with~\eqref{eq:eps_quad}, left) and unidirectional surge (with~\eqref{eq:eps_quad_uni}, $n = 1$, right). In the case of the unidirectional surge, the decay of $\rho$ is polynomial, resulting in slower convergence than in the bidirectional case.}
\label{fig:sim_quad_states}
\vspace{-0.5cm}
\end{figure*}

\begin{corollary}[\rev{Unidirectional surge preserves GAS and inverse optimality}]\label{cor:unidir_IOC}
Consider the system~\eqref{eq:ioc_sys_3d} on $\mathcal{D}$. Let $\tilde{\varepsilon}_1(\Theta)$, $\varepsilon_2(\Theta)$, and $\varepsilon_3(\Theta)$ be scalar-valued functions\rev{, continuous and positive on $\Sigma$}. Then, for the strict CLF~\eqref{eq:V_def}, the feedback law
\begin{subequations}\label{eq:unidir_IOC_ctrl_law}
    \begin{align}
        v^* 
        % &= -\rho\tilde{\varepsilon}_1h(\nu_1)(\mu_1')^{-1}\bigl(\tilde{\varepsilon}_1h(\nu_1)|\nu_1|\bigr)\,\sgn(\nu_1)\nonumber\\
        &= \rho\tilde{\varepsilon}_1h(\nu_1)(\mu_1')^{-1}\bigl(\tilde{\varepsilon}_1h(\nu_1)|\nu_1|\bigr)\\
        q^* &= -\varepsilon_2 (\mu_2')^{-1}(\varepsilon_2|\nu_2|)\,\sgn(\nu_2)\\
        r^* &= -\cos(\theta)\varepsilon_3 (\mu_3')^{-1}(\varepsilon_3|\nu_3|)\,\sgn(\nu_3)\,,
    \end{align}
\end{subequations}
where $h(s)$ is defined in~\eqref{eq:h_def}, renders $\Theta= (0,0,0,0,0)$ GAS on $\mathcal{D}$ in accordance with Definition~\ref{def-our-GAS}. Additionally, \rev{within the class of controls under which the solution of~\eqref{eq:ioc_sys_3d} exists on $[t_0,\infty)$ and $\Theta(t) \to 0$,} the control law~\eqref{eq:unidir_IOC_ctrl_law} minimizes the cost functional
\begin{align}\label{eq:J_cor_ioc_uni_no_roll}
J = \int_0^\infty \Biggl[&l(\Theta) + \mu_1\!\left(\frac{|v|}{\tilde{\varepsilon}_1h(\nu_1)\rho}\right)
+ \mu_2\!\left(\frac{|q|}{\varepsilon_2}\right)+ \rev{\mu_3\!\left(\frac{|\tilde r|}{\varepsilon_3}\right)}\Biggr] {\rm d}t\,,
\end{align}
where
\begin{align}\label{eq:uni_ioc_l_statepenalty}
l(\Theta) \coloneqq &\ell\mu_1(\tilde{\varepsilon}_1h(\nu_1)|\nu_1|)+\ell\mu_2(\varepsilon_2|\nu_2|)+\ell\mu_3(\varepsilon_3|\nu_3|)\,,
\end{align}
\rev{with $\tilde r$ as in~\eqref{eq:yawrate_actual} and, on the set $\{\tilde{\varepsilon}_1 h(\nu_1) = 0\}$, the surge penalty in~\eqref{eq:J_cor_ioc_uni_no_roll} defined by its extended-real limit, in line with the infinite control weight in~\cite{krstic2023inverse},
\begin{align}\label{eq:zero_scale_conv}
\mu_1\!\left(\frac{|v|}{\tilde{\varepsilon}_1 h(\nu_1)\rho}\right) \coloneqq
\begin{cases}
0, & v = 0\,,\\
+\infty, & v \neq 0\,.
\end{cases}
\end{align}}
\end{corollary}

\begin{revblock}
\begin{proof}
Let $a_1(\Theta) \coloneqq \tilde{\varepsilon}_1(\Theta)h(\nu_1(\Theta)) \ge 0$. On $\{\nu_1 < 0\}$, where $a_1 > 0$, the law~\eqref{eq:unidir_IOC_ctrl_law} coincides with~\eqref{eq:ioc_ctrl_opt_3d} for $\varepsilon_1 = a_1$ and gives $v^* > 0$; on $\{\nu_1 \ge 0\}$, $v^* = 0$. Hence $v^*(t) \ge 0$ for all $t$ and, since $(\mu_1')^{-1}\bigl(a_1|\nu_1|\bigr) \to 0$ as $\nu_1 \to 0^-$ by~\eqref{eq:mi_lhopital}, the feedback~\eqref{eq:unidir_IOC_ctrl_law} is continuous on $\mathcal{D}$.

For stability, as in~\eqref{eq:mi_Vdot_opt},
\begin{align}\label{eq:Vdot_unidir}
\dot V\big|_{\eqref{eq:unidir_IOC_ctrl_law}}
=& -a_1|\nu_1|(\mu_1')^{-1}(a_1|\nu_1|)\nonumber\\
&- \sum_{i=2}^{3}\varepsilon_i|\nu_i|(\mu_i')^{-1}(\varepsilon_i|\nu_i|) \le 0\,.
\end{align}
If $\dot V(\Theta) = 0$ for some $\Theta \in \Sigma\setminus\{0\}$, then $\nu_2 = \nu_3 = 0$ since $\varepsilon_2, \varepsilon_3 > 0$, and Lemma~\ref{lem:unidir_nu1} gives $\nu_1 < 0$, hence $a_1|\nu_1| > 0$ and the first term of~\eqref{eq:Vdot_unidir} is negative, a contradiction. Thus $\dot V < 0$ on $\Sigma\setminus\{0\}$ and, since the scalings are continuous and positive on $\Sigma$, the closed loop is forward complete, and $\Theta = (0,0,0,0,0)$ is GAS on $\mathcal{D}$ as in Step~1 of Lemma~\ref{lem:mi_ioc}, applied on $\Sigma$ as in the proof of Theorem~\ref{thm:IOC_3d}.

For optimality, let $(\bar u_1, \bar u_2, \bar u_3)$ in~\eqref{eq:u_bar_3d} be any control in the admissible class. Where $a_1 > 0$, \eqref{eq:mi_young} holds for $i = 1$ with $\varepsilon_1 = a_1$, with equality if and only if $\bar u_1 = v^*/\rho$. Where $a_1 = 0$, the bracketed term of~\eqref{eq:mi_J_proof} for $i = 1$ equals, by~\eqref{eq:zero_scale_conv}, $+\infty$ for $\bar u_1 \neq 0$ and $0$ for $\bar u_1 = 0 = v^*/\rho$, so it is nonnegative and vanishes if and only if $\bar u_1 = v^*/\rho$. For $i = 2,3$, equality in~\eqref{eq:mi_young} holds if and only if $\bar u_i = \bar u_i^*$, as in Step~3 of Lemma~\ref{lem:mi_ioc}. Repeating the computation~\eqref{eq:mi_J_proof}, with $\lim_{t\to\infty}V(\Theta(t)) = 0$ holding for every control in the admissible class, yields $J \ge V(\Theta(0))$, with equality under~\eqref{eq:unidir_IOC_ctrl_law}. In particular, under~\eqref{eq:unidir_IOC_ctrl_law}, the surge terms of the integrand vanish on $\{\nu_1 \ge 0\}$, so the optimal cost $V(\Theta(0))$ is finite.
\end{proof}
\end{revblock}

Since the guarantees of Corollary~\ref{cor:unidir_IOC} hold for every $n \in \mathbb{Z}^+$, \rev{the regularity affects neither stability nor optimality. Larger $n$, however, attenuates the surge command near $\nu_1 = 0$ and slows the terminal approach, as quantified in Example~\ref{ex:quad}; the inverse optimal designs of this section and of Section~\ref{sec:ioc_roll_aug} guarantee GAS, in contrast to the GES of Theorem~\ref{thm:GES_main}.} 

Additionally, in view of the cost functional~\eqref{eq:J_cor_ioc_uni_no_roll}, we examine the role of the unidirectionality-enforcing function $h(\nu_1)$. Since $h(\nu_1)=0$ whenever $\nu_1\ge 0$, the term $\mu_1\left(\frac{|v|}{\tilde{\varepsilon}_1 h(\nu_1)\rho}\right)$ \rev{equals $+\infty$ by~\eqref{eq:zero_scale_conv} for any $v \neq 0$}. This is in fact desirable, as it results in the optimal policy to assign infinite cost to \rev{all nonzero surge commands} for $\nu_1\ge 0$, thereby enforcing the constraint $v(t)\ge 0$ for all $t$. \rev{The zero penalty assigned to $v = 0$ in~\eqref{eq:zero_scale_conv} keeps the optimal cost finite, i.e., the closed loop under~\eqref{eq:unidir_IOC_ctrl_law} spends intervals of positive duration on $\{\nu_1 \ge 0\}$ while the vehicle reorients.}

\rev{This enforcement reduces the state penalty} associated with $\nu_1$ in~\eqref{eq:uni_ioc_l_statepenalty} relative to the bidirectional formulation of Theorem~\ref{thm:IOC_3d}. In particular, when $\nu_1>0$, the term $\ell\mu_1(\tilde{\varepsilon}_1 h(\nu_1)|\nu_1|)$ becomes zero, resulting in a smaller contribution to the running cost away from the origin. Nonetheless, Lemma~\ref{lem:unidir_nu1} guarantees that the overall state penalty $l(\Theta)$ when $\nu_1 \rev{\ge} 0$ is nonzero away from the origin. This asymmetry reflects the tradeoff inherent in enforcing forward-only motion while retaining inverse optimality and global asymptotic stability.

\subsection{Near classical direct optimal control}

% \begin{figure*}[!t]
% \centering
% \begin{minipage}[c]{0.48\textwidth}\centering\includegraphics[width=.95\linewidth]{Figs/jrnl_sims/ExV1_bi_states_sph.pdf}\end{minipage}\hfill%
% \begin{minipage}[c]{0.48\textwidth}\centering\includegraphics[width=.95\linewidth]{Figs/jrnl_sims/ExV1_bi_states_cart.pdf}\end{minipage}
% \caption{Example~\ref{ex:quad}, bidirectional surge (quadratic penalty~\eqref{eq:mu_quad} with~\eqref{eq:eps_quad}): spherical states (left) and Cartesian/Euler states (right) of the trajectories in Fig.~\ref{fig:sim_quad_traj}.}
% \label{fig:sim_quad_bi}
% \vspace{-0.2cm}
% \end{figure*}

% \begin{figure*}[!t]
% \centering
% \begin{minipage}[c]{0.48\textwidth}\centering\includegraphics[width=.95\linewidth]{Figs/jrnl_sims/ExV1_uni_states_sph.pdf}\end{minipage}\hfill%
% \begin{minipage}[c]{0.48\textwidth}\centering\includegraphics[width=.95\linewidth]{Figs/jrnl_sims/ExV1_uni_states_cart.pdf}\end{minipage}
% \caption{Example~\ref{ex:quad}, unidirectional surge (quadratic penalty~\eqref{eq:mu_quad} with~\eqref{eq:eps_quad_uni}, $n = 1$) with spherical states (left) and Cartesian/Euler states (right) of the trajectories in Fig.~\ref{fig:sim_quad_traj}.}
% \label{fig:sim_quad_uni}
% \vspace{-0.5cm}
% \end{figure*}

\begin{example}\label{ex:quad}
Consider the uniform quadratic penalty
\begin{align}\label{eq:mu_quad}
\mu_i(s) = \frac{s^2}{2}\,,
\end{align}
where $i \in \{1,2,3\}$, whose Legendre--Fenchel transform is $\ell\mu_i(s) = s^2/2$, since
$(\mu_i')^{-1}(s) = s$. The resulting optimal control law for the bidirectional surge from Theorem~\ref{thm:IOC_3d} is then
\vspace{-0.2cm}
\begin{subequations}\label{eq:ex_quadratic_roll_free_bi}
    \begin{align}
    v^* &= -\rho\varepsilon_1^2 \nu_1\\
    q^* &= -\varepsilon_2^2\nu_2\\
    r^* &= -\cos(\theta)\varepsilon_3^2\nu_3 \,,
    \end{align}
\end{subequations}
which minimizes the cost functional~\eqref{eq:J_thrm_ioc_bidir_no_roll} that is of the classical quadratic form, with the running cost
\begin{align}
    l(\Theta) = \frac{1}{2}\left(\varepsilon_1^2 \nu_1^2 + \varepsilon_2^2 \nu_2^2 + \varepsilon_3^2 \nu_3^2\right)\,.
\end{align}

Now, consider the choice
\begin{align}\label{eq:eps_quad}
\varepsilon_i = \sqrt{\frac{2\left(\Theta^\top Q\,\Theta + \hat\varepsilon\right)}{\nu_1^2 + \nu_2^2 + \nu_3^2 +\hat\varepsilon}}\,,
\end{align}
where $Q\succeq 0$ and $\hat\varepsilon > 0$. We observe that as $\hat\varepsilon \to 0$, the running cost approaches
\vspace{-0.15cm}
\begin{align}\label{eq:l_quad_limit}
l(\Theta) \to \Theta^\top Q\,\Theta\,,
\end{align}
recovering the classical
direct optimal control's quadratic in state running cost.

The analogous result for the unidirectional surge case in Corollary~\ref{cor:unidir_IOC}, using the quadratic penalty~\eqref{eq:mu_quad}, is given by the optimal control law
\begin{subequations}\label{eq:ex_quadratic_roll_free_uni}
    \begin{align}
    v^* &= -\rho\tilde\varepsilon_1^2 \rev{h^2(\nu_1)} \nu_1\\
    q^* &= -\varepsilon_2^2\nu_2\\
    r^* &= -\cos(\theta)\varepsilon_3^2\nu_3 \,,
    \end{align}
\end{subequations}
which minimizes the cost functional~\eqref{eq:J_cor_ioc_uni_no_roll} with the running cost
\begin{align}
    l(\Theta) = \frac{1}{2}\left(\tilde{\varepsilon}_1^2h^2(\nu_1) \nu_1^2 + \varepsilon_2^2 \nu_2^2 + \varepsilon_3^2 \nu_3^2\right)\,.
\end{align}
Similarly, choosing
\begin{align}\label{eq:eps_quad_uni}
\tilde\varepsilon_1 = \varepsilon_2 = \varepsilon_3 = \sqrt{\frac{2\left(\Theta^\top Q\,\Theta + \hat\varepsilon\right)}{h^2(\nu_1)\nu_1^2 + \nu_2^2 + \nu_3^2 +\hat\varepsilon}}\,,
\end{align}
then yields the same limit as~\eqref{eq:l_quad_limit}. For this example we choose $n = 1$ for $h(\nu_1)$. \rev{On the aligned ray $\delta=\gamma=\zeta=\theta=0$, where $\nu_1 = -k_1\rho^2$ by~\eqref{eq:nu1_final}, the unidirectional law~\eqref{eq:ex_quadratic_roll_free_uni} gives $\dot\rho = -\tilde{\varepsilon}_1^2 k_1^{2n+1}\rho^{4n+3}$, so the terminal approach is polynomial, of order $t^{-1/(4n+2)}$, in contrast to the exponential rate of Theorem~\ref{thm:GES_main}. The fast initial convergence in Figs.~\ref{fig:sim_quad_traj} and~\ref{fig:sim_quad_states} comes from the large early control effort, while the terminal phase visible in Fig.~\ref{fig:sim_quad_states} (right) is polynomial.}

Figs.~\ref{fig:sim_quad_traj} and~\ref{fig:sim_quad_states} show the resulting closed-loop responses with $Q = I_{5\times5}$ and $\hat\varepsilon = 10^{-6}$ (bidirectional) and $10^{-4}$ (unidirectional). The resulting control can be impractical for platforms with actuator limits. While parameters can be tuned to reduce the control effort, a guaranteed bound on the input is desirable, hence motivating the later Example~\ref{ex:lncos}, where the user prescribes explicit input bounds.
\end{example}

%%%%%%%%%%%%%%%%%%%%%%%%%%%%%%%%%%%%%%%%%%%%%%%%%%%%%%
\section{Inverse Optimal Design for \rev{Roll-Augmented} System}\label{sec:ioc_roll_aug}

While Lemma~\ref{lem:roll_ext} extends the stabilizing property of the controllers of Theorem~\ref{thm:IOC_3d} and Corollary~\ref{cor:unidir_IOC} to the roll-augmented kinematics~\eqref{eq:roll_sph}, it does not extend their optimality. The roll feedback~\eqref{eq:roll_law_p} is imposed rather than derived from a cost, and the roll rate $p$ is left unpenalized in~\eqref{eq:J_thrm_ioc_bidir_no_roll}. We therefore construct inverse optimal controllers directly for the four-input augmented system, based on the strict CLF~\eqref{eq:Va_explicit} established in Theorem~\ref{thm:roll_GES}. By~\eqref{eq:roll_reg}, the inputs satisfy $\bar q = \dot\theta$ and $\tilde{\bar r} = \dot\psi$, so the cost below penalizes the same quantities as~\eqref{eq:J_thrm_ioc_bidir_no_roll}, namely the surge, the pitch, and yaw rates, since $q = \dot\theta$ and $\tilde r = \dot\psi$ in~\eqref{eq:cart_3d_sys}, augmented by a penalty on the roll rate $p$. Moreover, since $M(\phi) \in SO(2)$, it holds that $q^2 + r^2 = \bar q^2 + \bar r^2$, so any bound enforced on the norm of the \rev{rotated} rates $(\bar q, \bar r)$ is inherited by the body rates, namely, enforcing $\bar q^2 + \bar r^2 \le u_{\max}^2$ guarantees the componentwise bounds $|q| \le u_{\max}$ and $|r| \le u_{\max}$.

\subsection{Roll-augmented bidirectional IOC}

\begin{theorem}[\rev{Inverse optimal feedback renders roll-augmented vehicle GAS}]\label{thm:IOC_roll}
Consider the system~\eqref{eq:roll_sph} on $\mathcal{D}_{\mathrm{a}}$
and, under \eqref{eq:channel_rates} and \eqref{eq:roll_sub}, rewrite
it as
\begin{align}\label{eq:ioc_sys_roll}
\dot{\Xi} = \bar{g}_1^{\mathrm{a}}(\Xi)\frac{v}{\rho} + g_p(\Xi)\,p
+ g_2^{\mathrm{a}}(\Xi)\,\bar q
+ \bar g_3^{\mathrm{a}}(\Xi)\frac{\bar r}{\cos\theta}\,,
\end{align}
where
\begin{align}\label{eq:g_roll_def}
\bar g_1^{\mathrm{a}} \coloneqq
\begin{bmatrix} \bar g_1\\ 0\end{bmatrix},\quad
g_2^{\mathrm{a}} \coloneqq
\begin{bmatrix} g_2\\ 0\end{bmatrix},\quad
\bar g_3^{\mathrm{a}} \coloneqq
\begin{bmatrix} \bar g_3\\ \sin\theta\end{bmatrix},
\end{align}
with $\bar g_1$, $g_2$, and $\bar g_3$ as in~\eqref{eq:ioc_sys_3d}, and
\begin{align}\label{eq:gp_def}
g_p \coloneqq
\begin{bmatrix} 0 & 0 & 0 & 0 & 0 & 1\end{bmatrix}^{\!\top}\,.
\end{align}
Let
$\mu_i\in\mathcal{K}_\infty[0,w_i)$, where $w_i > 0$ is finite or
infinite and $i\in\{1,p,2,3\}$, let $\varepsilon_i(\Xi)$ be scalar-valued functions\rev{, continuous and positive on $\Sigma_{\mathrm{a}}$}, and let
$\nu_1^{\mathrm{a}} \coloneqq L_{\bar g_1^{\mathrm{a}}}V_{\mathrm{a}}$,
$\nu_p^{\mathrm{a}} \coloneqq L_{g_p}V_{\mathrm{a}}$,
$\nu_2^{\mathrm{a}} \coloneqq L_{g_2^{\mathrm{a}}}V_{\mathrm{a}}$, and
$\nu_3^{\mathrm{a}} \coloneqq L_{\bar g_3^{\mathrm{a}}}V_{\mathrm{a}}$
denote the Lie derivatives of the strict CLF~\eqref{eq:Va_explicit}.
Then, \rev{with $\tilde{\bar r}$ as in~\eqref{eq:roll_sub},} the cost functional
\begin{align}\label{eq:J_ioc_roll}
J_{\mathrm{a}} = \int_0^\infty \Biggl[&l_{\mathrm{a}}(\Xi)
+ \mu_1\!\left(\frac{|v|}{\varepsilon_1\rho}\right)
+ \mu_p\!\left(\frac{|p|}{\varepsilon_p}\right)\nonumber\\
&+ \mu_2\!\left(\frac{|\bar q|}{\varepsilon_2}\right)
+ \rev{\mu_3\!\left(\frac{|\tilde{\bar r}|}{\varepsilon_3}\right)}
\Biggr] {\rm d}t\,,
\end{align}
where
\begin{align}\label{eq:run_cost_ioc_roll}
l_{\mathrm{a}}(\Xi) \coloneqq&\
\ell\mu_1(\varepsilon_1|\nu_1^{\mathrm{a}}|)
+\ell\mu_p(\varepsilon_p|\nu_p^{\mathrm{a}}|)\nonumber\\
&+\ell\mu_2(\varepsilon_2|\nu_2^{\mathrm{a}}|)
+\ell\mu_3(\varepsilon_3|\nu_3^{\mathrm{a}}|)\,,
\end{align}
and for all initial conditions on $\mathcal{D}_{\mathrm{a}}$, is
minimized\rev{, within the class of controls under which the solution of~\eqref{eq:ioc_sys_roll} exists on $[t_0,\infty)$ and $\Xi(t) \to 0$,} by the feedback law
\begin{subequations}\label{eq:ioc_ctrl_opt_roll}
\begin{align}
v^* &= -\rho\,\varepsilon_1
(\mu_1')^{-1}(\varepsilon_1|\nu_1^{\mathrm{a}}|)\,
\sgn(\nu_1^{\mathrm{a}})\label{eq:ioc_opt_roll_v}\\
p^* &= -\varepsilon_p
(\mu_p')^{-1}(\varepsilon_p|\nu_p^{\mathrm{a}}|)\,
\sgn(\nu_p^{\mathrm{a}})\label{eq:ioc_opt_roll_p}\\
\begin{bmatrix} q^*\\ r^*\end{bmatrix}
&= M(\phi)^{\!\top}
\begin{bmatrix}
-\varepsilon_2
(\mu_2')^{-1}(\varepsilon_2|\nu_2^{\mathrm{a}}|)\,
\sgn(\nu_2^{\mathrm{a}})\\[2pt]
-\cos(\theta)\,\varepsilon_3
(\mu_3')^{-1}(\varepsilon_3|\nu_3^{\mathrm{a}}|)\,
\sgn(\nu_3^{\mathrm{a}})
\end{bmatrix},\label{eq:ioc_opt_roll_qr}
\end{align}
\end{subequations}
where, by~\eqref{eq:channel_rates}, the bracketed entries
of~\eqref{eq:ioc_opt_roll_qr} are the minimizing \rev{rotated} rates
$(\bar q^*, \bar r^*)$. \rev{The feedback~\eqref{eq:ioc_ctrl_opt_roll} renders $\Xi = (0,\ldots,0)$ GAS on $\mathcal{D}_{\mathrm{a}}$ in accordance with Definition~\ref{def-our-GAS} as extended in Remark~\ref{rem:roll_def}.}
Additionally, the continuous feedback
\begin{subequations}\label{eq:ioc_ctrl_roll}
\begin{align}
v &= -\rho\,\varepsilon_1
\frac{\ell\mu_1(\varepsilon_1|\nu_1^{\mathrm{a}}|)}
{\varepsilon_1|\nu_1^{\mathrm{a}}|}\,
\sgn(\nu_1^{\mathrm{a}})\label{eq:ioc_roll_v}\\
p &= -\varepsilon_p
\frac{\ell\mu_p(\varepsilon_p|\nu_p^{\mathrm{a}}|)}
{\varepsilon_p|\nu_p^{\mathrm{a}}|}\,
\sgn(\nu_p^{\mathrm{a}})\label{eq:ioc_roll_p}\\
\begin{bmatrix} q\\ r\end{bmatrix}
&= M(\phi)^{\!\top}
\begin{bmatrix}
-\varepsilon_2
\dfrac{\ell\mu_2(\varepsilon_2|\nu_2^{\mathrm{a}}|)}
{\varepsilon_2|\nu_2^{\mathrm{a}}|}\,
\sgn(\nu_2^{\mathrm{a}})\\[6pt]
-\cos(\theta)\,\varepsilon_3
\dfrac{\ell\mu_3(\varepsilon_3|\nu_3^{\mathrm{a}}|)}
{\varepsilon_3|\nu_3^{\mathrm{a}}|}\,
\sgn(\nu_3^{\mathrm{a}})
\end{bmatrix},\label{eq:ioc_roll_qr}
\end{align}
\end{subequations}
is continuous in
$(\nu_1^{\mathrm{a}},\nu_p^{\mathrm{a}},
\nu_2^{\mathrm{a}},\nu_3^{\mathrm{a}})$ and
renders $\Xi= (0,\ldots,0)$ GAS on $\mathcal{D}_{\mathrm{a}}$ in accordance with
Definition~\ref{def-our-GAS} as extended in Remark~\ref{rem:roll_def}.
\end{theorem}

\begin{proof}
Under~\eqref{eq:channel_rates} and~\eqref{eq:roll_sub}, the system~\eqref{eq:ioc_sys_roll} is of the form~\eqref{eq:mi_sys} with $m = 4$, $x = \Xi$, $(\mathcal{X}, |\cdot|_{\mathcal{X}}) = (\rev{\Sigma_{\mathrm{a}}}, |\cdot|_{\mathcal{D}_{\mathrm{a}}})$\rev{, on which, as in the proof of Theorem~\ref{thm:IOC_3d}, the vector fields extend continuously, the sublevel sets of $V_{\mathrm{a}}$ are compact, and solutions from $\mathcal{D}_{\mathrm{a}}$ satisfy $\rho(t) > 0$ at every finite time}, and, for $i \in \{1, p, 2, 3\}$, the vector fields $(g_1, g_p, g_2, g_3) = (\bar g_1^{\mathrm{a}}, g_p, g_2^{\mathrm{a}}, \bar g_3^{\mathrm{a}})$, the functions $\mu_i$ and $\varepsilon_i$, and the inputs
\begin{align}\label{eq:u_bar_roll}
(\bar u_1, \bar u_p, \bar u_2, \bar u_3)
= \left(\frac{v}{\rho},\ p,\ \bar q,\ \frac{\bar r}{\cos\theta}\right)\,,
\end{align}
under which \eqref{eq:mi_J}, \eqref{eq:mi_l} coincide with \eqref{eq:J_ioc_roll}, \eqref{eq:run_cost_ioc_roll}, and \eqref{eq:mi_opt} and \eqref{eq:mi_damp}, mapped through $M(\phi)^{\!\top}$ in \eqref{eq:channel_rates}, with \eqref{eq:ioc_ctrl_opt_roll} and \eqref{eq:ioc_ctrl_roll}. Since $V_{\mathrm{a}}$ is a strict CLF by Theorem~\ref{thm:roll_GES}, conditions (i) and (ii) of Lemma~\ref{lem:mi_ioc} hold by Definition~\ref{def-CLF} as extended in Remark~\ref{rem:roll_def}. Hence, by Lemma~\ref{lem:mi_ioc}, \eqref{eq:ioc_ctrl_opt_roll} minimizes \eqref{eq:J_ioc_roll} \rev{and renders $\Xi = (0,\ldots,0)$ GAS on $\mathcal{D}_{\mathrm{a}}$}, and \eqref{eq:ioc_ctrl_roll} is continuous in $(\nu_1^{\mathrm{a}},\nu_p^{\mathrm{a}},\nu_2^{\mathrm{a}},\nu_3^{\mathrm{a}})$ and renders $\Xi= (0,\ldots,0)$ GAS on $\mathcal{D}_{\mathrm{a}}$ in accordance with Definition~\ref{def-our-GAS} as extended in Remark~\ref{rem:roll_def}.
\end{proof}

\begin{figure*}[!t]
\centering
\begin{minipage}[c]{0.28\textwidth}\centering\includegraphics[width=\linewidth]{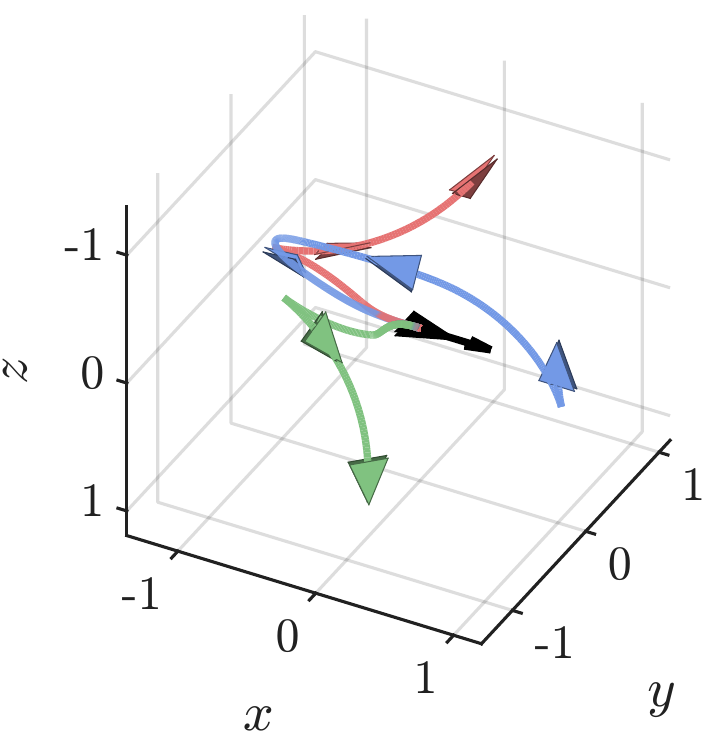}\end{minipage}\hfill%
\begin{minipage}[c]{0.20\textwidth}\centering\includegraphics[width=\linewidth]{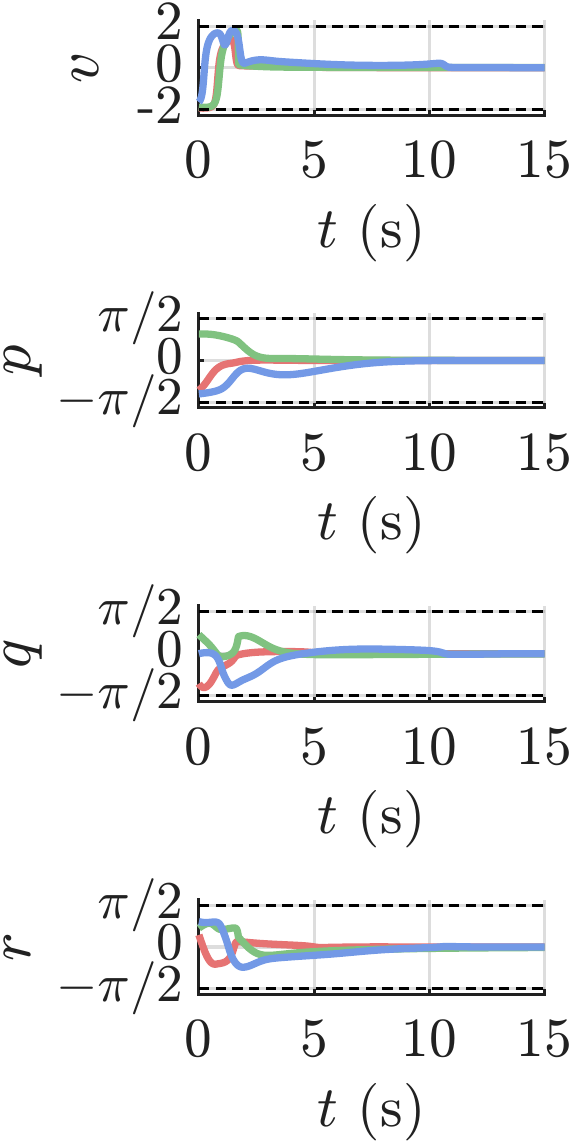}\end{minipage}\hfill%
\begin{minipage}[c]{0.28\textwidth}\centering\includegraphics[width=\linewidth]{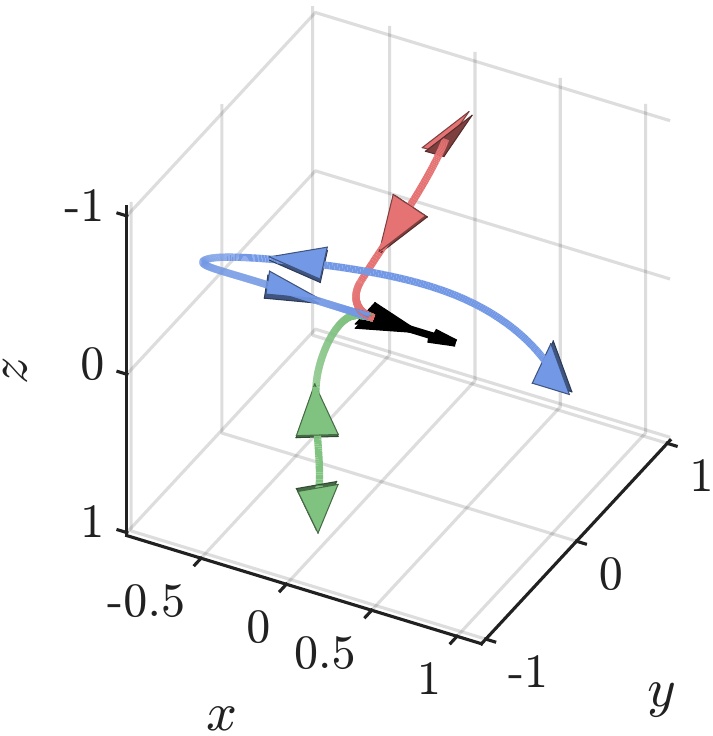}\end{minipage}\hfill%
\begin{minipage}[c]{0.20\textwidth}\centering\includegraphics[width=\linewidth]{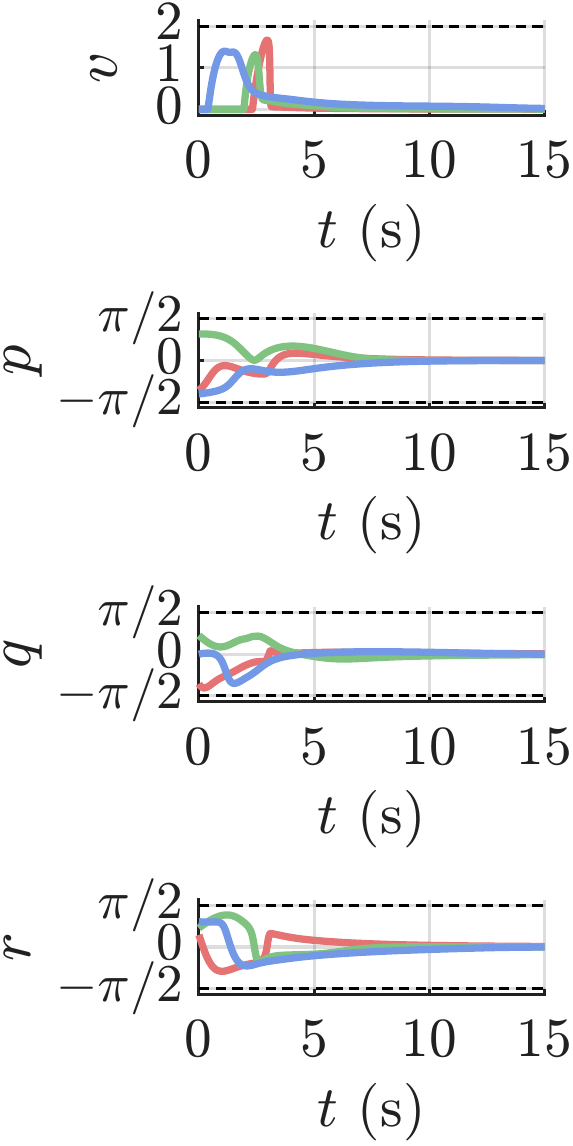}\end{minipage}
\caption{Closed-loop 3D trajectories of Example~\ref{ex:lncos} (NED convention) and optimal inputs of the roll-augmented system~\eqref{eq:roll_sph} under the bidirectional law~\eqref{eq:ctrl_opt_ln_bi} (left) and the unidirectional law~\eqref{eq:ctrl_opt_ln_uni} (right) from $\Xi_i(0) = (\Theta_i(0),\phi_i(0))$ with $\Theta_i(0)$ as in Fig.~\ref{fig:sim_quad_traj} and $(\phi_1(0),\phi_2(0),\phi_3(0)) = (\pi/3,-\pi/4,\pi/2)$, with the CLF gains $(k_1,k_2,k_3,k_4,k_5,k_\phi) = (0.6,2,2,2,2,1)$, $v_{\max} = 2$, $c_i = 0.5$ (bidirectional) and $c_1 = 2$, $c_p = c_2 = c_3 = 0.5$ (unidirectional). The black dart and arrow mark the target at the origin. All inputs remain within the bounds $|v(t)| < 2$ and $|p(t)|,|q(t)|,|r(t)| < \pi/2$ (dashed) for all $t$, and the unidirectional surge satisfies $v(t) \ge 0$.}
\label{fig:sim_lncos_traj}
\vspace{-0.5cm}
\end{figure*}

\subsection{Roll-augmented unidirectional IOC}

As in the unidirectional design of Corollary~\ref{cor:unidir_IOC},
\rev{unidirectional} motion requires $v(t) \ge 0$ for all $t$, which hinges on
the counterpart of Lemma~\ref{lem:unidir_nu1} for the augmented Lie
derivatives, stated next.

\begin{lemma}[Strictly negative surge Lie derivative at zero roll and attitude errors]\label{lem:unidir_nu1_roll}
For the strict CLF~\eqref{eq:Va_explicit},
\begin{align}\label{eq:nu_roll}
\bigl(\nu_1^{\mathrm{a}}, \nu_p^{\mathrm{a}},
\nu_2^{\mathrm{a}}, \nu_3^{\mathrm{a}}\bigr)
= \bigl(\nu_1,\ k_\phi\phi,\ k_5 e_2,\ -k_4e_1 + k_\phi\phi\sin\theta\bigr)\,,
\end{align}
and, for all $\Xi\in\Sigma_{\mathrm{a}}\setminus\{0\}$, \rev{$\bigl(\nu_p^{\mathrm{a}}, \nu_2^{\mathrm{a}}, \nu_3^{\mathrm{a}}\bigr) = (0,0,0) \implies \nu_1^{\mathrm{a}}(\Xi) < 0$ and, equivalently,
$\nu_1^{\mathrm{a}}(\Xi) \ge 0 \implies
\bigl(\nu_p^{\mathrm{a}}, \nu_2^{\mathrm{a}},
\nu_3^{\mathrm{a}}\bigr) \neq (0,0,0)$}.
\end{lemma}

\begin{proof}
Since $V_{\mathrm{a}} = V + k_\phi\phi^2/2$ and only $g_p$ and
$\bar g_3^{\mathrm{a}}$ have nonzero $\phi$-components,
\eqref{eq:nu_roll} follows with $\nu_2 =k_5 e_2$ and $\nu_3 = -k_4 e_1$ as in
the proof of Lemma~\ref{lem:unidir_nu1}. If
$\bigl(\nu_p^{\mathrm{a}}, \nu_2^{\mathrm{a}},
\nu_3^{\mathrm{a}}\bigr) = (0,0,0)$, then $\phi = 0$, hence
$e_1 = e_2 = 0$, so $\Theta \neq 0$, and Lemma~\ref{lem:unidir_nu1} gives $\nu_1^{\mathrm{a}} = \nu_1 < 0$.
\end{proof}

\rev{With $\varepsilon_1(\Xi) = \tilde{\varepsilon}_1(\Xi)h(\nu_1^{\mathrm{a}})$ and $h$ from~\eqref{eq:h_def}, the counterpart of Corollary~\ref{cor:unidir_IOC} for the roll-augmented system is stated next.}

\begin{corollary}[\rev{Unidirectional surge preserves GAS and inverse optimality under roll}]\label{cor:unidir_IOC_roll}
Consider the system~\eqref{eq:ioc_sys_roll} on
$\mathcal{D}_{\mathrm{a}}$. Let $\tilde{\varepsilon}_1(\Xi)$,
$\varepsilon_p(\Xi)$, $\varepsilon_2(\Xi)$, and $\varepsilon_3(\Xi)$
be scalar-valued functions\rev{, continuous and positive on $\Sigma_{\mathrm{a}}$}. Then, for the strict
CLF~\eqref{eq:Va_explicit}, the feedback law
\begin{align}\label{eq:unidir_v_roll}
v^* = \rho\,\tilde{\varepsilon}_1h(\nu_1^{\mathrm{a}})\,
(\mu_1')^{-1}\bigl(\tilde{\varepsilon}_1h(\nu_1^{\mathrm{a}})
|\nu_1^{\mathrm{a}}|\bigr)\,,
\end{align}
together with $p^*$, $q^*$, and $r^*$ as
in~\eqref{eq:ioc_opt_roll_p} and~\eqref{eq:ioc_opt_roll_qr},
guarantees $v^*(t) \ge 0$ for all $t$, renders $\Xi= (0,\ldots,0)$ GAS on
$\mathcal{D}_{\mathrm{a}}$ in accordance with
Definition~\ref{def-our-GAS} as extended in Remark~\ref{rem:roll_def},
and minimizes\rev{, within the class of controls under which the solution of~\eqref{eq:ioc_sys_roll} exists on $[t_0,\infty)$ and $\Xi(t) \to 0$,} the cost functional
\begin{align}\label{eq:J_cor_ioc_uni_roll}
J_{\mathrm{a}} = \int_0^\infty \Biggl[&l_{\mathrm{a}}(\Xi)
+ \mu_1\!\left(\frac{|v|}
{\tilde{\varepsilon}_1h(\nu_1^{\mathrm{a}})\rho}\right)
+ \mu_p\!\left(\frac{|p|}{\varepsilon_p}\right)\nonumber\\
&+ \mu_2\!\left(\frac{|\bar q|}{\varepsilon_2}\right)
+ \rev{\mu_3\!\left(\frac{|\tilde{\bar r}|}{\varepsilon_3}\right)}
\Biggr] {\rm d}t\,,
\end{align}
where
\begin{align}\label{eq:uni_ioc_l_roll}
l_{\mathrm{a}}(\Xi) \coloneqq&\
\ell\mu_1\bigl(\tilde{\varepsilon}_1h(\nu_1^{\mathrm{a}})
|\nu_1^{\mathrm{a}}|\bigr)
+\ell\mu_p(\varepsilon_p|\nu_p^{\mathrm{a}}|)\nonumber\\
&+\ell\mu_2(\varepsilon_2|\nu_2^{\mathrm{a}}|)
+\ell\mu_3(\varepsilon_3|\nu_3^{\mathrm{a}}|)\,,
\end{align}
\rev{with $\tilde{\bar r}$ as in~\eqref{eq:roll_sub} and the surge penalty on the set $\{\tilde{\varepsilon}_1 h(\nu_1^{\mathrm{a}}) = 0\}$ defined as in~\eqref{eq:zero_scale_conv}.}
\end{corollary}

\begin{revblock}
\begin{proof}
Identical to the proof of Corollary~\ref{cor:unidir_IOC}, with $V_{\mathrm{a}}$, $\mathcal{D}_{\mathrm{a}}$, $\Sigma_{\mathrm{a}}$, the Lie derivatives~\eqref{eq:nu_roll}, the inputs~\eqref{eq:u_bar_roll}, the computations of Lemma~\ref{lem:mi_ioc} for $m = 4$, and Lemma~\ref{lem:unidir_nu1_roll} in place of Lemma~\ref{lem:unidir_nu1}.
\end{proof}
\end{revblock}

As in the roll-free case, when $\nu_1^{\mathrm{a}} \ge 0$ the surge
input~\eqref{eq:unidir_v_roll} vanishes, and
Lemma~\ref{lem:unidir_nu1_roll} guarantees that $\dot V_{\mathrm{a}}$
remains negative definite through the remaining
\rev{inputs} $(p, \bar q, \bar r)$.

\begin{figure*}[!t]
\centering
\begin{minipage}[c]{0.48\textwidth}\centering\includegraphics[width=.95\linewidth]{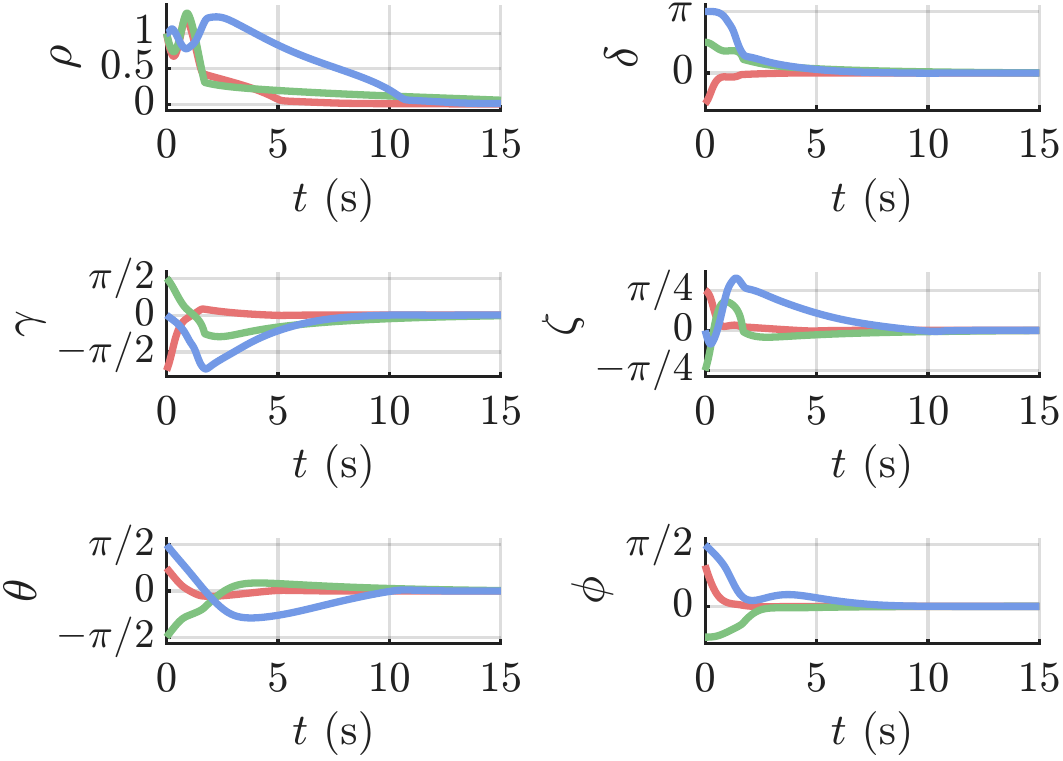}\end{minipage}\hfill%
\begin{minipage}[c]{0.48\textwidth}\centering\includegraphics[width=.95\linewidth]{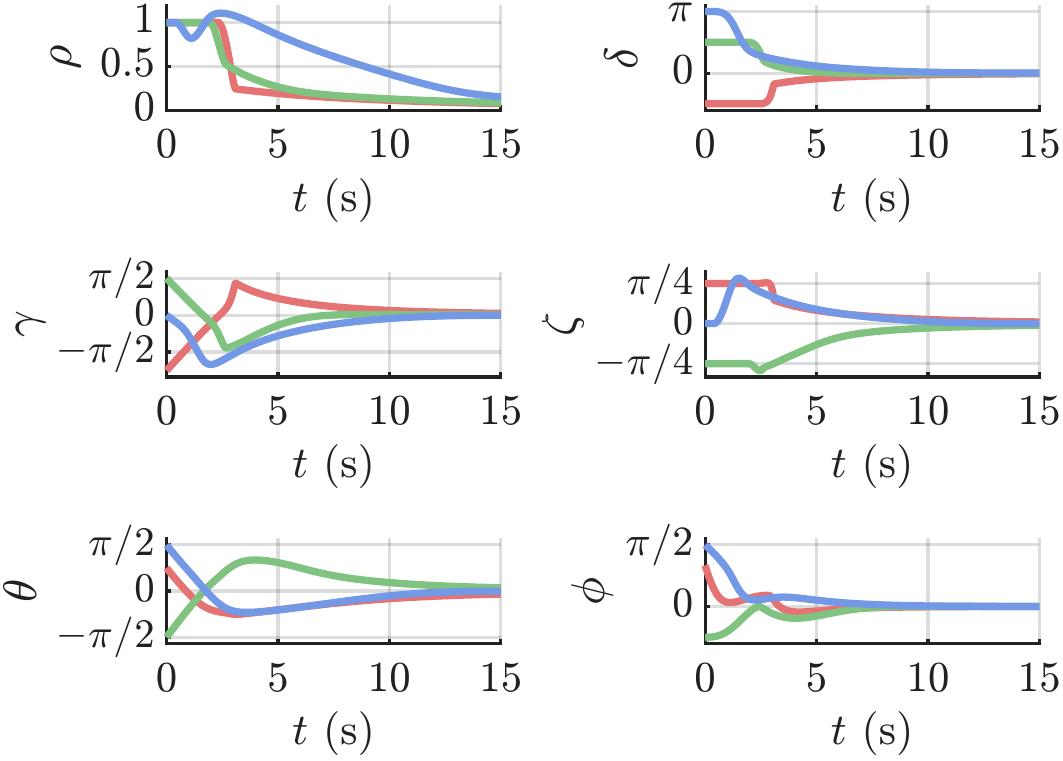}\end{minipage}
\caption{Example~\ref{ex:lncos}, spherical states of the trajectories in Fig.~\ref{fig:sim_lncos_traj}: bidirectional surge (penalty~\eqref{eq:mu_lncos} with~\eqref{eq:eps_ln_bi}, left) and unidirectional surge (with~\eqref{eq:eps_ln_uni}, $n = 1$, right).}
\label{fig:sim_lncos_states}
\vspace{-0.5cm}
\end{figure*}

% \begin{figure*}[!t]
% \centering
% \begin{minipage}[c]{0.48\textwidth}\centering\includegraphics[width=.95\linewidth]{Figs/jrnl_sims/ExVI1_bi_states_sph.pdf}\end{minipage}\hfill%
% \begin{minipage}[c]{0.48\textwidth}\centering\includegraphics[width=.95\linewidth]{Figs/jrnl_sims/ExVI1_bi_states_cart.pdf}\end{minipage}
% \caption{Example~\ref{ex:lncos}, bidirectional surge (input-constraining penalty~\eqref{eq:mu_lncos} with~\eqref{eq:eps_ln_bi}) with spherical states (left) and Cartesian/Euler states (right) of the trajectories in Fig.~\ref{fig:sim_lncos_traj}.}
% \label{fig:sim_lncos_bi}
% \vspace{-0.2cm}
% \end{figure*}

% \begin{figure*}[!t]
% \centering
% \begin{minipage}[c]{0.48\textwidth}\centering\includegraphics[width=.95\linewidth]{Figs/jrnl_sims/ExVI1_uni_states_sph.pdf}\end{minipage}\hfill%
% \begin{minipage}[c]{0.48\textwidth}\centering\includegraphics[width=.95\linewidth]{Figs/jrnl_sims/ExVI1_uni_states_cart.pdf}\end{minipage}
% \caption{Example~\ref{ex:lncos}, unidirectional surge (penalty~\eqref{eq:mu_lncos} with~\eqref{eq:eps_ln_uni}, $n = 1$) with spherical states (left) and Cartesian/Euler states (right) of the trajectories in Fig.~\ref{fig:sim_lncos_traj}.}
% \label{fig:sim_lncos_uni}
% \vspace{-0.5cm}
% \end{figure*}

\begin{example}[Bounded bidirectional and unidirectional parking]\label{ex:lncos}
Consider the input-constraining penalty
\begin{align}\label{eq:mu_lncos}
\mu_i(s) = -c_i\ln(\cos(s))\,,
\end{align}
where $c_i>0$ and $i\in\{1,p,2,3\}$, for which $(\mu_i')^{-1}(s) = \arctan(s/c_i)$ and $\ell\mu_i(s) = s\arctan(s/c_i) - \frac{c_i}{2}\ln\bigl(1+s^2/c_i^2\bigr)$. The resulting optimal control law for the bidirectional surge from Theorem~\ref{thm:IOC_roll} is then
\begin{subequations}\label{eq:ctrl_opt_ln_bi}
\begin{align}
v^* &= -\rho\,\varepsilon_1
\arctan\!\left(\frac{\varepsilon_1\nu_1^{\mathrm{a}}}{c_1}\right)\\
p^* &= -\varepsilon_p
\arctan\!\left(\frac{\varepsilon_p\nu_p^{\mathrm{a}}}{c_p}\right)\\
\begin{bmatrix} q^*\\ r^*\end{bmatrix} &= M(\phi)^{\!\top}
\begin{bmatrix}
-\varepsilon_2\arctan\bigl(\varepsilon_2\nu_2^{\mathrm{a}}/c_2\bigr)\\[2pt]
-\cos(\theta)\,\varepsilon_3\arctan\bigl(\varepsilon_3\nu_3^{\mathrm{a}}/c_3\bigr)
\end{bmatrix},
\end{align}
\end{subequations}
which minimizes the cost functional~\eqref{eq:J_ioc_roll} with the running cost~\eqref{eq:run_cost_ioc_roll}. Now, impose the input bounds $v_{\max} = 2$ and $p_{\max} = q_{\max} = r_{\max} = \pi/2$, and choose
\begin{align}\label{eq:eps_ln_bi}
\varepsilon_1 = \frac{2}{\pi}\frac{v_{\max}}{0.01+\rho}\,,\quad
\varepsilon_p = 1\,,\quad
\varepsilon_2 = \varepsilon_3 = \frac{1}{\sqrt{2}}\,.
\end{align}
Since $|\arctan(\cdot)| < \pi/2$ for every $c_i > 0$, it holds that $|v^*| < v_{\max}\rho/(0.01+\rho) < v_{\max}$ and $|p^*| < p_{\max}$. For the bound on the pitch and yaw inputs, \eqref{eq:eps_ln_bi} guarantees the bound $(\bar q^*)^2 + (\bar r^*)^2 < (\pi/2)^2$, which implies $(q^*)^2+(r^*)^2 < (\pi/2)^2$ and thus, $|q^*| < q_{\max}$ and $|r^*| < r_{\max}$.

The analogous result for the unidirectional surge case in Corollary~\ref{cor:unidir_IOC_roll}, using the penalty~\eqref{eq:mu_lncos} and $h$ from~\eqref{eq:h_def} with $n = 1$, is given by the optimal control law \begin{align}\label{eq:ctrl_opt_ln_uni} v^* = \rho\,\tilde{\varepsilon}_1 h(\nu_1^{\mathrm{a}}) \arctan\!\left(\frac{\tilde{\varepsilon}_1 h(\nu_1^{\mathrm{a}})|\nu_1^{\mathrm{a}}|}{c_1}\right)\,, \end{align} 
together with $p^*$, $q^*$, and $r^*$ as in~\eqref{eq:ctrl_opt_ln_bi}, which minimizes the cost functional~\eqref{eq:J_cor_ioc_uni_roll} with the running cost~\eqref{eq:uni_ioc_l_roll}. Replacing the choice of $\varepsilon_1$ in~\eqref{eq:eps_ln_bi} by \begin{align}\label{eq:eps_ln_uni} 
\tilde{\varepsilon}_1(\Xi) = \frac{2}{\pi}\frac{v_{\max}} {(0.01+\rho)\bigl(0.01+|\nu_1^{\mathrm{a}}|\bigr)}\,, 
\end{align} 
gives, with $n = 1$, $\tilde{\varepsilon}_1 h(\nu_1^{\mathrm{a}}) \le \frac{2}{\pi}\frac{v_{\max}}{0.01+\rho}$, so the same bounds are enforced, in addition to $v^*(t) \ge 0$ for all $t$.

In both cases, the weights $c_i$ shape the feedback without moving the bounds: $c_i$ sets the small-signal slope $\varepsilon_i^2/c_i$ and the corner at $|\nu_i^{\mathrm{a}}| = c_i/\varepsilon_i$, while the saturation level $\varepsilon_i\pi/2$ is independent of $c_i$. \rev{Since} $\ell\mu_i(s)$ behaves as $s^2/(2c_i)$ near the origin, a larger $c_i$ also lowers the state penalty $l_{\mathrm{a}}(\Xi)$. Figs.~\ref{fig:sim_lncos_traj} and~\ref{fig:sim_lncos_states} show the closed-loop responses from the initial conditions of Example~\ref{ex:quad} with the inclusion of roll. All inputs remain within the user-defined bounds.
\end{example}
\section{Conclusion}

This paper presents constructive analytical results for global exponential stabilization, inverse optimality, robustness, and input-constrained control of the 3D nonholonomic vehicle actuated in surge velocity, pitch rate, and yaw rate. A spherical coordinate transformation circumvents Brockett's obstruction while excluding only an unavoidable codimension-two set of measure zero. We develop backstepping feedback laws and strict barrier CLFs that guarantee global exponential stability with user-assigned decay rates, avoidance of the singular set, and exponential attractivity from all Cartesian initial conditions off the polar axis, and we show that any such design extends to the roll-augmented kinematics through a static rotation of the rate inputs and a simple roll feedback. Leveraging the strict CLFs, we introduce a general inverse optimal framework for driftless multi-input systems that yields broad families of stabilizers for the roll-free and roll-augmented vehicles, minimizing a meaningful cost functional with either bidirectional or unidirectional surge, infinite gain margins, and user-defined input bounds through Legendre–Fenchel transformations. \rev{Together, these results provide constructive designs for the stabilization of the 3D nonholonomic vehicle and can be extended to other nonholonomic vehicles, unifying exponential stability, optimality, and actuator constraints in a systematic framework.}

\vspace{-0.2cm}
\section*{Acknowledgment}

The authors used Claude (Anthropic) and ChatGPT (OpenAI) for language editing, text refinement, and to assist in certain mathematical derivations. All AI-generated outputs were fully verified by the authors, who take sole responsibility for the content and correctness of this manuscript.
%%%%%%%%%%%%%%%%%%%%%%%%%%%%%%%%%%%%%%%%%%%%%%%%%%%%%%

\section*{References}
\vspace{-0.5cm}
\bibliographystyle{IEEEtranS}
\bibliography{bib,bib-Dubins,root}

% Generated by IEEEtranS.bst, version: 1.14 (2015/08/26)
\begin{thebibliography}{10}
\providecommand{\url}[1]{#1}
\csname url@samestyle\endcsname
\providecommand{\newblock}{\relax}
\providecommand{\bibinfo}[2]{#2}
\providecommand{\BIBentrySTDinterwordspacing}{\spaceskip=0pt\relax}
\providecommand{\BIBentryALTinterwordstretchfactor}{4}
\providecommand{\BIBentryALTinterwordspacing}{\spaceskip=\fontdimen2\font plus
\BIBentryALTinterwordstretchfactor\fontdimen3\font minus \fontdimen4\font\relax}
\providecommand{\BIBforeignlanguage}[2]{{%
\expandafter\ifx\csname l@#1\endcsname\relax
\typeout{** WARNING: IEEEtranS.bst: No hyphenation pattern has been}%
\typeout{** loaded for the language `#1'. Using the pattern for}%
\typeout{** the default language instead.}%
\else
\language=\csname l@#1\endcsname
\fi
#2}}
\providecommand{\BIBdecl}{\relax}
\BIBdecl

\bibitem{aguiar2007trajectory}
A.~P. Aguiar and J.~P. Hespanha, ``Trajectory-tracking and path-following of underactuated autonomous vehicles with parametric modeling uncertainty,'' \emph{IEEE Transactions on Automatic Control}, vol.~52, no.~8, pp. 1362--1379, 2007.

\bibitem{aguiar2002global}
A.~P. Aguiar and A.~M. Pascoal, ``Global stabilization of an underactuated autonomous underwater vehicle via logic-based switching,'' in \emph{Proceedings of the 41st IEEE Conference on Decision and Control ({CDC})}, vol.~3.\hskip 1em plus 0.5em minus 0.4em\relax IEEE, 2002, pp. 3267--3272.

\bibitem{aicardi1995}
M.~Aicardi, G.~Casalino, A.~Bicchi, and A.~Balestrino, ``Closed loop steering of unicycle like vehicles via {L}yapunov techniques,'' \emph{IEEE Robotics \& Automation Magazine}, vol.~2, no.~1, pp. 27--35, 1995.

\bibitem{aicardi2001cusp}
M.~Aicardi, G.~Cannata, G.~Casalino, and G.~Indiveri, ``Cusp-free, time-invariant, 3d feedback control law for a nonholonomic floating robot,'' \emph{The International Journal of Robotics Research}, vol.~20, no.~4, pp. 300--311, 2001.

\bibitem{alonge2001trajectory}
F.~Alonge, F.~D'Ippolito, and F.~M. Raimondi, ``Trajectory tracking of underactuated underwater vehicles,'' in \emph{Proceedings of the 40th IEEE Conference on Decision and Control (CDC)}, vol.~5.\hskip 1em plus 0.5em minus 0.4em\relax IEEE, 2001, pp. 4421--4426.

\bibitem{astolfi1996discontinuous}
A.~Astolfi, ``Discontinuous control of nonholonomic systems,'' \emph{Systems \& Control Letters}, vol.~27, no.~1, pp. 37--45, 1996.

\bibitem{bharadwaj1998geometry}
S.~Bharadwaj, M.~Osipchuk, K.~Mease, and F.~Park, ``Geometry and inverse optimality in global attitude stabilization,'' \emph{Journal of Guidance, Control, and Dynamics}, vol.~21, no.~6, pp. 930--939, 1998.

\bibitem{bhat2000topological}
S.~P. Bhat and D.~S. Bernstein, ``A topological obstruction to continuous global stabilization of rotational motion and the unwinding phenomenon,'' \emph{Systems \& Control Letters}, vol.~39, no.~1, pp. 63--70, 2000.

\bibitem{bloch1996stabilization_slidingmode}
A.~Bloch and S.~Drakunov, ``Stabilization and tracking in the nonholonomic integrator via sliding modes,'' \emph{Systems \& Control Letters}, vol.~29, no.~2, pp. 91--99, 1996.

\bibitem{brockett1983asymptotic}
R.~W. Brockett, ``{Asymptotic Stability and Feedback Stabilization},'' in \emph{Differential Geometric Control Theory}, R.~W. Brockett, R.~S. Millman, and H.~J. Sussmann, Eds.\hskip 1em plus 0.5em minus 0.4em\relax Birkhäuser Boston, 1983, pp. 181--208.

\bibitem{bullo2005geometric}
F.~Bullo and A.~D. Lewis, \emph{Geometric Control of Mechanical Systems}, ser. Texts in Applied Mathematics.\hskip 1em plus 0.5em minus 0.4em\relax New York: Springer, 2005, vol.~49.

\bibitem{byrnes1991passivity}
C.~I. Byrnes, A.~Isidori, J.~C. Willems \emph{et~al.}, ``Passivity, feedback equivalence, and the global stabilization of minimum phase nonlinear systems,'' \emph{IEEE Transactions on automatic control}, vol.~36, no.~11, pp. 1228--1240, 1991.

\bibitem{caccia2000guidance}
M.~Caccia and G.~Veruggio, ``Guidance and control of a reconfigurable unmanned underwater vehicle,'' \emph{Control engineering practice}, vol.~8, no.~1, pp. 21--37, 2000.

\bibitem{chaturvedi2011rigid}
N.~A. Chaturvedi, A.~K. Sanyal, and N.~H. McClamroch, ``Rigid-body attitude control,'' \emph{IEEE Control Systems Magazine}, vol.~31, no.~3, pp. 30--51, 2011.

\bibitem{chitsaz2007time}
H.~Chitsaz and S.~M. LaValle, ``Time-optimal paths for a {Dubins} airplane,'' in \emph{Proceedings of the 46th IEEE Conference on Decision and Control ({CDC})}.\hskip 1em plus 0.5em minus 0.4em\relax IEEE, 2007, pp. 2379--2384.

\bibitem{coron1992global_controllable}
J.-M. Coron, ``Global asymptotic stabilization for controllable systems without drift,'' \emph{Mathematics of Control, Signals and Systems}, vol.~5, no.~3, pp. 295--312, 1992.

\bibitem{coron1994relation}
J.-M. Coron and L.~Rosier, ``A relation between continuous time-varying and discontinuous feedback stabilization,'' \emph{J. Math. Syst., Estimation, Control}, vol.~4, pp. 67--84, 1994.

\bibitem{deluca1998feedback}
A.~De~Luca, G.~Oriolo, and C.~Samson, ``Feedback control of a nonholonomic car-like robot,'' in \emph{Robot Motion Planning and Control}, ser. Lecture Notes in Control and Information Sciences, J.-P. Laumond, Ed.\hskip 1em plus 0.5em minus 0.4em\relax Springer-Verlag, 1998, vol. 229, pp. 171--253.

\bibitem{deng1997stochastic}
H.~Deng and M.~Krsti{\'c}, ``Stochastic nonlinear stabilization—ii: Inverse optimality,'' \emph{Systems \& control letters}, vol.~32, no.~3, pp. 151--159, 1997.

\bibitem{do2015global}
K.~D. Do, ``Global inverse optimal stabilization of stochastic nonholonomic systems,'' \emph{Systems \& Control Letters}, vol.~75, pp. 41--55, 2015.

\bibitem{do2002global}
K.~D. Do, Z.-P. Jiang, J.~Pan, and H.~Nijmeijer, ``Global output feedback universal controller for stabilization and tracking of underactuated {ODIN}--an underwater vehicle,'' in \emph{Proceedings of the 41st IEEE Conference on Decision and Control (CDC)}, vol.~1.\hskip 1em plus 0.5em minus 0.4em\relax IEEE, 2002, pp. 504--509.

\bibitem{egeland1994exponential}
O.~Egeland, E.~Berglund, and O.~J. S{\o}rdalen, ``Exponential stabilization of a nonholonomic underwater vehicle with constant desired configuration,'' in \emph{Proceedings of the IEEE International Conference on Robotics and Automation (ICRA)}.\hskip 1em plus 0.5em minus 0.4em\relax IEEE, 1994, pp. 20--25.

\bibitem{egeland1996feedback}
O.~Egeland, M.~Dalsmo, and O.~J. Soerdalen, ``Feedback control of a nonholonomic underwater vehicle with a constant desired configuration,'' \emph{The International journal of robotics research}, vol.~15, no.~1, pp. 24--35, 1996.

\bibitem{fossen1994guidance}
T.~I. Fossen, \emph{Guidance and control of ocean vehicles}.\hskip 1em plus 0.5em minus 0.4em\relax Chichester: John Wiley \& Sons, 1994.

\bibitem{fossen2023alos}
------, ``An adaptive line-of-sight ({ALOS}) guidance law for path following of aircraft and marine craft,'' \emph{IEEE Transactions on Control Systems Technology}, vol.~31, no.~6, pp. 2887--2894, 2023.

\bibitem{fossen2024alos3d}
T.~I. Fossen and A.~P. Aguiar, ``A uniform semiglobal exponential stable adaptive line-of-sight ({ALOS}) guidance law for {3-D} path following,'' \emph{Automatica}, vol. 163, p. 111556, 2024.

\bibitem{freeman1996robust}
R.~A. Freeman and P.~V. Kokotovi{\'c}, \emph{Robust Nonlinear Control Design: State-Space and {L}yapunov Techniques}, ser. Systems \& Control: Foundations \& Applications.\hskip 1em plus 0.5em minus 0.4em\relax Boston, MA: Birkh{\"a}user, 1996.

\bibitem{wang24_force_controlled_safestable}
T.~Han and B.~Wang, ``Safety-critical stabilization of force-controlled nonholonomic mobile robots,'' \emph{IEEE Control Systems Letters}, vol.~8, pp. 2469--2474, 2024.

\bibitem{hardy_inequalities_1989}
G.~H. Hardy, J.~E. Littlewood, and G.~Pólya, \emph{Inequalities}, 2nd~ed.\hskip 1em plus 0.5em minus 0.4em\relax Cambridge, U.K.: Cambridge University Press, 1989.

\bibitem{Heetal2022exp3D}
X.~He, Z.~Sun, Z.~Geng, and A.~Robertsson, ``Exponential set-point stabilization of underactuated vehicles moving in three-dimensional space,'' \emph{IEEE/CAA Journal of Automatica Sinica}, vol.~9, no.~2, pp. 270--282, 2022.

\bibitem{hespanha1999_hybrid_stabilization}
J.~P. Hespanha and A.~S. Morse, ``Stabilization of nonholonomic integrators via logic-based switching,'' \emph{Automatica}, vol.~35, no.~3, pp. 385--393, 1999.

\bibitem{jiang2010controlling}
Z.-P. Jiang, ``Controlling underactuated mechanical systems: A review and open problems,'' \emph{Advances in the theory of control, signals and systems with physical modeling}, pp. 77--88, 2010.

\bibitem{jurdjevic1978controllability}
V.~Jurdjevic and J.~P. Quinn, ``Controllability and stability,'' \emph{Journal of differential equations}, vol.~28, no.~3, pp. 381--389, 1978.

\bibitem{kalman1964inverse_opt}
\BIBentryALTinterwordspacing
R.~E. Kalman, ``When is a linear control system optimal?'' \emph{Journal of Basic Engineering}, vol.~86, no.~1, pp. 51--60, 03 1964. [Online]. Available: \url{https://doi.org/10.1115/1.3653115}
\BIBentrySTDinterwordspacing

\bibitem{kim2026IOC_3D_ACC}
\BIBentryALTinterwordspacing
K.~H. Kim, M.~Diagne, and M.~Krstic, ``Inverse optimal feedback stabilization of {3D} nonholonomic vehicles in spherical coordinates,'' 2026. [Online]. Available: \url{https://arxiv.org/abs/2609.36212}
\BIBentrySTDinterwordspacing

\bibitem{Part2_kimIOC2025}
\BIBentryALTinterwordspacing
K.~H. Kim, V.~Todorovski, and M.~Krstic, ``Nonholonomic {R}obot {P}arking by {F}eedback -- {P}art {II}: {N}onmodular, {I}nverse {O}ptimal, {A}daptive, {P}rescribed/{F}ixed-{T}ime and {S}afe {D}esigns,'' 2025. [Online]. Available: \url{https://arxiv.org/abs/2511.15219}
\BIBentrySTDinterwordspacing

\bibitem{kim2026spherical3d}
\BIBentryALTinterwordspacing
------, ``Global exponential stabilization of a {3D} nonholonomic vehicle in spherical coordinates,'' 2026. [Online]. Available: \url{https://arxiv.org/abs/2609.24192}
\BIBentrySTDinterwordspacing

\bibitem{krstic_rigidspacecraft_1999}
M.~Krstic and P.~Tsiotras, ``Inverse optimal stabilization of a rigid spacecraft,'' \emph{IEEE Transactions on Automatic Control}, vol.~44, no.~5, pp. 1042--1049, 1999.

\bibitem{krstic2023inverse}
\BIBentryALTinterwordspacing
M.~Krstic, ``\BIBforeignlanguage{en}{Inverse {Optimal} {Safety} {Filters}},'' \emph{\BIBforeignlanguage{en}{IEEE Transactions on Automatic Control}}, pp. 1--16, 2023. [Online]. Available: \url{https://ieeexplore.ieee.org/document/10130642/}
\BIBentrySTDinterwordspacing

\bibitem{krstic2024offender}
------, ``An offender--defender safety game,'' \emph{Annual Reviews in Control}, vol.~57, p. 100939, 2024.

\bibitem{krstic1998inverse}
M.~Krstic and Z.-H. Li, ``Inverse optimal design of input-to-state stabilizing nonlinear controllers,'' \emph{IEEE Transactions on Automatic Control}, vol.~43, no.~3, pp. 336--350, 1998.

\bibitem{lapierre2003nonlinear}
L.~Lapierre, D.~Soetanto, and A.~Pascoal, ``Nonlinear path following with applications to the control of autonomous underwater vehicles,'' in \emph{Proceedings of the 42nd IEEE Conference on Decision and Control ({CDC})}, vol.~2.\hskip 1em plus 0.5em minus 0.4em\relax IEEE, 2003, pp. 1256--1261.

\bibitem{li2018design}
Y.~Li, Y.~Li, and Q.~Wu, ``Design for three-dimensional stabilization control of underactuated autonomous underwater vehicles,'' \emph{Ocean Engineering}, vol. 150, pp. 327--336, 2018.

\bibitem{lugo2014dubins}
I.~Lugo-C{\'a}rdenas, G.~Flores, S.~Salazar, and R.~Lozano, ``Dubins path generation for a fixed wing uav,'' in \emph{2014 International conference on unmanned aircraft systems (ICUAS)}.\hskip 1em plus 0.5em minus 0.4em\relax IEEE, 2014, pp. 339--346.

\bibitem{moylan1973nonlinear}
P.~Moylan and B.~Anderson, ``Nonlinear regulator theory and an inverse optimal control problem,'' \emph{IEEE Transactions on Automatic Control}, vol.~18, no.~5, pp. 460--465, 1973.

\bibitem{pettersen1999time}
K.~Y. Pettersen and O.~Egeland, ``Time-varying exponential stabilization of the position and attitude of an underactuated autonomous underwater vehicle,'' \emph{IEEE Transactions on Automatic Control}, vol.~44, no.~1, pp. 112--115, 1999.

\bibitem{praly2022fonctions}
\BIBentryALTinterwordspacing
L.~Praly and D.~Bresch-Pietri, \emph{Fonctions de Lyapunov, stabilit{\'e}, stabilisation et att{\'e}nuation de perturbations: Partie 2 Stabilisation}.\hskip 1em plus 0.5em minus 0.4em\relax Spartacus-idh, 2022, no. pt. 2. [Online]. Available: \url{https://spartacus-idh.com/liseuse/118/#page/1}
\BIBentrySTDinterwordspacing

\bibitem{prieur2003robust}
C.~Prieur and A.~Astolfi, ``Robust stabilization of chained systems via hybrid control,'' \emph{IEEE Transactions on Automatic Control}, vol.~48, no.~10, pp. 1768--1772, 2003.

\bibitem{restrepo_3d_2019}
E.~Restrepo, I.~Sarras, A.~Loria, and J.~Marzat, ``3{D} {UAV} navigation with moving-obstacle avoidance using barrier {L}yapunov functions,'' \emph{IFAC-PapersOnLine}, vol.~52, no.~12, pp. 49--54, 2019.

\bibitem{ryan1994brockett}
E.~P. Ryan, ``On {B}rockett’s condition for smooth stabilizability and its necessity in a context of nonsmooth feedback,'' \emph{SIAM Journal on Control and Optimization}, vol.~32, no.~6, pp. 1597--1604, 1994.

\bibitem{samson1990velocity}
C.~Samson, ``Velocity and torque feedback control of a nonholonomic cart,'' in \emph{Advanced Robot Control: Proceedings of the International Workshop on Nonlinear and Adaptive Control: Issues in Robotics, Grenoble, France}.\hskip 1em plus 0.5em minus 0.4em\relax Springer-Verlag, 1990, vol. 162, pp. 125--151.

\bibitem{sepulchre1997constructive}
R.~Sepulchre, M.~Jankovic, and P.~V. Kokotovic, \emph{Constructive Nonlinear Control}, ser. Communications and Control Engineering.\hskip 1em plus 0.5em minus 0.4em\relax Berlin, Heidelberg: Springer-Verlag, 1997.

\bibitem{sontag2013mathematical}
E.~D. Sontag, \emph{Mathematical control theory: deterministic finite dimensional systems}.\hskip 1em plus 0.5em minus 0.4em\relax Springer Science \& Business Media, 2013, vol.~6.

\bibitem{Part1_todorovskiCLF2025}
\BIBentryALTinterwordspacing
V.~Todorovski, K.~H. Kim, A.~Astolfi, and M.~Krstic, ``Nonholonomic {R}obot {P}arking by {F}eedback -- {P}art {I}: Modular {S}trict {CLF} {D}esigns,'' 2025. [Online]. Available: \url{https://arxiv.org/abs/2511.15119}
\BIBentrySTDinterwordspacing

\bibitem{todorovski2026global}
------, ``Global exponential stabilization of the kinematic bicycle model of a car in polar coordinates,'' \emph{arXiv preprint arXiv:2607.26442}, 2026.

\bibitem{vamvoudakis2022nonequilibrium}
K.~G. Vamvoudakis, F.~Fotiadis, A.~Kanellopoulos, and N.-M.~T. Kokolakis, ``Nonequilibrium dynamical games: A control systems perspective,'' \emph{Annual Reviews in Control}, vol.~53, pp. 6--18, 2022.

\end{thebibliography}

\vspace{-1cm}
\begin{IEEEbiography}
[{\includegraphics[width=1in,height=1.25in,clip,keepaspectratio]{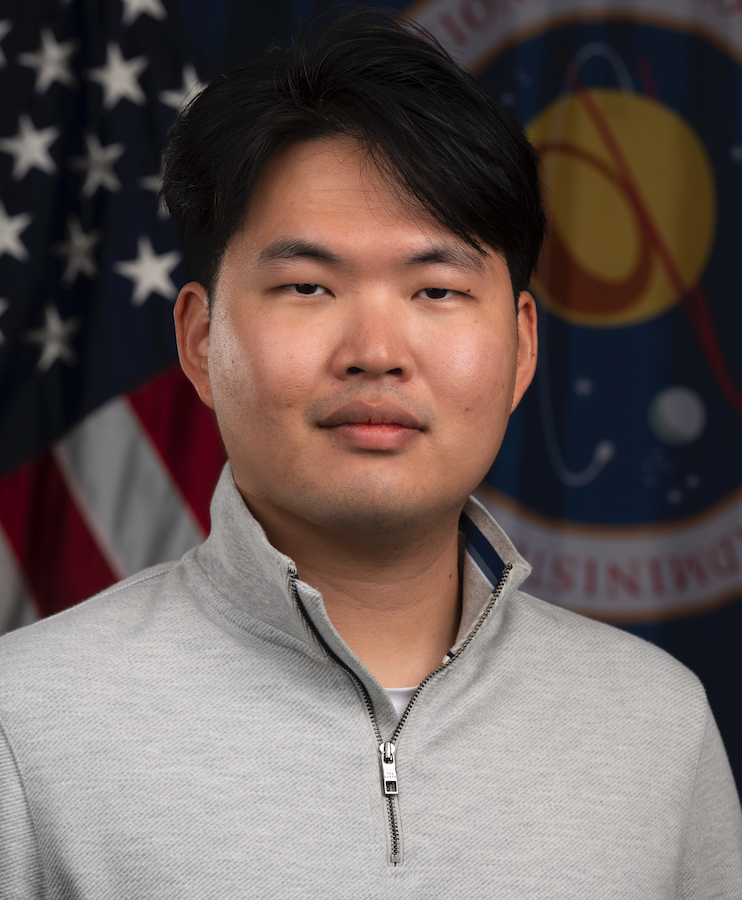}}]{Kwang Hak Kim}
received the B.S. degree in Aerospace Engineering from the University of Texas at Austin, Austin, TX, USA, in 2020, and the M.S. degree in Engineering Sciences (Mechanical Engineering) from the University of California San Diego, La Jolla, CA, USA, in 2024. He is currently pursuing the Ph.D. degree in Engineering Sciences (Mechanical Engineering) at the University of California San Diego. His research interests include nonlinear and safety-critical control of nonholonomic systems.
\end{IEEEbiography}

\vspace{-1 cm}

\begin{IEEEbiography}[{\includegraphics[width=1in,height=1.25in,clip,keepaspectratio]{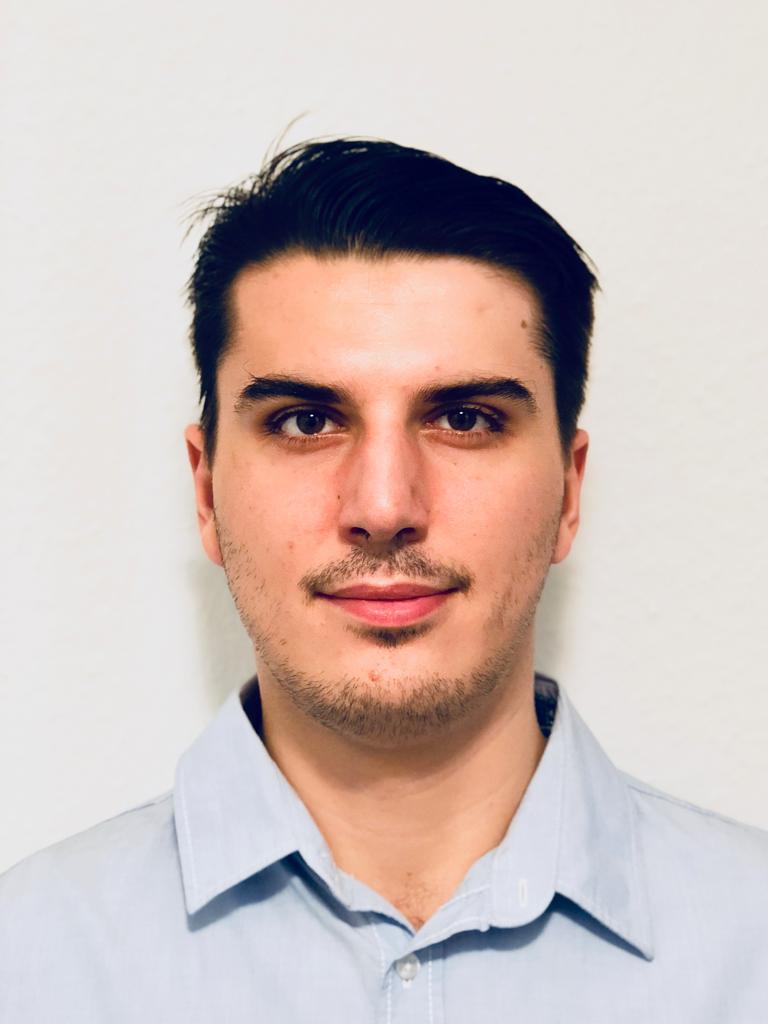}}]{Velimir Todorovski}

received the B.Sc. degree in Electrical Engineering, Electronics and Information Technology from the University of Erlangen-Nuremberg, Erlangen, Germany, and the M.Sc. degree in Electrical and Computer Engineering from Technical University of Munich, Munich, Germany, in 2019 and 2022, respectively.
He is currently a Ph.D. student with the Department of Mechanical and Aerospace Engineering,  University of California San Diego, La Jolla, CA, USA. His research interests include model-free optimization and safe stabilization of nonholonomic systems.
\end{IEEEbiography}

\vspace{-1cm}

\begin{IEEEbiography}[{\includegraphics[width=1in,height=1.25in,clip,keepaspectratio]{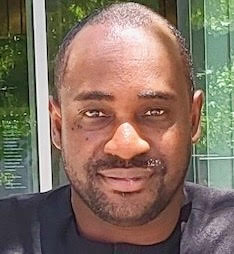}}]{Mamadou Diagne}
(Senior Member, IEEE) is Associate Professor with the Department of Mechanical and Aerospace Engineering at the University of California San Diego, affiliated with the Department of Electrical and Computer Engineering. He completed his Ph.D. in 2013 at University Claude Bernard Lyon I, Villeurbanne (France). Between 2017 and 2022, he was Assistant Professor with the Department of Mechanical Aerospace and Nuclear Engineering at Rensselaer Polytechnic Institute, Troy, New York. From 2013 to 2016, he was a Postdoctoral Scholar, first at the University of California San Diego, and then at the University of Michigan. His research focuses on the control of distributed parameter systems (DPS), adaptive control, time-delay systems, extremum seeking control and safe nonlinear control. He served as Associate Editor for the IEEE Transactions on Automatic Control, Automatica, Systems \& Control Letters and ASME Journal of Dynamic Systems, Measurement, and Control. He chairs the IEEE TC on DPS and serves as Vice-Chair for Education of the IFAC TC on DPS, Vice-Chair of the ASME DSCD TC on DPS, and Vice-Chair for Industry of the IFAC TC on Adaptive and Learning Systems. He is the founder of the African Control Systems Symposium (AFCONS). He received the NSF Career Award in 2020.
\end{IEEEbiography}

\vspace{-1cm}
\begin{IEEEbiographynophoto}{Miroslav Krsti\'{c}}
(Fellow, IEEE) For the photograph and biography of Miroslav Krsti\'{c}, see the January, 2026 issue of the IEEE Transactions on Automatic Control.
% \hyperlink{https://ieeexplore.ieee.org/abstract/document/11080060}{https://ieeexplore.ieee.org/abstract/document/11080060}.
\end{IEEEbiographynophoto}

%%%%%%%%%%%%%%%%%%%%%%%%%%%%%%%%%%%%%%%%%%%%%%%%%%%%%%%%%%%%%%%%%%%%%%%%%%%%%%%%%%%%%%%%%%%%%%%%%%%%%%%%%%%%

\end{document}